\documentclass{amsart}
\usepackage[utf8]{inputenc}
\usepackage{amsmath,amssymb,latexsym,amsthm,mathrsfs,bbm}
\usepackage{enumitem}
\usepackage[pagebackref=true]{hyperref}
\usepackage{color}
\usepackage[dvipsnames]{xcolor}
\usepackage{cleveref}
\usepackage{appendix}
\usepackage{float}
\usepackage[margin=1in]{geometry}

\usepackage{physics}
\usepackage{xparse}
\usepackage{tikz}
\usepackage{tikz-3dplot}
\usepackage{mathdots}
\usepackage{yhmath}
\usepackage{cancel}
\usepackage{siunitx}
\usepackage{array}
\usepackage{multirow}
\usepackage{gensymb}
\usepackage{tabularx}
\usepackage{booktabs}
\usepackage{tkz-euclide}
\usetikzlibrary{arrows.meta}
\usetikzlibrary{backgrounds,calc,positioning}
\usetikzlibrary{angles,quotes}

\newtheorem{theorem}{Theorem}[section]
\newtheorem{corollary}[theorem]{Corollary}
\newtheorem{lemma}[theorem]{Lemma}
\newtheorem{proposition}[theorem]{Proposition}

\newtheorem{remark}[theorem]{Remark}

\theoremstyle{definition} 
\newtheorem{definition}[theorem]{Definition}
\numberwithin{equation}{section}

\newcommand{\R}{\mathbb{R}}
\newcommand{\Rthree}{\mathbb{R}^3}
\newcommand{\Rd}{\mathbb{R}^d}

\newcommand{\iR}{\int_{\mathbb{R}^3}}
\newcommand{\iiRs}{\iint_{\mathbb{R}^6}}

\newcommand{\Sphere}{\mathbb{S}}
\newcommand{\Stwo}{\mathbb{S}^{2}}
\newcommand{\iS}{\int_{\mathbb{S}^{2}}}

\newcommand{\bn}{\bar{\nabla}}
\newcommand{\tn}{\tilde{\nabla}}

\newcommand{\indicator}[1]{\mathbbm{1}_{#1}}
\newcommand{\Sk}{\mathbb{S}_{k^\perp}^1}
\newcommand{\itheta}[2][\pi/2]{\int_{#2 = 0}^{#1}}
\newcommand{\iSk}{\int_{\Sk}}
\newcommand{\as}[1]{\overset{\leftrightarrow}{#1}}

\title{Boltzmann to Landau from the gradient flow perspective}
\author{Jos\'e A. Carrillo}
\thanks{Mathematical Institute, University of Oxford, Oxford OX2 6GG, UK (carrillo@maths.ox.ac.uk)}
\author{Matias G. Delgadino}
\thanks{Department of Mathematics, The University of Texas at Austin, Texas 78712, USA (matias.delgadino@math.utexas.edu)}
\author{Jeremy Wu}
\thanks{Mathematical Institute, University of Oxford, Oxford OX2 6GG, UK (jeremy.wu@maths.ox.ac.uk)}
\date{\today}
\makeatletter
\define@key{z sphericalkeys}{radius}{\def\myradius{#1}}
\define@key{z sphericalkeys}{theta}{\def\mytheta{#1}}
\define@key{z sphericalkeys}{phi}{\def\myphi{#1}}
\tikzdeclarecoordinatesystem{z spherical}{
\setkeys{z sphericalkeys}{#1}%
\pgfpointxyz{\myradius*sin(\mytheta)*cos(\myphi)}{\myradius*sin(\mytheta)*sin(\myphi)}{\myradius*cos(\mytheta)}}
\tikzoption{canvas is xy plane at z}[]{%
\def\tikz@plane@origin{\pgfpointxyz{0}{0}{#1}}%
\def\tikz@plane@x{\pgfpointxyz{1}{0}{#1}}%
\def\tikz@plane@y{\pgfpointxyz{0}{1}{#1}}%
\tikz@canvas@is@plane
}
\makeatother

\begin{document}
\tdplotsetmaincoords{110}{00} 
\maketitle

\begin{abstract}
	We revisit the grazing collision limit connecting the Boltzmann equation to the Landau(-Fokker-Planck) equation from their recent reinterpretations as gradient flows. Our results are in the same spirit as the $\Gamma$-convergence of gradient flows technique introduced by Sandier and Serfaty \cite{SS04,S11}. In this setting, the grazing collision limit reduces to showing the lower semi-continuous convergence of the Boltzmann entropy-dissipation to the Landau entropy-dissipation.
\end{abstract}

\section{Introduction}
\label{sec:intro}
The Boltzmann equation is the central equation in kinetic theory modelling particle collisions in a gas, and many other interacting particle systems \cite{CIP}. The Landau equation is the most important partial differential equation in collisional kinetic theory for plasma; it describes the evolution of the density of colliding particles in plasma physics~\cite{LP81}. The Landau equation can be derived as the grazing collision limit of the Boltzmann equation, that is when collisions with small angular deviation become predominant. We seek to reformulate the well-known grazing collision relation between the non-cutoff Boltzmann and the Landau equations \cite{DL92,D92,Goud97} from the recently developed gradient flow perspectives~\cite{E19,CDDW20}, respectively. For a given collision kernel $B$, the spatially homogeneous Boltzmann equation in $\R^3$ reads
\begin{equation}
	\label{eq:boltzintro}
	\partial_t f(v) = \int_{\R^3}\iS [f'f_*' - ff_*]B\left(|v-v_*|, \theta \right) d\sigma dv_*  =: Q_B(f,f).
\end{equation}
Here, we have used the usual abbreviations and notations $f = f(v), f_* = f(v_*), f' = f(v'), f_*' = f(v_*'),$ where the post-collision velocities are given by
\begin{align*}
	v' = \frac{v+v_*}{2} + \frac{|v-v_*|}{2}\sigma, \quad 
	v_*' = \frac{v+v_*}{2} - \frac{|v-v_*|}{2}\sigma, \quad \sigma \in \Stwo.
\end{align*}
Intuitively, the second argument of $B$ is the independent variable which represents the angle of collisions $\theta \in [0, \pi/2]$ and can be implicitly recovered from the relations
\begin{align*}
	k = \frac{v-v_*}{|v-v_*|}, \quad 
	\cos \theta = k\cdot \sigma.
\end{align*}
A typical example of the Boltzmann kernel is
\[
B(|z|, \theta) = |z|^\gamma b(\theta), \quad \sin\theta\, b(\theta) =: \beta(\theta)\ge 0 \quad \gamma \in [-4,0].
\]
Here, $\gamma>0$ is referred as the hard potential case and $\gamma<0$ as the soft potential case (moderately soft for $\gamma\in [-2,0)$ and very soft for $\gamma\in [-4,-2)$). We highlight $\gamma = 0$ as the Maxwellian molecule case and $\gamma = -3$ as the physically relevant Coulomb case. 
One can formally derive the Landau equation in the case when the bulk of the collisions happen with a small angle $\theta \ll 1$. More specifically, fix $\epsilon>0$ and extend $\beta$ from $[0,\pi/2]$ to the whole real line by zero. We consider the scaling (discussed in~\Cref{sec:kernel}) that concentrates $\beta$ around $\theta =0$ given by
\[
\beta^\epsilon(\theta) := \frac{\pi^3}{\epsilon^3}\beta\left(
\frac{\pi \theta}{\epsilon}
\right), \quad \theta \in [0,\epsilon/2].
\]
Denoting the new collision kernel $B^\epsilon$ that is induced through this scaling, this gives rise to a new collision operator $Q_B^\epsilon$ which replaces the right-hand side of~\eqref{eq:boltzintro}. Taking $\epsilon\to 0$ is known as the \textbf{grazing collision limit}. More precisely, for a fixed sufficiently smooth $f$, the formal computations of Degond and Lucquin-Desreux~\cite{DL92} and Desvillettes~\cite{D92} show the convergence
\[
Q_B^\epsilon(f,f) \overset{\epsilon\downarrow 0}{\to} Q_L(f,f),
\]
where the Landau collision operator $Q_L(f,f)$ is given by
\begin{align*}
	Q_L(f,f) = C_\beta\nabla_v\cdot \left\{f
	\int_{\R^3}f_* |v-v_*|^{2+\gamma} \Pi[v-v_*](\nabla_v \log f - \nabla_{v_*} \log f_*) dv_*
	\right\}.
\end{align*}
Here, the constant is $C_\beta = \frac{\pi}{8}\int_0^{\pi/2}\theta^2 \beta(\theta)d\theta$ and the matrix $\Pi[v-v_*]$ is the projection onto $\{v-v_*\}^\perp$,
\[
\Pi[v-v_*] = I - \frac{(v-v_*)\otimes (v-v_*)}{|v-v_*|^2}.
\]
For simplicity, we shall hereafter assume $C_\beta = 1$ which fixes a normalization for $\beta$. Originally, Landau~\cite{LP81} derived what is now known as the Landau equation (also known as the Landau-Fokker-Planck equation)
\[
\partial_t f = Q_L(f,f),
\]
as a model to replace the Boltzmann equation for grazing collisions, sidestepping the singularities arising from $\beta$ around $\theta\sim 0$. Of course, while these preliminary computations established the formal `convergence of the collision operators' the natural question is rigorous convergence of solutions
\[
\partial_t f^\epsilon = Q_B^\epsilon(f^\epsilon,f^\epsilon) \to \partial_t f = Q_L(f,f)
\]
which is the main topic of this paper. This question has already been answered~\cite{H14,V98} including quantitative estimates by Godinho~\cite{G13} for short times, in which he considers solutions within the well-posedness framework theory of Fournier-Mouhot~\cite{FM09} and Fournier-Gu\'erin~\cite{FG08}. We seek to revisit this limiting process from the perspective of gradient flows as mentioned earlier; our contribution is to streamline the proof of the grazing collision limit under the mildest assumptions of uniformly bounded initial second moment and entropy, for all values of $\gamma$. As a contrast, Godinho~\cite{G13} requires the initial conditions to have at least 7th order moments, or more depending on $\gamma$, for the well-posedness theory to apply. In this article, we build on Villani's identification of the entropy-dissipation structure~\cite{V98} by following the program set by Sandier-Serfaty~\cite{SS04,S11}. Although this procedure is by now well-known, our contributions include a new inequality relating the Boltzmann entropy-dissipation to the Landau entropy-dissipation as well as a method to prove the detailed steps from Sandier-Serfaty. We illustrate the parallel views of H-solutions and our notion of solutions in the diagram below. Horizontal arrows denote the passage of the grazing collision limit, while vertical arrows denote the recent equivalent views between H-solutions~\cite{V98} of Boltzmann or Landau and gradient flow solutions. References are attached to the arrows corresponding to the respective contributions. In this paper, we consider what we denote `H-gradient flows' as our notion of solution, the precise definition can be found in~\Cref{sec:assumptions}.
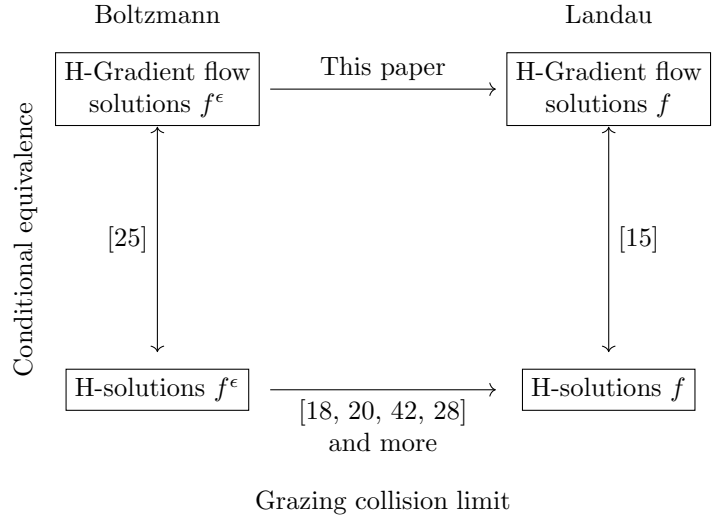
\begin{figure}[H]
	\centering
	\begin{tikzpicture}
		\node[draw,align = center] at (0,4) {H-Gradient flow\\solutions $f^\epsilon$};
		\node[draw,align = center]at (6,4) {H-Gradient flow\\solutions $f$};
		\node[draw, align = center]at (0,0) {H-solutions $f^\epsilon$};
		\node[draw,align = center] at (6,0) {H-solutions $f$};
		\node at (0,5) {Boltzmann};
		\node at (6,5) {Landau};
		\draw[<->] (0,0.5) -- (0,3.5) node[pos=0.5, left] {\cite{E19}};
		\draw[<->] (6,0.5) -- (6,3.5) node[pos=0.5, right] {\cite{CDDW20}};
		\draw[->] (1.5,0) -- (4.5,0) node[pos=0.5, below, align = center] {\cite{DL92,D92,V98,G13} \\ and more};
		\draw[->] (1.5,4) -- (4.5,4) node[pos=0.5, above] {This paper};
		\node at (3,-1.5) {Grazing collision limit};
		\node at (-1.75,2) {\rotatebox{90}{Conditional equivalence}};
	\end{tikzpicture}
\end{figure}

\paragraph{\textbf{Previous Results on the Grazing Collision Limit.-}}
We now briefly discuss some of the results concerning different notions of solution for the Boltzmann equation in relation to the grazing collision limit. The earliest well-posedness result for usual \textit{weak solutions} to the Boltzmann equation is due to Arkeryd~\cite{A72,A722} who required cut-off assumptions on the collision kernel. In particular, this excludes the physically relevant soft potential cases $\gamma < 0$. Nevertheless, this well-posedness theory was sufficient for Arsen'ev and Buryak~\cite{AB90} in 1990 who rigorously proved convergence in the grazing collision limit from Boltzmann to Landau.

In an important breakthrough, Villani introduced the notion of \textit{H-solutions}~\cite{V98} which treated the grazing collision limit for soft potentials $\gamma \in [-3,0]$ and hard potentials $\gamma >0$. Shortly after, in a collaboration with Alexandre~\cite{AV02,AV04} they upgraded from weak to strong convergence in the grazing collision limit, by applying the regularity estimate they achieved with Desvillettes and Wennberg~\cite{ADVW00}. The argument for the gain in compactness relied on velocity average techniques~\cite{GLPS88} applied to \textit{renormalized solutions}~\cite{DL89r,DL89}.

Later on, quantitative rates of convergence in the grazing collision limit were established by Godinho~\cite{G13}. His results relied on the uniqueness theorems (still open for Landau in the Coulomb case $\gamma=-3$) of Fournier who collaborated with Gu\'erin~\cite{FG08} and Mouhot~\cite{FM09} by treating these equations as the Kolmogorov-Fokker-Planck equations associated to certain \textit{stochastic processes}.

More recently, Erbar~\cite{E19} characterized weak solutions of the Boltzmann equation as \textit{gradient flows} of the entropy in the Maxwellian case $\gamma=0$. The current authors and Desvillettes~\cite{CDDW20} proved an analogous characterization for the Landau equation in the soft potential case not including Coulombic interaction $\gamma>-3$. We expect Erbar's result for the Boltzmann equation can be extended to soft potentials (at least down to $\gamma>-3$ as in~\cite{CDDW20}) although this is left for future work. The primary technical issue in Erbar's result which does not extend for $\gamma< 0$ is the unavailability of a lower bound for the Boltzmann entropy-dissipation in the spirit of \cite{D15,D16} for the Landau equation. These estimates allowed the current authors and Desvillettes to treat $\gamma>-3$ for the Landau equation; a similar estimate is expected for the Boltzmann entropy-dissipation for $\gamma>-3$. Moreover, for both Boltzmann and Landau, we believe this gradient flow correspondence can be rigorously proven for the full range of $\gamma$ considered in this paper, for which we would require a significant improvement of the technical estimates in~\cite{E19,CDDW20}.

While Erbar's characterization was not proven for $\gamma<0$, in this manuscript we will focus on solutions that dissipate entropy. This mechanism can be captured by the renormalized solutions of Alexandre and Villani~\cite{AV02}, in which they developed an existence theory for a large range of kernels $\gamma\in[-3,0]$, and very mild assumptions on the initial data; finite second moment and entropy. In fact, this is the setting which we consider in this paper, and we introduce a slightly stronger notion of solution than H-solutions from~\cite{V98} (see~\Cref{def:hcms}) which, however, is weaker than renormalized solutions from~\cite{AV02}.



\smallskip

\paragraph{\textbf{Entropy dissipation gradient flow structure.-}}
Over the last two decades, the gradient flow community has been very active in PDEs starting from  
the significant gradient flow landmarks by Jordan, Kinderleher and Otto~\cite{JKO98}, Benamou and Brenier~\cite{BB00}, Otto~\cite{O01} and the seminal reference book by Ambrosio, Gigli and Savare~\cite{AGS08}. Some of the advantages of gradient flow techniques include new insights into new functional inequalities, stable numerical methods and quantitative understanding of trends to equilibrium in tandem with uniqueness of solutions.

In \cite{SS04,S11} Sandier and Serfaty utilized the Energy Dissipation Inequality as a way to streamline the characterization the limit of evolutions that have a gradient flow structure. Effectively the problem reduces to checking the lower semi-continuous convergence of the associated dissipations and metric derivative. This approach has been heavily used in the recent years in a wide array of scenarios. Making a non-exhaustive list we mention the works in Cahn-Hilliard \cite{choksi2011small, bellettini2012convergence, delgadino2018convergence}, diffusion to reaction limits \cite{arnrich2012passing}, particle methods second order \cite{carrillo2019blob} and fourth order \cite{matthes2017convergent} non-linear diffusion, congested crowd motion \cite{alexander2014quasi} and dislocations \cite{blass2015dynamics}. In this manuscript, we follow the strategy of Sandier and Serfaty to give a straightforward and self-contained proof of the grazing collision limit.

Formally, the gradient flow structure of a PDE is given by understanding the evolution of a Lyapunov functional as a steepest descent in some specific metric. We can be more explicit (but still formal) with the Boltzmann and Landau equations as examples. The famous H-theorem asserts that the Boltzmann entropy
\[
\mathcal{H}[f] = \int f \log f
\]
is a Lyapunov functional for both of these equations. More specifically, consider $f^\epsilon(t), f(t)$ solutions to the Boltzmann and Landau equations, respectively; formally calculating the evolution of the entropy we obtain
\begin{align}
	\label{eq:Hthm}
	\begin{split}
		\mathcal{H}[f^\epsilon(t)]    + \int_0^t\underbrace{\frac{1}{4}\iiRs \iS [{f^\epsilon}'{f^\epsilon}_*' - f^\epsilon f_*^\epsilon](\log{f^\epsilon}'{f^\epsilon}_*' - \log f^\epsilon f_*^\epsilon) B^\epsilon d\sigma dv_* dv}_{=:D_B^\epsilon(f^\epsilon)\ge 0}ds &= \mathcal{H}[f^\epsilon(0)] ,  \\
		\mathcal{H}[f(t)]     
		+ \int_0^t\underbrace{2\iiRs |v-v_*|^{2+\gamma}\left|\Pi[v-v_*](\nabla-\nabla_*)\sqrt{ff_*}\right|^2dv_* dv}_{=:D_L(f)\ge 0}ds &= 
		\mathcal{H}[f(0)] .
	\end{split}
\end{align}
We will refer to the underbraced terms $D_B^\epsilon$ and $D_L$ as the Boltzmann and Landau dissipations, respectively. \Cref{eq:Hthm} implies not only that the Boltzmann entropy is a Lyapunov functional along the flows of the Boltzmann and Landau equations, but also it formally quantifies its descent. Beyond the physical relevance of~\eqref{eq:Hthm} as evidence of the arrow of time, mathematically this mechanism is at the core of Villani's infinite time horizon existence theory of solutions to both the Boltzmann and Landau equations, see \cite{V98}. Roughly speaking, if $\mathcal{H}[f(0)]<+\infty$ in the case of Landau, then the second equation of~\eqref{eq:Hthm} implies that $\mathcal{H}[f(t)]$ is a decreasing function of time with dissipation given by $\int_0^tD_L(f)dt <+\infty$. Villani recognised the finiteness of $\int_0^tD_L(f)dt$ as a functional regularity statement on $f$. We also point out that this was the focus of Desvillettes' results~\cite{D15,D16}; finite Landau entropy-dissipation implies finiteness of some weighted Fisher information functional. This notion of solution is known as H-solutions, and one of its salient features is that it only assumes boundedness of relevant physical quantities of the initial data. This perspective was taken further first by Erbar~\cite{E19} for Boltzmann and then the current authors and Desvillettes~\cite{CDDW20} for Landau by considering~\eqref{eq:Hthm} as a \textit{steepest descent} formulation of entropy with a specific `metric' associated to the dissipation. In these works, the metrics are constructed to rewrite~\eqref{eq:Hthm} as a so-called `Energy Dissipation (In)equality' (EDI or EDE)
\begin{align}
	\label{eq:EDE}
	\begin{split}
		\mathcal{H}[f^\epsilon(t)] + \frac{1}{2}\int_0^t D_B^\epsilon(f^\epsilon(s))\;ds +\frac{1}{2}\int_0^t |\dot{f}^\epsilon|_\epsilon^2(s)\; ds &\overset{(\le)}{=}
		\mathcal{H}[f^\epsilon(0)] ,     \\
		\mathcal{H}[f(t)] + \frac{1}{2}\int_0^t D_L(f(s))\;ds + \frac{1}{2}\int_0^t|\dot{f}|_L^2(s)\; ds &\overset{(\le)}{=} 
		\mathcal{H}[f(0)] .
	\end{split}
\end{align}
The quantities $|\dot{f}^\epsilon|_\epsilon^2$ and $|\dot{f}|_L^2$ are the metric derivatives with respect to the Boltzmann metric~\cite{E19} $d_\epsilon$, and Landau metric~\cite{CDDW20} $d_L$. The main contribution of \cite{E19} and \cite{CDDW20} is to show that, under assumptions on the collision kernel, an absolutely continuous curve $f:[0,\infty)\to L^1(\R^3)$ with bounded dissipation is a weak solution to Boltzmann or Landau if and only if it satisfies the respective EDI. For this paper, we consider curves $f^\epsilon(t)$ that satisfy the first EDI of~\eqref{eq:EDE} as `H-gradient flow' solutions to the Boltzmann equation and similarly for Landau. Our goal is to understand the grazing collision limit $\epsilon\downarrow 0$ by passing to the limit in the EDI characterization ~\eqref{eq:EDE}.


\smallskip

\paragraph{\textbf{Plan of the paper.-}}
The plan of the paper is the following. In \Cref{sec:assumptions}, we list the assumptions that we need to state our main result in~\Cref{thm:grazcoll} and describe the main steps of the proof. \Cref{sec:sigmarep} contains the notations we will use to set up the framework of the grazing collision limit. We recall very formally the grazing collision limit in~\Cref{sec:formalgc} which features many similar computations to be repeated later in the paper. The remaining notations and definitions pertaining to abstract gradient flow theory are recalled in~\Cref{sec:gradflow} with an emphasis on the abstract theory developed in~\cite{E19,CDDW20}. We start combining the gradient flow theory with the grazing collision limit in~\Cref{sec:cpct} which elaborates the compactness mechanisms we use to produce candidate limits for the Landau equation. The next two~\Cref{sec:gammadiss,sec:gammamd} contain the technical proofs in the passage of the limit $\epsilon\downarrow 0$. Finally, various results needed that were already present in the literature are recalled in the appendices~\ref{sec:lm} and~\ref{sec:strcpct}.


\section{Main Result}
\label{sec:assumptions}
Motivated by the EDI~\eqref{eq:EDE}, we formalize the notion of solutions to Boltzmann and Landau we consider here. In the following definition, we refer to the Boltzmann and Landau metric derivatives. These definitions as well as other technical gradient flow concepts are recalled later in~\Cref{sec:gradflow} for the sake of presentation.
\begin{definition}[H-gradient flows for Boltzmann and Landau]
	\label{def:hcms}
	For $\epsilon>0$ and $T>0$, we say that an absolutely continuous curve $f^\epsilon : t \in [0,T] \mapsto L^1_{\ge 0}(\R^3)$ with respect to the Boltzmann metric is an \textit{H-gradient flow} solution to the Boltzmann equation (with kernel $B^\epsilon$) if the Energy Dissipation Inequality holds for every $t \in [0,T]$
	\[
	\mathcal{H}[f^\epsilon(t)] + \frac{1}{2}\int_0^t D_B^\epsilon(f^\epsilon(s)) \;ds+ \frac{1}{2}\int_0^t|\dot{f}^\epsilon|_\epsilon^2(s) \;ds\leq \mathcal{H}[f^\epsilon(0)] <\infty ,
	\]
	and it preserves mass and dissipates the second moment
	\begin{equation}\label{eq:aux}
		\int_{\R^3}f_t^\epsilon(v) dv = 1, \quad \int_{\R^3}|v|^2 f_t^\epsilon(v)dv \le \int_{\R^3}|v|^2 f_s^\epsilon(v)dv<\infty, \quad \forall \, 0 \le s\le t\le T.
	\end{equation}
	Likewise, for $T>0$ an absolutely continuous curve $f: t \in [0,T] \mapsto  L^1_{\ge 0}(\R^3)$ with respect to the Landau metric is an \textit{H-gradient flow} solution to the Landau equation if the Energy Dissipation Inequality holds for every $t \in [0,T]$
	\[
	\mathcal{H}[f(t)] + \frac{1}{2}\int_0^t D_L(f(s))\;ds + \frac{1}{2}\int_0^t|\dot{f}|_L^2(s)\;ds\leq  \mathcal{H}[f(0)] <\infty ,
	\]
	and it preserves mass and dissipates the second moment as in \eqref{eq:aux} replacing $f_t^\epsilon$ by $f_t$.
\end{definition}
\begin{remark}
	\label{rem:sug}
	The notion of H-gradient flow is strictly weaker than the notion of curves of maximal slope introduced in \cite{AGS08}. More specifically, we do not require that the dissipations $D_B^\epsilon$ and $D_L$ to be strong upper gradients (see~\Cref{sec:gradflow}). This extra property was shown by Erbar~\cite{E19} for the Boltzmann equation in the Maxwellian potential case. The current authors and Desvillettes~\cite{CDDW20} showed this property for the Landau equation in the soft potential case under the following additional integrability assumptions for $\gamma\in (-3,0]$. Together with bounded entropy and the time integrability of the dissipation, the additional criteria for $f$ is that there exists some $p \in ( 3/(3+\gamma),\infty]$ such that
	\[
	(1+|v|^2)^{1-\frac{\gamma}{2}}f \in L^\infty((0,T);L^1\cap L^p(\R^3)).
	\]
\end{remark}
\begin{remark}
	\label{rem:renormalH}
	The more classical notion of renormalized solutions are weak solutions that also dissipate entropy
	$$
	\mathcal{H}(f(t))+\int_0^t D(f(t))\le \mathcal{H}(f_0).
	$$
	It can be checked that renormalized solutions are also H-gradient flow solutions, see~\Cref{rem:mdeqdiss}. The existence of renormalized solutions (and hence H-gradient flow solutions) can be shown subject to boundedness assumptions on the initial data, see \cite[Corollary 2.1 and Appendix]{AV02}.
\end{remark}
For H-gradient flows, the initial entropy controls the entropy at later times as well as the integrability of the dissipation and the metric derivative. We therefore consider H-gradient flows of Boltzmann, $f^\epsilon$, subject to the following assumptions.
\begin{enumerate}[label=(\textbf{A\arabic*})] 
	\item \label{ass:bddmom} For every $\epsilon>0$, we assume that the initial probability densities $f^\epsilon(0) = f_0^\epsilon$ converge in the weak-* topology to some probability density $f_0$. Furthermore, we assume a uniform second moment bound and convergence in entropy
	\[
	\sup_{\epsilon>0}\int_{\R^3} (1+|v|^2)f_0^\epsilon(v)dv < +\infty, \quad \mathcal{H}[f_0^\epsilon] \overset{\epsilon\downarrow 0}{\to} \mathcal{H}[f_0] < + \infty.
	\] 
	\item \label{ass:beta} There exists $\gamma\in[-4,0)$, such that the $\epsilon$-collision kernel satisfies 
	$$
	B^\epsilon(r,\theta) \sin\theta = r^\gamma \beta^\epsilon(\theta), 
	$$
	where
	\[
	\beta^\epsilon(\theta) = \frac{\pi^3}{\epsilon^3}\beta\left(
	\frac{\pi\theta}{\epsilon}
	\right), \quad \theta\in(0,\epsilon/2).
	\]
	The function $\beta$ satisfies that for every $\delta>0$
	\[
	\sup_{\theta\in[\delta,\pi/2]}\beta(\theta) < + \infty, \quad \text{supp}\; \beta \in [0,\pi/2].
	\]
	and that there exists $\nu \in (0,2)$ and $c_1>0$ such that
	$$
	c_1 \theta^{-1-\nu} \le \beta(\theta), \quad \forall \theta \in [0,\pi/2].
	$$
	The most important quantitative assumption on the kernel is \textit{finite angular momentum transfer}~\cite{V98}
	\begin{equation}
		\label{eq:betaint}
		\int_0^{\pi/2}\theta^2 \beta(\theta)\; d\theta =\frac{8}{\pi}.
	\end{equation}
\end{enumerate}
\begin{remark}
	The choice of $\frac{8}{\pi}$ in~\eqref{eq:betaint} is a normalization constant that fixes $C_\beta$ = 1 as in~\Cref{sec:intro}.
\end{remark}

\begin{remark}
	\label{rem:uppbdd}
	Our results also readily generalize to more general interaction kernels $B^\epsilon$ which do not decouple or satisfy the specific scaling of \cref{ass:beta}. As in~\cite{ADVW00,AV02}, we consider kernels that satisfy the following bound on the total cross section
	\begin{equation}\label{Tepsbound}
		T^\epsilon(|v-v_*|) := \int_0^{\pi/2}\theta^2 \sin\theta B^\epsilon(|v-v_*|,\theta) d\theta\le C (|v-v_*|^{-4} + 1)\omega(|v-v_*|^2).
	\end{equation}
	where $C>0$ is a constant and $\omega$ is a bounded positive function such that $\omega(r) \to 0$ as $r\to \infty$ and $r\to 0$. 
	
	As for the grazing collision limit $\epsilon \downarrow 0$, we require that there exists a function $T$ such that
	\begin{equation}\label{Tepsconvergence}
		|T^\epsilon(r) - T(r)| \le o(1)\left(
		r^{-4} + 1\right)\omega(r), \quad \text{as }\epsilon\downarrow 0.    
	\end{equation}
	Up to a multiplicative constant, the limiting Landau collision operator reads
	\[
	Q_L(f,f) = \nabla\cdot \left(f \int_{\R^3}f_* T(|v-v_*|)|v-v_*|^2\Pi[v-v_*](\nabla\log f - \nabla_* \log f_*) dv_*\right).
	\]
	We will discuss this generalization in more detail in~\Cref{rem:totalcrosssection}.
\end{remark}

\begin{remark}\label{rem:simplification}
	The Coulombic collision kernel which couples $\gamma =-3$ with $\nu = 2$ (see the discussion in~\Cref{sec:kernel}) fails the finite angular momentum transfer~\eqref{eq:betaint}. In this case, a minor logarithmic cut-off adjustment is needed; for example~\cite{V98} we consider instead
	\[
	\beta^\epsilon(\theta) := \frac{1}{\log \epsilon^{-1}}1_{\theta \ge \epsilon}\beta(\theta).
	\]
	Under this new scaling, we require
	\[
	\int_0^{\pi/2}\theta^2 \beta^\epsilon(\theta)d\theta \overset{\epsilon\downarrow 0}{\to} \frac{8}{\pi}.
	\]
	Our methods can be adapted to cover this case as well. The strong compactness estimate as it is written~\Cref{sec:strcpct} technically fails, but since the cut-off disappears in the grazing collision limit, the necessary strong compactness is still valid~\cite{AV02,AV04}.
\end{remark}

\begin{theorem}[Grazing collision limit of H-gradient flow solutions from Boltzmann to Landau]
	\label{thm:grazcoll}
	Suppose $(f^\epsilon)_{\epsilon>0}$ is a family of H-gradient flow solutions to the Boltzmann equation satisfying assumptions~\ref{ass:bddmom} and~\ref{ass:beta}. Then there exists $f$ an H-gradient flow solution to the Landau equation, such that up to a subsequence $f^\epsilon(t)$ converges in the weak-* topology against bounded and continuous functions to $f(t)$ for every $t\in[0,T]$. 
	
	Moreover, for $\gamma\in[-2,0]$ and fixed $\phi \in \dot{W}^{1,\infty}(\R^3)$ (resp. $\gamma\in[-4,-2)$ and fixed $\phi \in \dot{W}^{2,\infty}(\R^3)$), the function $t\mapsto \int f(t)\phi$ is H\"older continuous with exponent $\frac{1}{2}.$
\end{theorem}

\begin{proof}
	Our starting point is the EDI from the definition of H-gradient flow solutions to the Boltzmann equation. By definition and the finite initial quantities in~\ref{ass:bddmom}, we have the uniform bounds
	\begin{equation}
		\label{eq:bounds}
		\max\left(\sup_{\epsilon>0}\sup_{t\in[0,T]} \int_{\R^3} (1+|v|^2)f_t^\epsilon(v) dv,\ \sup_{\epsilon>0}\sup_{t\in[0,T]}\mathcal{H}[f^\epsilon(t)],\ \sup_{\epsilon>0} \int_0^T D_B^\epsilon(f_t^\epsilon) dt,\ \sup_{\epsilon>0} \int_0^T |\dot{f}^\epsilon|_\epsilon^2(t) dt\right) < + \infty.
	\end{equation}
	These uniform bounds and assumption~\ref{ass:beta} are used in the following steps.
	\begin{enumerate}
		\item Extract a convergent subsequence of $f^\epsilon$ and some limit $f$ with the claimed time regularity (\Cref{sec:cpct}).
		\item Establish the estimate (\Cref{sec:gammadiss})
		\[
		\liminf_{\epsilon\downarrow 0}D_B^\epsilon(f^\epsilon(t)) \ge D_L(f(t))\qquad\mbox{a.e. $t\in(0,T)$}.
		\] \label{step:gammadiss}
		\vspace{-0.2cm}
		\item Establish the estimate (\Cref{sec:gammamd})
		\[
		\liminf_{\epsilon\downarrow 0}|\dot{f}^\epsilon(t)|_\epsilon^2 \ge |\dot{f}(t)|_L^2 \qquad\mbox{a.e. $t\in(0,T)$}.
		\] \label{step:gammamd}
	\end{enumerate}
	Next, we pass to the limit $\epsilon\downarrow 0$ in the EDI using \ref{ass:bddmom}, Fatou's Lemma, Step~\ref{step:gammadiss}, Step~\ref{step:gammamd}, and the lower semi-continuity of $\mathcal{H}$ to obtain
	\begin{align*}
		\mathcal{H}[f_0]=\liminf_{\epsilon\downarrow 0}\mathcal{H}[f^\epsilon_0] &\ge \liminf_{\epsilon\downarrow 0}\mathcal{H}[f^\epsilon(t)] + \frac{1}{2}\liminf_{\epsilon\downarrow 0}\int_0^t D_B^\epsilon(f^\epsilon(s))ds + \frac{1}{2} \liminf_{\epsilon\downarrow 0} \int_0^t |\dot{f}^\epsilon|_\epsilon^2(s)ds\\
		&\ge\mathcal{H}[f(t)] + \frac{1}{2}\int_0^t D_L(f(s))ds + \frac{1}{2} \int_0^t |\dot{f}|_L^2(s)ds,
	\end{align*}
	which implies that $f$ is an H-gradient flow solution to Landau.
\end{proof}
	\begin{remark}
		According to the same second moment and entropy bounds from assumptions~\ref{ass:bddmom} and~\ref{ass:beta}, we recover the results in Villani~\cite{V98}, that is the convergence from $f^\epsilon$ to $f$ weakly in $L^p((0,T); \, L^1(\R^3))$ for every $1 \le p < +\infty$. As for the time regularity, we are able to expand the class of test functions to bounded Lipschitz functions for $\gamma\in[-2,0]$.
	\end{remark}

\begin{remark}[Affine Representation]
	The main idea to showing Step~\ref{step:gammadiss} (\Cref{sec:gammadiss}) and Step~\ref{step:gammamd} (\Cref{sec:gammamd}) of the previous proof is to rewrite these expressions in what we will hereafter refer to as \textit{affine representation}. In the context of optimal transport gradient flows, this method was first utilized by Otto~\cite[equation (187)]{O01} for the Fisher information. More explicitly, we have the characterization of the Fisher information as
	\[
	\int_{\Rd} |\nabla \sqrt{f}|^2 = \sup_{\psi \in C_c^\infty(\Rd)} \left\{
	2\int_{\Rd} \nabla \sqrt{f} \cdot \nabla \psi - \int_{\Rd} |\nabla \psi|^2
	\right\}.
	\]
	The left-hand side (quadratic in $\nabla \sqrt{f}$) is equal to a supremum over particular affine expressions of $\nabla \sqrt{f}$ on the right-hand side. We establish a similar equality for both the dissipations and the metric derivatives. Taking the Landau dissipation for example and denoting the differential operator $\tn = |v-v_*|^{1+\frac{\gamma}{2}}\Pi[v-v_*](\nabla - \nabla_*)$, we show (\Cref{sec:landiss})
	\[
	D_L(f) = 2\iiRs |\tn \sqrt{ff_*}|^2 = \sup_{\psi}\left\{
	4\iiRs (\tn \sqrt{ff_*})\cdot \tn \psi - 2 \iiRs |\tn \psi|^2
	\right\}.
	\]
	The specific set of test functions $\psi$ for which the supremum is taken will be specified in later sections. We note that we have dropped the differentials in the integrals. To avoid burdensome notation we will do this throughout the paper when the variables of integration are clear.
\end{remark}
When the limit $f$ has enough integrability (see~\Cref{rem:sug}), we can apply the results of \cite{CDDW20} to obtain that the Landau dissipation is a strong upper gradient, see \Cref{defn:sug}. Following the gradient flow $\Gamma$-convergence arguments of Sandier-Serfaty~\cite{SS04,S11}, we can readily show that the solution $f^\epsilon$ converges strongly to $f$. This is the content of our next result.
\begin{corollary}
	\label{cor:SSGamma}
	Suppose $(f^\epsilon)_{\epsilon>0}$ is a family of H-gradient flows (with equality in the EDE~\eqref{eq:EDE}) of the Boltzmann equation in $t\in[0,T]$ satisfying assumptions~\ref{ass:bddmom} and~\ref{ass:beta}. Assume further that $D_L$ is a strong upper gradient for the limit curve $f$ obtained in~\Cref{thm:grazcoll}. Then for every $t\in [0,T]$, we have
	\[
	\mathcal{H}[f^\epsilon(t)] \overset{\epsilon\downarrow 0}{\to} \mathcal{H}[f(t)].
	\]
	Moreover, we also obtain
	\[
	D_B^\epsilon(f^\epsilon) \to D_L(f), \quad |\dot{f}^\epsilon|^2_\epsilon \to |\dot{f}|^2_L, \quad \text{in }L_{loc.}^1(0,T).
	\]
\end{corollary}
\begin{proof}
	This proof follows the gradient flow $\Gamma$-convergence arguments of Sandier-Serfaty~\cite{SS04,S11}. We fix $t\in[0,T]$. Repeating the passage to the limit $\epsilon\downarrow 0$ from the proof of~\Cref{thm:grazcoll}, we have
	\[
	\liminf_{\epsilon\downarrow 0}(-\mathcal{H}[f^\epsilon(t)]) \ge -\mathcal{H}[f_0] + \frac{1}{2}\int_0^t D_L(f(s)) + |\dot{f}|_L^2(s) ds.
	\]
	By Young's inequality and the assumption that $D_L$ is a strong upper gradient for $f$, we have
	\[
	\frac{1}{2}\int_0^t D_L(f(s)) + |\dot{f}|_L^2(s) ds \ge -\int_0^t\sqrt{D_L(f(s))}|\dot{f}|_L(s) ds \ge \mathcal{H}[f_0] - \mathcal{H}[f(t)].
	\]
	These previous inequalities yield
	\[
	\liminf_{\epsilon\downarrow 0}-\mathcal{H}[f^\epsilon(t)] \ge -\mathcal{H}[f(t)],
	\]
	which, together with the lower semi-continuity of $\mathcal{H}$, gives the strong convergence
	$
	\mathcal{H}[f^\epsilon(t)] \overset{\epsilon\downarrow 0}{\to} \mathcal{H}[f(t)].
	$
	This forces all of the inequalities to be equalities and the rest of the proof proceeds exactly the same as in~\cite{SS04,S11}.
\end{proof}

\section{Notations and formulation of grazing collision limit}
\label{sec:sigmarep}
For numbers $a, \,b\in\R$ we use the symbol $a \lesssim_{\alpha, \,\beta, \,\dots} b$ to mean $
a \le C b$ for some constant $C = C(\alpha, \,\beta, \,\dots) >0.$ In the case where the dependence of the constant is explicit, we will drop the subscript on `$\lesssim$'. We may also write $a = \mathcal{O}(b)$ to mean $a\lesssim b$. We will use $a \sim b$ to denote both $a\lesssim b$ and $b \lesssim a$. For ease of notation, we will drop the differentials of integration when the variables we are integrating on are evident.

We define the Boltzmann collision operator for a fixed collision kernel $B$ acting on test functions $\psi$ by
\begin{equation}
	\label{eq:sigmaact}
	\langle Q_B(f,f), \psi\rangle := -\frac{1}{4}\iiRs \iS [f'f_*' - ff_*](\psi' + \psi_*' - \psi - \psi_*) B\left(
	|v-v_*|, \theta
	\right)d\sigma dv_* dv.
\end{equation}
The pre-post collision quantities are defined as follows for $v, \,v_* \in \R^3$ and $\sigma \in \Stwo$.
\begin{align*}
	k = \frac{v-v_*}{|v-v_*|}, \quad \cos \theta = k\cdot \sigma, \quad v' = \frac{v+v_*}{2} + \frac{|v-v_*|}{2}\sigma, \quad
	v_*' = \frac{v+v_*}{2} - \frac{|v-v_*|}{2}\sigma.
\end{align*}
We shall make precise in~\Cref{sec:kernel} the typical example of the kernel $B$ since the grazing collision limit comes from the concentration of small $\theta$ here. In~\Cref{sec:spherical}, we construct a coordinate system parametrizing $\sigma\in\Stwo$ by $\theta\in[0,\pi/2]$ and another variable which makes the grazing collision computations more explicit in the rest of this paper. We gather these notations and recall the formal grazing collision computations in~\Cref{sec:formalgc} as an example of many similar computations in this paper.
\subsection{Comments on the assumptions of the kernel}
\label{sec:kernel}
For $\gamma\in [-4,0), \, \nu \in (0,2)$, we recall~\ref{ass:beta} where the form of the kernel is
\begin{equation}
	\label{eq:collisionkernel}
	B(r,\theta) \sin \theta = r^\gamma \beta(\theta), \quad \beta(\theta) \gtrsim \theta^{-1-\nu}.
\end{equation}
We keep $\gamma, \,\nu$ decoupled however the physically relevant case~\cite{V98,G13} is given by
\[
-3 \le \gamma = \frac{s-5}{s-1}, \quad  \nu = \frac{2}{s-1}, \quad s\ge 2.
\] 
We define the $\epsilon$-collision kernel $B^\epsilon$ through the relation in~\eqref{eq:collisionkernel} where $\beta$ is extended to $(0,+\infty)$ by zero and we consider
\[
\beta^\epsilon(\theta) = \frac{\pi^3}{\epsilon^3}\beta\left(\frac{\pi\theta}{\epsilon} \right), \quad \theta \in (0,\epsilon/2).
\]
\begin{remark}
	The finite angular momentum transfer~\eqref{eq:betaint} is minimal for a mathematical theory~\cite{V98,AV02,ADVW00} in the non-cutoff Boltzmann equation.
	\begin{itemize}
		\item Clearly, the scaling power of $\epsilon$ localizes the singularity in $\beta^\epsilon$ around $\theta = 0$. As well, the choice of $\epsilon^{-3}$ preserves~\eqref{eq:betaint} so that
		\[
		\int_0^{\epsilon/2}\theta^2\beta^\epsilon(\theta) d\theta = \int_0^{\pi/2}\theta^2 \beta(\theta) d\theta = \frac{8}{\pi}, \quad \forall \epsilon>0.
		\]
		\item 
		The strong form of the Landau collision operator $Q_L(f,f)$ is a second-order derivative on $f$ while, at first glance, $Q_B^\epsilon(f,f)$ evaluates no derivatives of $f$. The finite angular moment transfer allows the interpretation of $Q_B^\epsilon$ as a \textbf{second-order difference quotient ``in the angular direction''}~\cite{ADVW00,AV02} on $f^\epsilon$.
	\end{itemize}
\end{remark}
We will denote quantities with sub or superscript $B$ meaning that the choice of collision kernel is arbitrary modulo~\eqref{eq:betaint}. Quantities with sub or superscript $\epsilon$ will be in specific reference to the $\epsilon$-collision kernel $B^\epsilon$ described above.
Based on these notations, we record for reference the most physically relevant parameters
\begin{equation*}
	\begin{array}{c|c|c|c|c}
		{}    &s    &\gamma     &\nu     &\beta(\theta) \overset{\theta\downarrow 0}{\sim} \theta^{-1-\nu}  \\\hline
		\text{Maxwellian}   &5  &0  &1/2    &\theta^{-3/2} \\\hline
		\text{Coulomb}     &2   &-3     &2   &\theta^{-3}
	\end{array}.
\end{equation*}
In all these examples mentioned, notice that the following lack of integrability always holds
\[
\int_0^{\pi/2}\beta(\theta)d\theta = + \infty.
\]
This can be interpreted as an `affluence of grazing collisions'~\cite{FG08,FM09}.

\subsection{Spherical coordinates}
\label{sec:spherical}
According to $\cos \theta = k\cdot \sigma$, we describe the construction of a new coordinate system for which $\sigma \in \Stwo$ can be parameterized by $(\theta, \,\phi) \in [0, \,\pi/2]\times [0, \,2\pi]$ where $\theta$ is the polar angle and $\phi$ is the azimuthal angle. In terms of the integration, this allows us to write the change of variables formula
\[
\iS d\sigma = \int_0^{\pi/2}d\theta \sin \theta \int_0^{2\pi}d\phi,
\]
so that the concentrated scaling $\theta \sim \epsilon$ is easier to treat. As we shall see, these coordinates allow to identify different mechanisms in the grazing collision limit. The average over the azimuthal angle $\phi$ induces the second-order differentiation in the orthogonal direction of $v-v_*$ seen in the Landau operator while the integration over the polar angle $\theta$ treats the singular kernal $\beta(\theta)$.

Without loss of generality, we can assume $B(|z|, \,\cdot)$ is supported in $[0, \,\pi/2]$ due to standard symmetrising. This is because any configuration of pre-post collision velocities for $\theta \in [\pi/2, \,\pi]$ corresponds to $\theta - \pi/2\in[0, \,\pi/2]$ when switching $v\leftrightarrow v_*$, see~\Cref{fig:pic1}. 
\begin{figure}[H]
	\centering
	\begin{tikzpicture}
		\draw[YellowGreen] (-90:4cm) -- (90:4cm) node[pos=0.75,left] {$P_{k,\sigma}\cap \{k\}^\perp$};
		\draw (0,0) circle (4cm);
		\draw[Circle-Circle] (180:4.06cm) -- (0:4.06cm) node[pos=0,left] {$v_*$} coordinate (v) node[pos=1,right] {$v$} node[pos=0.5,above left] {$\frac{v+v_*}{2}$};
		\draw[thick,purple,->] (0,0) -- (0:2cm) node[pos=0.5, below] {$k$};
		\draw[Circle-Circle] (210:4.06cm) -- (30:4.06cm) node[pos=0,left] {$v_*'$} coordinate (vp) node[pos=1,right] {$v'$};
		\draw[thick,blue,->] (0,0) -- (30:2cm) node[pos=0.5, above] {$\sigma$};
		\filldraw[black] (0,0) circle (1.5pt); 
		\draw[dashed, thin] (0,0) -- (15:4cm);
		\draw[->,Periwinkle] (4,0) -- (30:4) node[pos=0.6, left] {$\omega$};
		\draw[red] (0:1cm) arc (0:30:1cm) node[pos=0.4,left] {$\theta$};
		\begin{scope}[rotate around = {15:(0,0)}]
			\draw[thin,Dandelion] (3.5,0) -- (3.5,-0.36) -- (3.86,-0.36);
		\end{scope}
	\end{tikzpicture}
	\caption{Geometry of elastic collision in $\sigma$-representation.}
	\label{fig:pic1}
\end{figure}
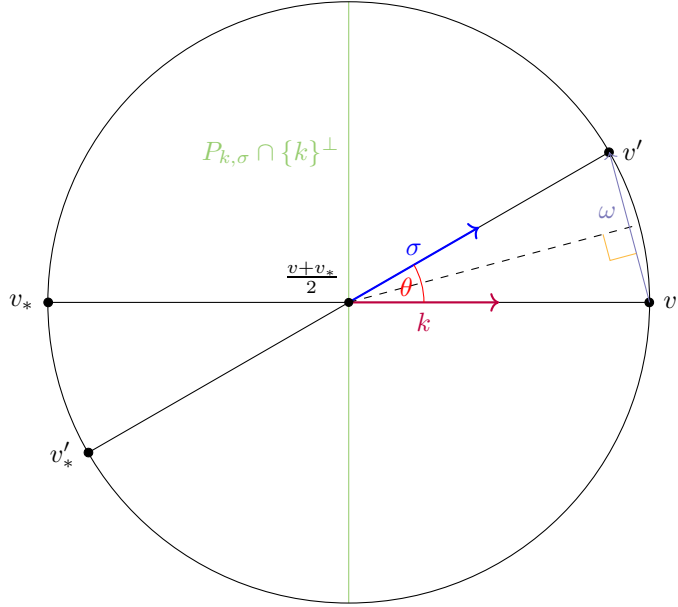
That the mid-point/momentum of the velocities is conserved is a consequence of only considering elastic collisions
\[
\frac{v+v_*}{2} = \frac{v'+v_*'}{2}.
\]
Let us refer to the plane spanned by $\sigma$ and $k$ by $P_{k, \,\sigma}$. So~\Cref{fig:pic1} gives a perspective of $P_{k, \,\sigma}$ with its normal vector coming directly out of the page. Consider the line obtained by $P_{k, \,\sigma} \cap \{ k\}^\perp$. Upon intersection with $\Stwo$ (centred at $\frac{v+v_*}{2})$, this line reduces to two vectors which differ by a sign; $\Stwo \cap P_{k, \,\sigma}\cap \{k\}^\perp = \{p_1, \,p_2\}$ where $p_1 = - p_2$ and we assign $p = p_1$ the `+' choice in the decomposition
\[
\sigma = \cos \theta \, k + \sin \theta \, p.
\]
Note this is nothing but the orthogonal decomposition of $\sigma$ with respect to $k$ and $\{k\}^\perp$. In general, we will abuse notation for this coordinate transformation by referring to $\Stwo \simeq \partial B_1(\frac{v+v_*}{2})$. We also introduce the following notation which is drawn in~\Cref{fig:phi}
\[
\Sk := \Stwo \cap \{k \}^\perp \simeq \Sphere^1.
\]
A three-dimensional perspective of~\Cref{fig:pic1} is depicted in~\Cref{fig:spherical}.
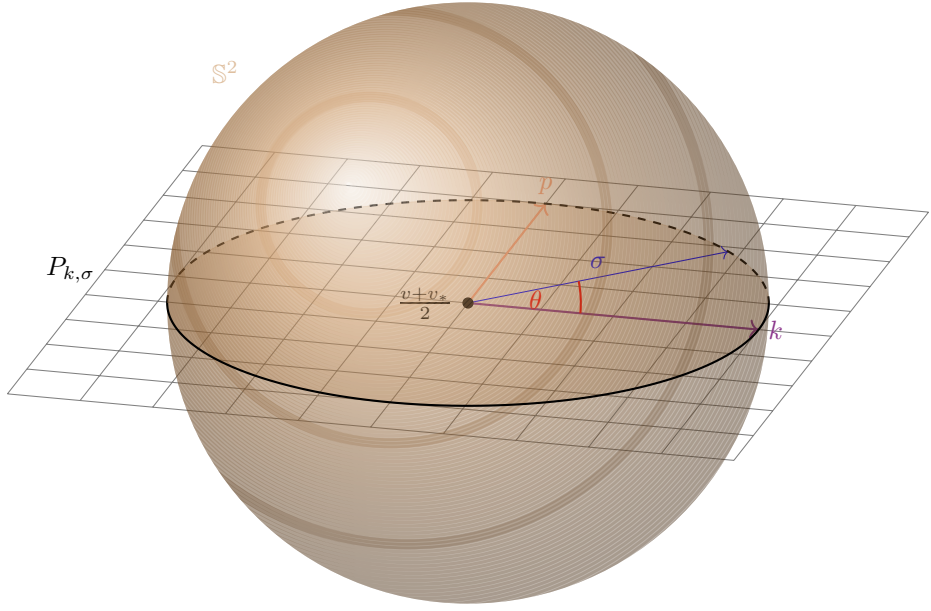
\begin{figure}[H]
	\centering
	\begin{tikzpicture}[font=\sffamily]
		\node[circle,fill,inner sep=1.5pt,label=left:$\frac{v+v_*}{2}$] (O) at (0,0,0){};
		\shade[ball color=brown,opacity=0.2] (O) circle (4);
		\shade[color=brown, opacity = 0.5] (-1,5,5) node[left] {$\Stwo$};
		\begin{scope}[tdplot_main_coords]
			\begin{scope}[on background layer]
				\begin{scope}[canvas is xy plane at z = 0,rotate around={15:(0,0)}]
					\draw[opacity = 0.5] (-5,-5) grid (5,5) ;
					\draw (-5,0) node[left] {$P_{k,\sigma}$}; 
					\filldraw[color=brown, opacity = 0.2] (0,0) circle (4);
				\end{scope}
				\draw[thick,Fuchsia,->] (0,0,0) -- (z spherical cs: radius = 4, theta = 90, phi = 15) coordinate (k) node[pos=1,right] {$k$};
				\draw[thick,Melon,->] (0,0,0) -- (z spherical cs: radius = 4, theta = 90, phi = 285) coordinate (p) node[pos=1, above] {$p$};
				\draw[blue,->] (0,0,0) -- (z spherical cs: radius = 4, theta = 90, phi = 330) coordinate (sigma) node[pos=0.5, above] {$\sigma$};
				\pic [red,draw,thick, "$\theta$", angle radius = 1.5cm] {angle = k--O--sigma };
				\draw[thick,dashed] plot[variable=\x,domain=\tdplotmainphi+180:\tdplotmainphi+360,smooth,samples=60]  (z spherical cs:radius=4,theta=90,phi=\x);
			\end{scope}
			\draw[thick] plot[variable=\x,domain=0:180,smooth,samples=60]  (z spherical cs:radius=4,theta=90,phi=\x);
		\end{scope}
	\end{tikzpicture}
	\caption{Three-dimensional perspective of $\Stwo$ with the plane spanned by $k,\sigma$ given by $P_{k,\sigma}$.}
	\label{fig:spherical}
\end{figure}
We introduce one final parameterisation for $p$, see~\Cref{fig:phi}. Consider fixed $k$ and $\theta\in [0,\pi/2]$, and distinct $\sigma_1, \,\sigma_2 \in \Stwo$ such that 
\[
\cos \theta = k\cdot \sigma_i, \quad i=1,2.
\]
Following the construction described earlier, this uniquely defines $p_i \in \Stwo \cap P_{k, \,\sigma_i} \cap \{ k\}^\perp$ such that
\[
\sigma_i = \cos \theta \, k + \sin \theta \, p_i, \quad i= 1,2.
\]
Notice that both $p_1, \, p_2 \in \Sk$. Thus, given orthonormal vectors $\{h, \,i\} \subset \Sk$, we can express
\[
p = \cos \phi \, h + \sin \phi \, i, \quad \text{for some }\phi \in [0, \,2\pi].
\]
This leads to the following change of variables; given $k\in \Stwo$ (determined by $v, \,v_*\in\Rthree$), we have
\[
\iS d\sigma = \int_0^{\pi/2}d\theta\sin\theta \int_{\Sk} dp =  \int_0^{\pi/2}d\theta\sin\theta \int_0^{2\pi}d\phi.
\]
We shall refer to both changes of variables to $(\theta,\, p)$ and $(\theta, \,\phi)$ as spherical coordinates. With these notations, we will also use the following expressions for the post-collision velocities as perturbations of the pre-collision velocities
\begin{align*}
	v' &= v + \frac{1}{2}|v-v_*|(\sigma - k) \\
	v_*' &= v_* - \frac{1}{2}|v-v_*|(\sigma - k).
\end{align*}
The post-collision velocity expressions in this form are especially useful in the grazing collision limit $|\sigma - k| \to 0$ (see~\eqref{eq:equivmom}). Using the spherical coordinate system described, we have
\begin{align}
	\sigma &= \cos \theta \, k + \sin \theta \, p, \label{eq:sigkp}\\
	v' &= v - \frac{1}{2}|v-v_*|(1-\cos \theta) k +  \frac{1}{2}|v-v_*|\sin \theta \, p, \notag\\
	v_*' &= v_* + \frac{1}{2}|v-v_*|(1-\cos \theta) k - \frac{1}{2}|v-v_*| \sin \theta \, p. \notag
\end{align}
Thus, in spherical coordinates, an equivalent form of~\eqref{eq:sigmaact} is
\begin{equation}
	\label{eq:angact}
	\langle Q(f,f),\psi\rangle = - \frac{1}{4}\iiRs \int_{\theta = 0}^{\pi/2}\iSk [f'f_*' - ff_*](\psi' + \psi_*' - \psi - \psi_*) B(|v-v_*|,\theta) \sin \theta \, dp d\theta dv_* dv.
\end{equation}

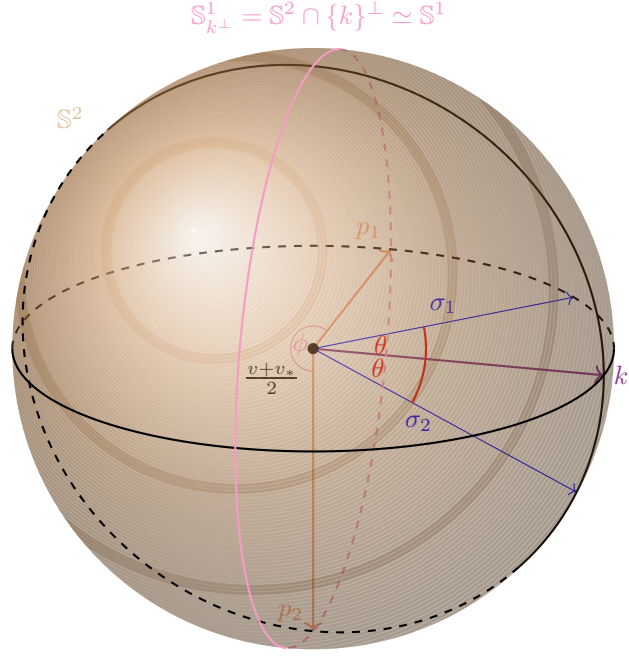
\begin{figure}[H]
	\centering
	\begin{tikzpicture}[font=\sffamily]
		\node[circle,fill,inner sep=1.5pt,label= below left:$\frac{v+v_*}{2}$] (O) at (0,0,0){};
		\shade[ball color=brown,opacity=0.2] (O) circle (4);
		\shade[color=Lavender] (2,6,5) node[above] {$\Sk = \Stwo \cap \{k\}^\perp \simeq \Sphere^1$};
		\shade[color=brown, opacity = 0.5] (-1,5,5) node[left] {$\Stwo$};
		\begin{scope}[tdplot_main_coords]
			\begin{scope}[on background layer]
				\draw[thick,Fuchsia,->] (0,0,0) -- (z spherical cs: radius = 4, theta = 90, phi = 15) coordinate (k) node[pos=1,right] {$k$};
				\draw[thick,Melon,->] (0,0,0) -- (z spherical cs: radius = 4, theta = 90, phi = 285) coordinate (p1) node[pos=1, above left] {$p_1$};
				\draw[blue,->] (0,0,0) -- (z spherical cs: radius = 4, theta = 90, phi = 330) coordinate (sigma1) node[pos=0.5, above] {$\sigma_1$};
				\pic [red,draw,thick, "$\theta$", angle radius = 1.5cm] {angle = k--O--sigma1 };
				\draw[thick,Melon,->] (0,0,0) -- (z spherical cs: radius = 4, theta = 180, phi = 15) coordinate (p2) node[pos=1, above left] {$p_2$};
				\pic [Lavender,draw,thin, "$\phi$", angle radius = 0.3cm] {angle = p1--O--p2 };
				\draw[blue,->] (0,0,0) -- (z spherical cs: radius = 4, theta = 115, phi = 15) coordinate (sigma2) node[pos=0.4, below] {$\sigma_2$};
				\pic [red,draw,thick, "$\theta$", angle radius = 1.5cm] {angle = sigma2--O--k };
				\draw[thick,dashed] plot[variable=\x,domain=\tdplotmainphi+180:\tdplotmainphi+360,smooth,samples=60]  (z spherical cs:radius=4,theta=90,phi=\x);
				\draw[thick] plot[variable=\x,domain=-45:135,smooth,samples=60]  (z spherical cs:radius=4,theta=\x,phi=15);
				\draw[thick, dashed, Lavender] plot[variable=\x,domain=0:200,smooth,samples=60]  (z spherical cs:radius=4,theta=\x,phi=285);
			\end{scope}
			\draw[thick] plot[variable=\x,domain=0:180,smooth,samples=60]  (z spherical cs:radius=4,theta=90,phi=\x);
			\draw[thick,dashed] plot[variable=\x,domain=135:315,smooth,samples=60]  (z spherical cs:radius=4,theta=\x,phi=15);
			\draw[thick, Lavender] plot[variable=\x,domain=200:380,smooth,samples=60]  (z spherical cs:radius=4,theta=\x,phi=285);
		\end{scope}
	\end{tikzpicture}
	\caption{Different configurations of $\sigma_i, \,p_i$ for $i=1, \,2$ given fixed $k, \,\theta$ such that $\cos \theta = k\cdot \sigma_i$.}
	\label{fig:phi}
\end{figure}

The general principle we want to fix is that the grazing collision limit is more easily treated in the spherical coordinate representation of~\eqref{eq:angact}. We notice the identity
\begin{equation}
	\label{eq:equivang}
	|\sigma - k |^2 = 2(1-k\cdot \sigma) = 2(1 - \cos \theta)
\end{equation}
which leads to the following useful estimates
\begin{equation}
	\label{eq:equivmom}
	\theta^2 \lesssim 1 - \cos \theta \lesssim 1 - k\cdot \sigma \lesssim |\sigma - k|^2 \lesssim \theta^2.
\end{equation}
More precisely, we recall
\begin{align}
	\label{eq:boundcos}
	\begin{split}
		\frac{2}{\pi^2}\theta^2 \le  1 - \cos \theta \le \frac{1}{2}\theta^2, \quad &\theta \in [0,\pi], \\
		\frac{2}{\pi}\theta \le  \sin \theta \le \theta, \quad &\theta \in [0,\pi/2].
	\end{split}
\end{align}

\begin{remark}
	Using~\eqref{eq:equivmom} and the spherical coordinate transformation we just described, the finite angular momentum transfer~\eqref{eq:betaint} can be equivalently expressed as
	\[
	\iS (1 - k\cdot \sigma) b(cos^{-1}(k\cdot\sigma))d\sigma < + \infty, \quad b(\theta) = \sin\theta \beta(\theta).
	\]
	This notation is sometimes used for example in~\cite{ADVW00}.
\end{remark}

\subsection{The formal grazing collision limit}
\label{sec:formalgc}
We dedicate this subsection to sketching the grazing collision limit while familiarizing the various notations that we will use throughout this paper. The sketch of the grazing collision limit is in~\Cref{lem:gcsketch}, which relies on some technical results (Lemmas~\ref{lem:estbnv2}--\ref{lem:epszerobnpsiavgv2}) which are postponed until after the sketch of the proof. Following Erbar's convention~\cite{E19}, for a function $\psi : \R^3 \to \R $ we define the discrete gradient operator
\[
\bn \psi (v,v_*,\sigma) := \psi(v') + \psi(v_*') - \psi(v) - \psi(v_*).
\]
To compress some expressions further, we extend Erbar's discrete gradient to functions $\psi : (v,v_*) \in \R^6 \mapsto \psi(v,v_*) \in \R$ by
\[
\bn \psi(v,v_*,\sigma) := \psi(v',v_*') + \psi(v_*',v') - \psi(v,v_*) - \psi(v_*,v).
\]
In this way, the action of $Q_B^\epsilon(f,f)$ on test functions $\psi$ reads either
\begin{align}
	\langle Q_B^\epsilon(f,f),\psi\rangle &= -\frac{1}{8}\iiRs \iS [\bn (ff_*)] [\bn \psi]B^\epsilon(|v-v_*|,\theta)  \,d\sigma dv_* dv, \label{eq:boltztest1}\\
	&= \frac{1}{2}\iiRs ff_* \iS \bn \psi B^\epsilon \, d\sigma dv_* dv, \label{eq:boltztest2}
\end{align}
where the second expression stems from changing variables in the pre-post collision velocities from~\eqref{eq:sigmaact}.

We define the Landau collision operator $Q_L(f,f)$ acting on test functions $\psi$ by
\begin{align}
	\label{eq:lanact}
	\begin{split}
		\langle Q_L(f,f),\psi\rangle :=-\frac{1}{2}\iiRs |v-v_*|^{2+\gamma}[(\nabla - \nabla_*)(ff_*)]^T \Pi[v-v_*](\nabla \psi - \nabla_* \psi_*) dv_* dv.
	\end{split}
\end{align}
Following and extending the notation in~\cite{CDDW20}, for functions $\psi(v)$ and $\psi(v,v_*)$, we define the operators
\begin{align*}
	\tn\psi (v,v_*) &:= |v-v_*|^{1+\frac{\gamma}{2}}\Pi[v-v_*](\nabla \psi(v) - \nabla\psi(v_*)), \\
	\tn \psi(v,v_*) &:= |v-v_*|^{1+\frac{\gamma}{2}}\Pi[v-v_*](\nabla - \nabla_*) \psi(v,v_*).
\end{align*}
This leads to the abbreviation of~\eqref{eq:lanact} for $Q_L(f,f)$ acting on test functions $\psi$ by either
\begin{align}
	\langle Q_L(f,f),\psi\rangle &= - \frac{1}{2}\iiRs [\tn (ff_*)]^T \tn \psi \, dv_* dv, \label{eq:lantest1} \\
	&= \frac{1}{2}\iiRs ff_* \tn \cdot \tn \psi \, dv_* dv. \label{eq:lantest2}
\end{align}
\eqref{eq:lantest2} is formally obtained by integrating by parts $\tn$ from~\eqref{eq:lantest1}. The operator $\tn\cdot$ is notation for the adjoint to $\tn$ meaning that for a vector field $A(v,v_*) \in \R^3$, it reads
\[
[\tn \cdot A](v,v_*) := (\nabla - \nabla_*)\cdot(\Pi[v-v_*]|v-v_*|^{1+\frac{\gamma}{2}}A(v,v_*)).
\]
To be explicit, $\tn\cdot \tn \psi$ should just be thought of as notation for
\[
\tn\cdot \tn \psi = |v-v_*|^{2+\gamma}(\nabla-\nabla_*)(\Pi[v-v_*](\nabla-\nabla_*)\psi),
\]
where we have used that $(\nabla-\nabla_*)|v-v_*| \perp \Pi[v-v_*]$.
\begin{lemma}[Formal Grazing Collision Limit]
	\label{lem:gcsketch}
	For fixed sufficiently smooth $f, \, \psi$ we have
	\[
	\langle Q_B^\epsilon(f,f),\psi\rangle \overset{\epsilon\downarrow 0}{\to} \langle  Q_L(f,f),\psi\rangle.
	\]
\end{lemma}
\begin{proof}[Sketch of the proof]
	We work at the level of~\eqref{eq:boltztest2} converging to~\eqref{eq:lantest2}. We change variables to spherical coordinates while factoring $B^\epsilon$ into its kinetic and angular parts, $B^\epsilon(v-v_*,\theta)\sin \theta = |v-v_*|^\gamma \beta^\epsilon(\theta)$;
	\begin{align*}
		\langle Q_B^\epsilon(f,f),\psi\rangle 
		&= \frac{1}{2}\iiRs ff_* |v-v_*|^\gamma \int_{\theta = 0}^{\epsilon/2}\beta^\epsilon(\theta) \iSk\bn \psi \, dp d\theta dv_* dv \\
		&= -\frac{\pi^2}{2}\iiRs ff_* |v-v_*|^\gamma \int_{\chi = 0}^{\pi/2} \beta(\chi) \frac{1}{\epsilon^2}\iSk \bn \psi \, dp d\chi dv_* dv.
	\end{align*}
	The last line involves the rescaling of $\theta = \epsilon\chi/\pi$. According to~\Cref{lem:epszerobnpsiavgv2} (in the particular case of functions $\psi = \psi(v)$), we have 
	\[
	\frac{1}{\epsilon^2\chi^2} \int_{\Sk} \bn \psi \, dp \overset{\epsilon\downarrow 0}{\to} \frac{1}{8\pi} (\nabla - \nabla_*)\cdot \left[ 
	|v-v_*|^2 \Pi[v-v_*](\nabla \psi - \nabla_* \psi_*)
	\right].
	\]
	Returning to the computations, the formal passage of the limit $\epsilon\downarrow 0$ inside the integral gives
	\begin{align*}
		\langle Q_B^\epsilon(f,f_*),\psi\rangle &\overset{\epsilon\downarrow 0}{\to}
		\frac{\pi}{16}\iiRs ff_* |v-v_*|^{2+\gamma}\left(
		\int_0^{\pi/2}\chi^2 \beta(\chi)d\chi
		\right) (\nabla-\nabla_*)\cdot [\Pi[v-v_*](\nabla-\nabla_*)\psi] dv_* dv \\
		&=  \frac{1}{2} \iiRs ff_* |v-v_*|^{2+\gamma} (\nabla-\nabla_*)\cdot [\Pi[v-v_*] (\nabla- \nabla_*)\psi]dv_* dv,
	\end{align*}
	where the last line follows from~\Cref{lem:pilem} and we have used the finite angular momentum transfer~\eqref{eq:betaint} with the normalization $C_\beta = \frac{\pi}{8}\int_0^\frac{\pi}{2} \chi^2 \beta(\chi) d\chi = 1$. Notationally, we recognize $\tn\cdot\tn\psi$ in the integrand.
\end{proof}
We invite the reader to (formally) verify the same passage of the grazing collision limit starting from the `first order' formulations of the collision operators; \eqref{eq:boltztest1} converging to~\eqref{eq:lantest1}. We reiterate from this proof that the averaged quantity $\iSk \bn \psi \, dp$ \textbf{behaves like a second order derivative of $\psi$} which we shall make precise in~\Cref{lem:epszerobnpsiavgv2}. Consequently, the term to control is the angular momentum transfer~\eqref{eq:betaint}. Furthermore, this approach involves \textbf{no derivatives of $f$} and is even amenable to weak-strong convergence pairs $\left(f^\epsilon, \, \iSk \bn \psi B \,  dp\right) \to \left(f, \, \tn\cdot \tn \psi\right)$ which we will exploit later on. By weak-strong convergence, we mean that $f^\epsilon$ converges to $f$ weakly whereas we shall consider sufficiently smooth fixed $\psi$ so that $\iSk \bn \psi B \, dp$ converges strongly to $\tn\cdot \tn \psi$.

We now provide the supplementary estimates that prove $\bn \overset{\epsilon\downarrow 0}{\to} \tn$ which we already used in the proof of~\Cref{lem:gcsketch}. Given the notation with $\bn, \,\tn$ we consider the more general case for these operators acting on functions $\psi(v,v_*)$ which will be used later; the specific case of functions of a single variable $\psi = \psi(v)$ follows as a corollary. To better facilitate this, let us consider the following classical volume-preserving change of variables to momentum and relative velocity coordinates also considered by Bobylev~\cite{B88} and Villani~\cite{V98}, for instance. Define
\[
x = \frac{v-v_*}{2}, \quad y = \frac{v+v_*}{2},
\]
and we also note in particular that $k = x/|x|$ which gives $\sigma - k = \sigma - \frac{x}{|x|}$. Recalling the definitions of the post-collision velocities, for fixed $\sigma \in \Stwo$, this leads to
\[
\left\{
\begin{array}{cl}
	v'     &= y + |x|\sigma  \\
	v_*'     &= y - |x|\sigma 
\end{array}
\right.\implies x' = |x|\sigma, \quad y' = y.
\]
Moreover, the differentiation in $\tn$ really only sees the $x$ direction in the sense that $\nabla_v - \nabla_{v_*} = \nabla_x$. Thus, abusing notation, we can write
\begin{equation}
	\label{eq:bntnx}
	\bn \psi(x,y) = \psi(|x|\sigma,y) + \psi(-|x|\sigma,y) - \psi(x,y) - \psi(-x,y), \quad \tn \psi = |2x|^{1+\frac{\gamma}{2}}\Pi[x]\nabla_x \psi(x,y),
\end{equation}
where we have identified $\psi = \psi(v,v_*) = \psi(x,y)$. Using this change of variables, we can prove
\begin{lemma}[Estimates for $\bn \psi$ adapted from~\cite{V98}]
	\label{lem:estbnv2}
	Fix $\psi \in C_c^\infty(\R^6)$, we have the first order estimate in $x$ and $\sigma - k$
	\[
	|\bn \psi| \le Lip_x(\psi) |2x||\sigma - k|,
	\]
	the second order estimate in $x$ but still first order in $\sigma - k$
	\[
	|\bn \psi| \le  \|D_x^2\psi\|_{L^\infty}|2x|^2 |\sigma - k|.
	\]
	By averaging over the circle $p\in \Sk$, we can obtain a second order estimate in $\sigma - k$
	\[
	\left|
	\frac{1}{2\pi}\int_{\Sk} \bn \psi \, dp
	\right| \le  \|D_x^2\psi\|_{L^\infty}|2x|^2|\sigma - k|^2,
	\]
	where the geometric meaning of  $\Sk$ and $p$ given $k$ and $\sigma$ are from Figures~\ref{fig:pic1}--\ref{fig:phi}.
\end{lemma}
\begin{proof}
	We use~\eqref{eq:bntnx} and the Fundamental Theorem of Calculus to write
	\begin{align*}
		\bn \psi &= \psi(|x|\sigma) + \psi(-|x|\sigma) - \psi(x) - \psi(-x) \\
		&= \int_0^1 \frac{d}{dt}\left\{\psi(t|x|\sigma + (1-t)x) + \psi(-[t|x|\sigma + (1-t)x])\right\} dt \\
		&= (|x|\sigma - x) \cdot \int_0^1 \nabla_x \psi(t|x|\sigma + (1-t)x) - \nabla_x \psi(-[t|x|\sigma + (1-t)x])dt.
	\end{align*}
	The first estimate is directly obtained from here. For the second estimate, we continue by using Taylor's formula to replace
	\begin{align*}
		\nabla_x \psi(t|x|\sigma + (1-t)x) =&\nabla_x \psi(x) + t(|x|\sigma - x) \cdot \int_0^1 D_x^2 \psi (s[t|x|\sigma + (1-t)x] + (1-s)x)ds, \\
		\nabla_x \psi(-[t|x|\sigma + (1-t)x]) = &\nabla_x\psi(-x) -t(|x|\sigma - x)\cdot \int_0^1 D_x^2 \psi (-\{s[t|x|\sigma + (1-t)x] + (1-s)x\})ds.
	\end{align*}
	Thus, we have
		\begin{align}
			\label{eq:Taylorbn}
			\begin{split}
				\bn \psi &= |x|(\sigma - k) \cdot (\nabla_x \psi(x) - \nabla_x\psi(-x)) \\
				&\qquad + |x|^2(\sigma - k)\otimes (\sigma - k) : \\
				&\qquad \qquad :\underbrace{\int_0^1 t\int_0^1 D_x^2 \psi (s[t|x|\sigma + (1-t)x] + (1-s)x) + D_x^2 \psi (-\{s[t|x|\sigma + (1-t)x] + (1-s)x\}) dsdt}_{\le \|D_x^2\psi\|_{L^\infty}}.
			\end{split}
	\end{align}
	The second term is bounded by $\|D_x^2\psi\|_{L^\infty}|x|^2|\sigma-k|^2 \le 2\|D_x^2\psi\|_{L^\infty}|x|^2|\sigma-k|$  (recalling $|\sigma - k|^2 = 2(1- \cos\theta)$ and we have restricted $\theta \in [0,\pi/2]$) which gives the right contribution for the second estimate. Thus, it suffices to look at the first term which yields the right estimate after an application of the Mean Value Theorem.
	
	Finally, for the third estimate which improves the order with respect to $\sigma - k$, we take the average of~\eqref{eq:Taylorbn} over $\Sk$. Again, the second term is bounded in the right way ($\sigma - k$ depends on $p$, but $|\sigma - k|$ does not), so we only focus on the first term
	\begin{align*}
		\frac{1}{2\pi}\int_{\Sk}|x|(\sigma- k)\cdot (\nabla_x\psi(x) - \nabla_x\psi(-x))dp 
		&= \frac{|x|}{2\pi}\iSk ((\cos \theta - 1)k + \sin\theta p) \cdot (N_k \, k + N_p \, p)\, dp \\
		&= \frac{|x|}{2\pi}\iSk (\cos\theta -1)N_k + \sin\theta N_p \, dp.
	\end{align*}
	where we have decomposed the vectors in the inner product under this geometry, see~\eqref{eq:sigkp}, as
	\[
	N_k = (\nabla_x \psi(x) - \nabla_x\psi(-x))\cdot k, \quad N_p = (\nabla_x \psi(x) - \nabla_x\psi(-x)) \cdot p. 
	\]
	Continuing, notice that the $N_p$ term is a linear combination of $\cos\phi$ and $\sin\phi$ ($\phi$ being the azimuthal angle in~\Cref{fig:phi}) and that the integral can be written as
	\[
	\iSk dp = \int_0^{2\pi} d\phi.
	\]
	Hence, the second term integrates to zero and we focus on the first term. The Mean Value Theorem gives $|N_k| = |k\cdot (\nabla_x\psi(x) - \nabla_x\psi(-x))| \le |2x|\|D_x^2\psi\|_{L^\infty}$ which yields
	\begin{align*}
		\left|
		\frac{|x|}{2\pi}\iSk (\sigma - k)\cdot (\nabla_x\psi(x) - \nabla_x\psi(-x))dp
		\right| & \le 2|x|^2\|D_x^2\psi\|_{L^\infty}|1 - \cos\theta| = \|D_x^2\psi\|_{L^\infty}|x|^2|\sigma - k|^2.
	\end{align*}
\end{proof}
\begin{lemma}[Behaviour of $\bn \psi/\epsilon$]
	\label{lem:epszerobnpsiv2}
	Under the previous notations, for $\psi \in C_c^\infty(\R^6),\, \epsilon>0,\, \chi = \pi\theta/\epsilon$, where $\cos \theta = k\cdot \sigma$ and $\theta \in [0,\epsilon/2]$, we have
	\[
	\frac{1}{\epsilon}\bn \psi = \frac{|x|}{\epsilon}\left(
	\left(
	\cos \frac{\epsilon\chi}{\pi} - 1
	\right)k + \sin \frac{\epsilon\chi}{\pi}\, p
	\right)\cdot \nabla_x [\psi(x) + \psi(-x)] + \mathcal{O}(\|D_x^2\psi\|_{L^\infty}\epsilon|x|^2).
	\]
	In particular, we have the convergence
	\[
	\frac{1}{\epsilon}\bn \psi \overset{\epsilon \downarrow 0}{\to} \frac{\chi}{\pi} |x| \, p\cdot \nabla_x [\psi(x) + \psi(-x)], \quad \text{pointwise }v, \,v_* \in \R^3.
	\]
	For functions $\psi = \psi(v)$, this is equivalent to
	\[
	\frac{1}{\epsilon}\bn \psi \overset{\epsilon\downarrow 0}{\to} \frac{\chi}{2\pi} |v-v_*| \,p\cdot (\nabla - \nabla_*) \psi.
	\]
\end{lemma}
\begin{proof}
	Firstly, we recall the size estimates of $\sigma - k$ from~\eqref{eq:equivang} and~\eqref{eq:equivmom} giving
	\[
	|\sigma - k|^2 = 2(1 - \cos\theta) \le\theta^2 = \epsilon^2 \frac{\chi^2}{\pi^2}.
	\]
	Using again $\sigma - k = (\cos\theta - 1)k + \sin\theta \, p$ and the substitution $\theta = \epsilon \chi/\pi$, we obtain the first identity starting from~\eqref{eq:Taylorbn}.
	
	Passing to the limit $\epsilon\downarrow 0$ is a matter of recalling Taylor expansions.
\end{proof}
\begin{lemma}[Behaviour of $\frac{1}{\epsilon^2}\iSk \bn \psi \, dp$]
	\label{lem:epszerobnpsiavgv2}
	Under the previous notations, for $\psi \in C_c^\infty(\R^6),\, \epsilon>0, \,\cos\theta = k\cdot \sigma,\, \theta \in [0,\epsilon/2],\, \chi = \pi\theta/\epsilon$, we have
	\[
	\frac{1}{\epsilon^2\chi^2}\iSk \bn \psi dp \overset{\epsilon\downarrow 0}{\to} \frac{1}{2\pi}\nabla_x\cdot (|x|^2\Pi[x]\nabla_x [\psi(x,y)+\psi(-x,y)]).
	\]
	For functions $\psi = \psi(v)$, this is equivalent to
	\[
	\frac{1}{\epsilon^2\chi^2}\iSk \bn \psi dp \overset{\epsilon\downarrow 0}{\to} \frac{1}{8\pi} (\nabla- \nabla_*)\cdot (|v-v_*|^2\Pi[v-v_*](\nabla-\nabla_*)\psi).
	\]
\end{lemma}
\begin{proof}
	We start from~\eqref{eq:Taylorbn}, the terms there are recalled
		\begin{align*}
			\bn \psi &= \overbrace{|x|(\sigma - k) \cdot (\nabla_x \psi(x) - \nabla_x\psi(-x))}^{=: T_1} \\
			&\qquad + |x|^2(\sigma - k)\otimes (\sigma - k) : \\
			&\qquad \qquad :\int_0^1 t\int_0^1 D_x^2 \psi (s[t|x|\sigma + (1-t)x] + (1-s)x) + D_x^2 \psi (-\{s[t|x|\sigma + (1-t)x] + (1-s)x\}) dsdt \\
			&=: T_1 + T_2.
	\end{align*}
	The idea is to individually take the limits
	\[
	\lim_{\epsilon \downarrow 0}\frac{1}{\epsilon^2}\iSk T_i dp, \quad i=1, \,2. 
	\]
	Let us start with $T_2$ which will involve a Dominated Convergence argument. By the triangle inequality and~\eqref{eq:equivang}
	\[
	|T_2| \le \|D_x^2\psi\|_{L^\infty}|x|^2|\sigma - k|^2 = \epsilon^2\|D_x^2\psi\|_{L^\infty}|x|^2  \frac{\chi^2}{\pi^2}.
	\]
	This is an integrable majorant when multiplied against $\epsilon^{-2}$ so we obtain
	\[
	\lim_{\epsilon\downarrow 0}\iSk \frac{T_2}{\epsilon^2}dp = \frac{\chi^2}{2\pi^2}|x|^2\iSk p\otimes p dp : (D_x^2\psi(x) + D_x^2\psi(-x)). 
	\]
	This is because, recalling the expression of $\sigma$ with respect to $k, \, p$ in~\eqref{eq:sigkp} and the scaling $\theta \sim \epsilon$, we see
	\[
	\sigma - k = \underbrace{(\cos\theta - 1)}_{\sim \epsilon^2} \, k + \underbrace{\sin \theta}_{\sim\epsilon} \, p.
	\]
	Therefore, the remaining contribution is $\lim_{\epsilon\downarrow 0} \epsilon^{-1}(\sigma - k) = p.$ Using~\Cref{lem:pilem}, we can compute $\iSk p\otimes p \, dp$ so that we obtain
	\[
	\lim_{\epsilon\downarrow 0}\iSk \frac{T_2}{\epsilon^2\chi^2}dp = \frac{|x|^2}{2\pi}\Pi[x]: (D_x^2\psi(x) + D_x^2\psi(-x)).
	\]
	Turning to the $T_1$ term, we directly integrate similar to the proof of~\Cref{lem:estbnv2}. Copying the notation there, we have
	\begin{align*}
		\iSk \frac{T_1}{\epsilon^2}dp &= \frac{1}{\epsilon^2}\iSk |x|(\sigma - k)\cdot (\nabla_x\psi(x) - \nabla_x\psi(-x)) dp = \frac{|x|N_k}{\epsilon^2}\iSk (\cos\theta - 1) dp \\
		&= 2\pi (x\cdot(\nabla_x\psi(x) - \nabla_x\psi(-x))) \frac{\cos\theta - 1}{\epsilon^2}
	\end{align*}
	recalling $N_k = (\nabla_x\psi(x) - \nabla_x(\psi(-x))\cdot k$ (independent of $p$!) and the other contribution vanishes. Changing variables $\theta = \epsilon \chi /\pi$, we use the fact that
	\[
	\frac{\cos(\epsilon\chi/\pi) - 1}{\epsilon^2} \overset{\epsilon\downarrow 0}{\to} -\frac{\chi^2}{2\pi^2}
	\qquad \mbox{to obtain} \qquad
	\lim_{\epsilon\downarrow 0}\iSk \frac{T_1}{\epsilon^2\chi^2}dp = -\frac{1}{\pi}x\cdot (\nabla_x\psi(x) - \nabla_x\psi(-x)).
	\]
	Putting both terms together, we have
	\begin{align*}
		&\frac{1}{\epsilon^2\chi^2}\iSk \bn \psi dp \overset{\epsilon\downarrow 0}{\to} \frac{|x|^2}{2\pi}\Pi[x]: (D_x^2\psi(x) + D_x^2\psi(-x)) -\frac{1}{\pi}x\cdot (\nabla_x\psi(x) - \nabla_x\psi(-x)).
	\end{align*}
	In $d$ dimensions, a direct computation shows $\nabla_x \cdot |x|^2 \Pi[x] = -(d-1)x$, which in this case for $d = 3$ allows one to recognize the product rule and conclude.
\end{proof}

\section{Gradient flow terminology and structure}
\label{sec:gradflow}
We first recall the abstract framework of gradient flows that can be found from~\cite[Chapter 1]{AGS08}. After these definitions have been clarified, we shall recall some of the theory of the Boltzmann and Landau continuity equations in~\cite{E19,CDDW20}. This perspective is crucial for the compactness results in~\Cref{sec:cpct}. Let us denote $(X,d)$ to be a complete (pseudo)-metric space $X$ with (pseudo)-metric $d$. We will refer to $a < b \in \R$ as endpoints of some interval. We denote $F:X\to (-\infty,\infty]$ to be a fixed proper function.
\begin{definition}[Absolutely continuous curve]
	\label{defn:abscontcurve}
	A function $\mu : t\in(a,b) \mapsto \mu_t\in X$ is said to be an \textit{absolutely continuous curve} if there exists $m\in L^1(a,b)$ such that for every $s\leq t\in(a,b),$
	\[
	d(\mu_t,\mu_s) \leq \int_s^tm(r)dr.
	\]
\end{definition}
Among all admissible $m$ in~\Cref{defn:abscontcurve}, one can make the following minimal selection by the Lebesgue differentiation theorem.
\begin{definition}[Metric derivative]\label{defn:metder}
	For an absolutely continuous curve $\mu : (a,b) \to X$, we define its \textit{metric derivative} at every $t\in(a,b)$ by
	\[
	|\dot{\mu}|(t) := \lim_{h\to0}\frac{d(\mu_{t+h},\mu_t)}{|h|}.
	\]
\end{definition}
We refer to~\cite[Theorem 1.1.2]{AGS08} for more properties of the metric derivative.
\begin{definition}[Strong upper gradient]
	\label{defn:sug}
	We say the function $g:X\to [0,\infty]$ is a \textit{strong upper gradient} with respect to $F$ if for every absolutely continuous curve $\mu : t\in (a,b) \mapsto \mu_t \in X$, we have that $\sqrt{g} \circ \mu :(a,b) \to [0,\infty]$ is Borel measurable and the following inequality holds
	\[
	|F[\mu_t] - F[\mu_s]| \leq \int_s^t \sqrt{g(\mu_r)}|\dot{\mu}|(r)dr, \quad \forall a < s \leq t <b.
	\]
\end{definition}

We now turn to the generalized notions of continuity equation and action as developed by~\cite{E19,CDDW20}. We denote by $\mathscr{P}_p$ the space of probability measures on $\R^3$ with finite $p>0$ moments endowed with the weak-* topology as members of the dual of bounded continuous functions (denoted $C_b(\R^3)$ or just $C_b$ if no confusion arises). We denote by $\mathscr{P}$ the space of probability measures without any moments assumptions. For sequences $\{\mu_n\}_{n\in \mathbb{N}} \subset \mathscr{P}$ converging in this topology to $\mu \in \mathscr{P}$, we write $\mu_n \overset{\sigma}{\rightharpoonup} \mu$. We will abuse notation if each of these measures are absolutely continuous with respect to Lebesgue measure; if $\mu_n = f_n \mathcal{L}$ and $\mu = f \mathcal{L}$ we may write $f_n \overset{\sigma}{\rightharpoonup} f$ meaning $\mu_n \overset{\sigma}{\rightharpoonup} \mu$. We denote $\mathcal{M}_B$ the space of (scalar) signed Radon measures on $\R^6\times \Stwo$ endowed with the weak-* topology as members of the dual of compactly supported continuous functions (denoted $C_0(\R^6\times \Stwo)$ or just $C_0$ if no confusion arises). We denote $\mathcal{M}_L$ the space of $\R^3$-valued signed Radon measures on $\R^6$ endowed with the weak-* topology as members of the dual of compactly supported continuous functions.
\begin{definition}[Generalized continuity equations]
	\label{def:ctyeqn}
	For Borel curves in time $t\in [0,T]$, $\mu : t \mapsto \mu_t \in \mathscr{P}$ and $M : t \mapsto M_t \in \mathcal{M}_B$ we say $(\mu, \,M)\in CRE_T$ (or $CRE$ if $T=1$), if they satisfy Erbar's \textit{collision rate equation}~\cite{E19}
	\[
	\partial_t \mu_t + \frac{1}{4}\bn \cdot M_t = 0,
	\]
	in the distributional sense. By this, we mean that for any $\phi \in C_c^\infty((0,T)\times \R^3)$,
	\[
	\int_0^T\int_{\R^3}\partial_t \phi(t,v) d\mu_t(v) dt + \frac{1}{4}\int_0^T\iiRs \iS \bn \phi(t,v,v_*,\sigma) dM_t(v,v_*,\sigma) dt = 0.
	\]
	Similarly, for Borel curves $\mu_t \in \mathscr{P}, \, M_t \in \mathcal{M}_L$, we say $(\mu, \,M) \in GCE_T$ (or $GCE$ if $T=1$), if they satisfy the \textit{grazing continuity equation}~\cite{CDDW20}
	\[
	\partial_t \mu_t + \frac{1}{2}\tn \cdot M_t = 0,
	\]
	in the distributional sense. For curves $(\mu_t)_t$ in $CRE_T$ or $GCE_T$, we also insist that the second moment is finite and non-increasing
	\[
	\int_{\R^3}|v|^2d\mu_t(v) \le  \int_{\R^3}|v|^2d\mu_s(v)\qquad\mbox{for any $0\le s\le t\le T$.}
	\]
\end{definition} 
Throughout the document, we will actually refer to probability measures with densities $\mu = f \mathcal{L}$ against Lebesgue measure. As well, if $M \in \mathcal{M}_B$ has a density against Lebesgue measure on $\R^6\times \Stwo$, we will refer to that density also by $M$ and similarly for measures in $\mathcal{M}_L$. For $(\mu,M)\in CRE_T$, consider $\tau = \mu(dv)\mu(dv_*) + |M|$ as a Radon measure on $\R^6\times \Stwo$ that dominates both $\mu\otimes \mu$ and $M$. Taking $N\tau=M$ and $g \tau=\mu(dv)\mu(dv_*)$, the densities of $M$ and $\mu\otimes\mu$ with respect to $\tau$, we  can define the \textit{Boltzmann action} of the curve $\mu$ by
\[
\mathcal{A}_B(\mu,M) := \frac{1}{4}\iiRs \iS \frac{|N|^2}{\left[\frac{g' - g}{\log g' - \log g}\right]B}d\sigma dv_* dv,
\]
where the square bracketed quantity is the logarithmic mean between $g$ and $g'=g(v',v_*')$. More precisely, writing $L(s,t)$ as the logarithmic mean between $s, t>0$ the integrand of $\mathcal{A}_B$ should be the function $\alpha(N, g, g')/B$ where $\alpha$ is the convex, lower semi-continuous, 1-homogeneous function
	\[
	\alpha(u,s,t) := \left\{
	\begin{array}{cc}
		\frac{|u|^2}{4L(s,t)}, &L(s,t)\neq 0, \\
		0, 	&L(s,t) = 0, \, u = 0 \\
		+\infty, 	&L(s,t)=0, \, u \neq 0
	\end{array}
	\right., \quad \mathcal{A}_B(\mu, M) = \iiRs \iS \alpha(N, g, g')/B.
	\]
	The appearance of $B$ is free and we have chosen to place it in the denominator so that the collision rate equation may be written without the collision kernel. In the case where both $M$ and $\mu$ have densities with respect to Lebesgue (denoting again $\mu = f\mathcal{L}$), we write
\[
\Lambda = \Lambda(f) = \frac{f'f_*' - ff_*}{\log f'f_*' - \log ff_*}.
\]
For $\mu^\epsilon = f^\epsilon \mathcal{L}$ corresponding to curves with respect to the collision kernel $B^\epsilon$, we will write $\Lambda^\epsilon = \Lambda(f^\epsilon).$ Analogously, for $(\mu,M) \in GCE_T$, we consider $\tau = \mu\otimes \mu + |M|$ a signed Radon measure on $\R^6$ that dominates both $\mu\otimes \mu$ and $M$, and we set $g\tau= \mu(dv) \mu(dv_*)$ and $N\tau=M$. We define the \textit{Landau action of the curve $\mu$} by
\[
\mathcal{A}_L(\mu,M) := \frac{1}{2}\iiRs \frac{|N|^2}{g}dv_* dv.
\]
More precisely, the integrand of $\mathcal{A}_L$ should be the function $\kappa(N, g)$ where $\kappa$ is the convex, lower semi-continuous, 1-homogeneous function
	\[
	\kappa(u, s) := \left\{
	\begin{array}{cc}
		\frac{|u|^2}{2s}, 	&s\neq 0 \\
		0, 	&s = 0, \, u = 0 \\
		+\infty, 	&s = 0, \, u \neq 0
	\end{array}
	\right., \quad \mathcal{A}_L(\mu, M) = \iiRs \kappa(N, g).
	\]
These definitions for the action functionals are independent of dominating measure $\tau$ since the integrands are 1-homogeneous. In the works of~\cite{E19,CDDW20}, (pseudo)-metrics denoted by $d_B$ and $d_L$ were constructed to give a gradient flow notion of solutions to the Boltzmann and Landau equations, respectively. For example, in the Landau case the distance is defined by
\[
d_L^2(\lambda_1,\lambda_2) := \inf \left\{ T
\int_0^T\mathcal{A}_L(\mu(t),M(t)) dt \, \left|
\begin{array}{c}
	(\mu,M) \in GCE_T,  \\
	\mu(0) = \lambda_1, \quad \mu(T) = \lambda_2
\end{array}
\right.
\right\}, \quad \lambda_1,\lambda_2 \in \mathscr{P}_2.
\]
The Boltzmann metric is defined analogously, by replacing $\mathcal{A}_L$ by $\mathcal{A}_B$ and $GCE_T$ by $CRE_T$. We will write $d_{B^\epsilon}$ or $d_\epsilon$ for the Boltzmann metric corresponding to the collision kernel $B^\epsilon$. We have already used the notation in~\eqref{eq:EDE}, but we make precise here that $|\dot{f}|_\epsilon, |\dot{f}|_L$ refer to the $d_\epsilon,d_L$-metric derivatives for a curve $\mu_t = f_t\mathcal{L}$, respectively. 

\begin{lemma}[Propositions A.9, A.11, and Corollary A.10 of~\cite{E19}]
	\label{lem:Erbarmetvel}
	For a fixed collision kernel $B$, $t \in [0,T]\mapsto f_t\mathcal{L} \in \mathscr{P}$ is absolutely continuous with respect to the Boltzmann distance $d_B$ if and only if there is a family of mobilities $M_t^B \in \mathcal{M}_B$ such that $(f, \, M^B) \in CRE_T$ with finite total action
	\[
	\int_0^T \mathcal{A}_B(f,M_B) dt = \frac{1}{4}\int_0^T\iiRs \iS \frac{|M^B|^2}{\Lambda(f) B}d\sigma dv_* dv dt < +\infty.
	\]
	Moreover, there is a unique $\tilde{M}^B\in \mathcal{M}_B$ such that
	\[
	|\dot{f}|_B^2(t) = \mathcal{A}_B(f,\tilde{M}_B) =  \frac{1}{4}\iiRs \iS \frac{|\tilde{M}^B|^2}{\Lambda(f) B} d\sigma dv_* dv.
	\]
	Furthermore, under the class of admissible $M^B\in\mathcal{M}_B$ (i.e. $(f,M^B)\in GCE_T$), $\tilde{M}^B$ is characterized by the minimization property
	\[
	\iiRs \iS \frac{|\tilde{M}^B|^2}{\Lambda(f) B} d\sigma dv_* dv \le \iiRs \iS \frac{|\tilde{M}^B+\eta|^2}{\Lambda(f) B} d\sigma dv_* dv.
	\]
	for any $\eta\in \mathcal{M}_B$ which is $\bn\cdot$-free, that is to say
	\[
	\iiRs \iS \bn \xi\; d\eta = 0, \quad \forall \xi \in C_c^\infty(\R^3).
	\]
	If a measure $M\in\mathcal{M}_B$ satisfies the minimization property, then $M = U \Lambda(f) B$ where
	\[
	U \in \overline{\{
		\bn \phi \, | \, \phi \in C_c^\infty(\R^3)
		\}}^{L^2(\Lambda(f)Bd\sigma dv_* dv)}.
	\]
	The analogous statements for the Landau equation hold in the sense of the following replacements: the collision rate equation $\to$ grazing continuity equation, $\mathcal{A}_B \to \mathcal{A}_L$, $|\dot{f}|_B \to |\dot{f}|_L$, $\bn \to \tn$, $\Lambda B d\sigma dv_* dv \to ff_* dv_* dv$.
\end{lemma}

We end this section with some remarks in order to demystify the gradient flow concepts in our proof of~\Cref{thm:grazcoll} without going into the details.
\begin{remark}
	We recall the uniform metric derivative integrability~\eqref{eq:bounds} in the proof of~\Cref{thm:grazcoll}
	\[
	\sup_{\epsilon>0}\int_0^T |\dot{f}^\epsilon|_\epsilon^2(t) dt < + \infty.
	\]
	Based on the abstract framework we have just reviewed, this estimate yields the following which we shall revisit in~\Cref{sec:cpct,sec:gammamd}.
	\begin{itemize}
		\item (Regularity) - Each curve we consider $t \mapsto f^\epsilon(t)$ is absolutely continuous with respect to $d_\epsilon$. Moreover, this property is uniform in $\epsilon>0$.
		\item (Compactness) - Furthermore, from~\Cref{lem:Erbarmetvel}, we can evaluate the metric derivative as the action of a unique collision rate $M^\epsilon$:
		$
		|\dot{f}^\epsilon|_\epsilon^2(t) = \mathcal{A}_B(f^\epsilon(t),M^\epsilon(t)).
		$
		This representation has two consequences - firstly, our assumptions allow us to prove compactness of $(f^\epsilon,M^\epsilon)_{\epsilon>0}$ in~\Cref{sec:cpct} to some limit $(f^0,M^0)$. Secondly, it is easier to work with the jointly convex integrands of $\mathcal{A}_B$ and $\mathcal{A}_L$ to show Step~\ref{step:gammamd} from the proof of~\Cref{thm:grazcoll};
		\[
		\liminf_{\epsilon\downarrow 0}|\dot{f}^\epsilon|_{\epsilon}^2(t) = \liminf_{\epsilon\downarrow 0}\mathcal{A}_B(f^\epsilon,M^\epsilon) \ge \mathcal{A}_L(f^0,M^0) \ge |\dot{f}|_L^2(t).
		\]
	\end{itemize}
\end{remark}

\begin{remark}
	\label{rem:mdeqdiss}
	Given a weak solution $f=f(t)=f_t$ of the Landau equation (resp. $f^\epsilon$ of the Boltzmann equation) that dissipates entropy as in~\Cref{rem:renormalH}, one immediately obtain an estimate for the metric derivative;
	\[
	|\dot{f}|_L^2(t) \le D_L(f(t)), \quad (\text{resp. }|\dot{f}^\epsilon|_\epsilon^2(t) \le D_B^\epsilon(f^\epsilon(t))).
	\]
	This is owed to the fact that one can take the admissible collision rate
		\[
		M = - ff_* \tn \log f, \quad (\text{resp. } M_B^\epsilon = - \Lambda(f^\epsilon) B^\epsilon \bn \log f^\epsilon),
		\]
		giving an upper bound for the square of the Landau metric
		\[
		d_L^2(f_0, \, f_1) \le \int_0^1 \mathcal{A}_L(f_t, M) dt = \frac{1}{2}\int_0^1 \iiRs ff_* |\tn \log f|^2 dv_* dv dt < +\infty.
		\]
		An analogous estimate holds for the Boltzmann metric. In this way, the dissipation of entropy implies the EDI, hence renormalized solutions are H-gradient flow solutions. 
\end{remark}
\section{Compactness of curves}
\label{sec:cpct}
The aim of this section is to deduce general compactness results of curves $(f^\epsilon)_{\epsilon>0}$ (not necessarily solutions in any sense) subject to the uniform moments and metric derivative bounds in~\eqref{eq:bounds} from the proof of~\Cref{thm:grazcoll}. The first half of this section establishes~\Cref{cor:limcurve} and~\Cref{prop:sqrtcpct} which are the main compactness results. The second half of this section, \Cref{sec:lanlim}, builds on these compactness results by confirming the passage of the grazing collision limit at the level of the generalized continuity equations; $CRE$ to $GCE$. Here, we have taken $T=1$ for simplicity but the results work for $CRE_T$ and $GCE_T$.

\begin{theorem}[Compactness of $f^\epsilon$]
	\label{cor:limcurve}
	Let $f^\epsilon : [0,1] \to \mathscr{P}_2$ be curves satisfying the uniform moments and metric derivative bounds~\eqref{eq:bounds}. Then there exists $f:[0,1]\to \mathscr{P}_2$ obtained by a convergent subsequence such that
	\[
	f^{\epsilon}(t) \overset{\sigma}{\rightharpoonup} f(t), \quad \forall t \in [0,1],
	\]
	and, in the case of $\gamma\in[-2,0]$, $f$ is continuous in duality against 1-Lipschitz functions (respectively, continuous in duality against bounded functions with second derivatives bounded by 1 in the case of $\gamma\in[-4,-2)$).
\end{theorem}
\begin{remark}
	By the well-known Kantorovich-Rubinstein duality, the continuity in duality against 1-Lipschitz functions in the case of $\gamma\in[-2,0]$ is equivalent to continuity in the 1-Wasserstein metric~\cite{V09}. This was noticed by Erbar~\cite{E19} in the case $\gamma=0$, whose proof we have generalized using the finite angular momentum transfer~\eqref{eq:betaint}.
\end{remark}
\begin{proof}[Proof of~\Cref{cor:limcurve}]
	This is an application of a general Ascoli-Arzel\`a compactness result~\cite[Proposition 3.3.1]{AGS08} together with~\Cref{lem:lowbdd}, which says that there exists $C>0$ an explicit constant depending on the finite momentum transfer~\eqref{eq:betaint}, and the uniform moments from~\eqref{eq:bounds} such that
	\[
	\left|
	\int_{\R^3}\psi(v) (f_t^\epsilon(v) - f_s^\epsilon(v)) dv
	\right|\le C d_\epsilon(f^\epsilon(t),f^\epsilon(s)), \quad \forall s, \,t \in [0,1],
	\]
	for any function $\psi$ with Lipschitz semi-norm bounded by 1 in the case $\gamma\in[-2,0]$ (respectively, second derivative bounded by 1 in the case $\gamma\in[-4,-2)$).

	By the absolute continuity of $f^\epsilon$ and the uniform $L^2$ integrability of the metric derivative, we obtain
	\[
	\sup_{\epsilon>0}\sup_{\psi} \left|
	\int_{\R^3}\psi(v) (f_t^\epsilon(v) - f_s^\epsilon(v)) dv
	\right| \le C |t - s|^\frac{1}{2}.
	\]
	This estimate with the basic weak compactness of $(f^\epsilon)_{\epsilon > 0} \subset \mathscr{P}$ by the moment bounds in~\eqref{eq:bounds} satisfies the conditions to apply a version of Ascoli-Arzela in this setting \cite[Proposition 3.3.1]{AGS08}.
\end{proof}
\begin{proposition}
	\label{prop:sqrtcpct}
	Let $(f^\epsilon)_{\epsilon>0}$ be curves satisfying the uniform moment and metric derivative bounds in~\eqref{eq:bounds}. We consider the subsequence of $f^\epsilon$ that converges to $f$ given by~\Cref{cor:limcurve}. Assume further that there exists $t_0\in[0,1]$ such that
	$$
	\sup_{\epsilon>0} D_B^\epsilon(f^\epsilon(t_0))<\infty,
	$$
	then
	\[
	\sqrt{f^\epsilon(t_0)} \to \sqrt{f(t_0)}, \text{ in } L_{loc.}^2.
	\]
\end{proposition}
\begin{remark}
	This result is reminiscent, but weaker, than those in~\cite{L98,AV02}. There, the stronger convergence $\sqrt{f^\epsilon} \to \sqrt{f}$ in $L_{t,v}^2$ is achieved by exploiting the extra time (and space, in the case of inhomogeneity) regularity from velocity averaging methods~\cite{GLPS88} on renormalized solutions to the Boltzmann equation~\cite{DL89}.
\end{remark}
\begin{proof}[Proof of~\Cref{prop:sqrtcpct}]
	The argument is standard after recalling the main estimate of Appendix~\ref{sec:strcpct} so we shall quickly sketch the main ideas. For brevity, we fix and suppress $t_0$. Using that $f^\epsilon$ are probability densities, we immediately obtain (up to a subsequence) the weak convergence $\sqrt{f^\epsilon} \rightharpoonup g$ in $L^2$ for some $g \in L^2$. According to Appendix~\ref{sec:strcpct}, we obtain the estimate
	\[
	\sup_{\epsilon>0} \left|\left|\sqrt{f_R^\epsilon}\right|\right|_{\dot{H}^\frac{\nu}{2}} \le C_R, \quad  \forall R>1,
	\]
	where $f_R^\epsilon = f^\epsilon \chi_R$ is a smooth cut-off approximation of $f^\epsilon$ vanishing outside $B_{R+1}$ and $C_R>0$ is a constant depending on $R$ and the value of $\sup_{\epsilon>0} D_B^\epsilon(f^\epsilon)$. This upgrades the convergence so that $\sqrt{f^\epsilon} \to g$ strongly in $L_{loc.}^2$. In particular, along a further subsequence, we have $f^\epsilon \to g^2$ pointwise almost every $v\in\R^3$. By~\Cref{cor:limcurve}, this identifies $g^2 = f$ and we are done.
\end{proof}
Curves in $CRE$ and $GCE$ are pairs of measures $(f,M)$. Assuming the bounds~\eqref{eq:bounds}, we have established compactness for the first component of these curves, $f^\epsilon$. We now state and prove the compactness result for the second component, $M^\epsilon$.
\begin{proposition}[(Scaled) compactness of $M^\epsilon$]
	\label{prop:limm}
	Suppose $(f^\epsilon,M^\epsilon) \in CRE$ is a pair of curves where $(f^\epsilon)_{\epsilon>0}$ satisfies the uniform moment and metric derivative bounds~\eqref{eq:bounds}. Assume $M^\epsilon$ is the optimal collision rate given by~\Cref{lem:Erbarmetvel}.
	Then, for any $\tilde{q}\in \left[-\frac{\gamma}{2},1 - \frac{\gamma}{2}\right]$ the family $\{|v-v_*|^{\tilde{q}}\theta M^{\epsilon}\}_{\epsilon>0}$ is a bounded set of Radon measures in the weak-* topology against $C_0$ functions. In particular, choosing
	\[
	q_\gamma := \left\{
	\begin{array}{cl}
		1,     &\gamma\in[-2,0)  \\
		2,     &\gamma\in[-4,-2) 
	\end{array}
	\right., \quad 0 < \delta < \delta_\gamma := \left\{
	\begin{array}{ll}
		-\frac{\gamma}{2},     &\gamma\in[-2,0)  \\
		-\frac{\gamma}{2}-1,     &\gamma\in[-4,-2) 
	\end{array}
	\right.,
	\]
	we have that the family
	\[
	\left\{
	|v-v_*|^{q_{\gamma}} \left(
	1 + \left[ 
	|v|^2 + |v_*|^2
	\right] 
	\right)^\frac{\delta}{2} \theta M^\epsilon
	\right\}_{\epsilon>0}
	\]
	is compact in the set of Radon measures.
\end{proposition}

\begin{proof}
	We will only show the uniform bound. The compactness statement uses the same argument because the choices of $q$ and $\delta$ depending on $\gamma$ ensure $q + \delta \in \left[-\frac{\gamma}{2}, 1 - \frac{\gamma}{2} \right)$. Fix $\Psi\in C_0([0,1] \times\R^6 \times \Stwo)$ non-negative and $\tilde{q}\in \left[-\frac{\gamma}{2}, 1 - \frac{\gamma}{2} \right]$; we use~\Cref{cor:lambdaf} to estimate
	\begin{align*}
		\int_0^1 \iiRs \iS \Psi |v-v_*|^{\tilde{q}}|\theta| |M^{\epsilon}|  &=\int_0^1 \iiRs \iS \frac{|M^{\epsilon}|}{\sqrt{\Lambda (f^{\epsilon})B^{\epsilon}}} |v-v_*|^{\tilde{q}}|\theta| \Psi \sqrt{\Lambda (f^{\epsilon})B^{\epsilon}}  \\
		& \le\left(
		\int_0^1 \iiRs \iS \frac{|M^{\epsilon}|^2}{\Lambda(f^{\epsilon})B^{\epsilon}}
		\right)^\frac{1}{2}\left(
		\int_0^1 \iiRs \iS|v-v_*|^{2\tilde{q}}|\theta|^2 \Psi^2 \Lambda(f^{\epsilon})B^{\epsilon}
		\right)^\frac{1}{2}.
	\end{align*}
	Here, we have multiplied and divided by $\sqrt{\Lambda^\epsilon B^\epsilon}$ and then applied Cauchy-Schwarz. This reveals precisely the $\epsilon$-action in the first term, which by the metric derivative bound in~\eqref{eq:bounds} and~\Cref{lem:Erbarmetvel}, is bounded. Focusing on the second term, we use~\Cref{cor:lambdaf} to estimate
	\begin{align*}
		&\int_0^1 \iiRs \iS |v-v_*|^{2\tilde{q}}|\theta|^2 \Psi^2 \left[ 
		\frac{{f^{\epsilon}}'{f_*^{\epsilon}}' + f^{\epsilon}f_*^{\epsilon}}{2}
		\right] B^\epsilon d\sigma dv_* dv dt.
	\end{align*}
	By symmetry, we can pass the post-collision velocity evaluations of ${f^\epsilon}'{f^\epsilon}_*'$ onto $\Psi$. We develop $B^\epsilon \sin \theta = |v-v_*|^\gamma\beta^\epsilon$ and continue the estimate
	\begin{align*}
		&\int_0^1 \iiRs \iS \Psi |v-v_*|^{\tilde{q}}|\theta| |M^\epsilon| d\sigma dv_* dv dt \\
		&\qquad\le C \left(
		\int_0^1 \iiRs \itheta[\epsilon/2]{\theta} \iSk|v-v_*|^{2\tilde{q}+\gamma} |\theta|^2 (\Psi^2 + {\Psi'}^2) f^{\epsilon}f_*^{\epsilon} \frac{\pi^3}{\epsilon^3}\beta \left(
		\frac{\pi\theta}{\epsilon}
		\right)dp d\theta dv_* dv dt
		\right)^\frac{1}{2}  \\
		&\qquad= \sqrt{2\pi}C\left(
		\int_0^1 \iiRs \int_0^{\pi/2} \pi^2|v-v_*|^{2\tilde{q} + \gamma}\chi^2(\Psi^2 + {\Psi'}^2) f^{\epsilon}f_*^{\epsilon} \beta(\chi) d\chi dv_* dv dt
		\right)^\frac{1}{2}
	\end{align*}
	The final expression is uniformly bounded in $\epsilon$ by the assumptions (notice $2\tilde{q} + \gamma \in [0,2]$) and finite angular momentum transfer~\eqref{eq:betaint}.
\end{proof}
	\begin{corollary}
		Consider the setting of~\Cref{prop:limm} and denote $\mathscr{M}\in \mathcal{M}((0,1)\times \R^6\times \Stwo)$ a limit of the family $\{|v-v_*|^{\tilde{q}}\theta M^\epsilon\}_\epsilon$. Then, $\mathscr{M}$ can be expressed as a family in time of signed measures on $\R^6\times \Stwo$.
	\end{corollary}
	\begin{proof}
		We repeat the proof of~\Cref{prop:limm} but fix a test function $\Psi(v, \, v_*, \, \sigma, \, t)\in C_0((0,1)\times \R^6\times \Stwo)$ now so that its time dependence is an indicator function, i.e.
		\[
		\Psi(v, \, v_*, \, \sigma, \, t) = \psi(v, \, v_*, \, \sigma) \chi_{[a,b]}(t), \quad \psi \in C_0(\R^6\times \Stwo).
		\]
		We continue from the last line of the previous proof to obtain
		\begin{align*}
			&\quad \int_a^b \iiRs \iS |\psi| |v-v_*|^{\tilde{q}}|\theta||M_t^\epsilon|d\sigma dv_* dv dt \\
			&\qquad \le \sqrt{2\pi} C \left(
			\int_a^b \iiRs \int_0^{\pi/2} \pi^2|v-v_*|^{2\tilde{q} + \gamma}\chi^2(\psi^2 + {\psi'}^2) f^{\epsilon}f_*^{\epsilon} \beta(\chi) d\chi dv_* dv dt
			\right)^\frac{1}{2}.
		\end{align*}
		The finite angular momentum transfer~\eqref{eq:betaint} is independent of time, moreover the zeroeth to second moments of $f^\epsilon$ are bounded uniformly in $\epsilon$ and $t$ from~\eqref{eq:bounds} so absorbing these terms into a constant leaves
		\[
		\int_a^b \iiRs \iS |\psi| |v-v_*|^{\tilde{q}}|\theta||M_t^\epsilon|d\sigma dv_* dv dt \le C |b-a|^\frac{1}{2}.
		\]
		This estimate holds in the limit $\epsilon\downarrow 0$ as well, so the measure $\mathscr{M}$ can also be disintegrated with respect to Lebesgue measure on $t\in[0,1].$
	\end{proof}
	From now on, we take for granted that the limits in~\Cref{prop:limm} are also families in time of signed measures on $\R^6\times \Stwo$.
\begin{lemma}[Comparison of certain topologies against the Boltzmann metric]
	\label{lem:lowbdd}
	Let $\mu_0,\mu_1\in \mathscr{P}_2(\Rthree)$ be probability measures that are absolutely continuous with respect to Lebesgue. There exists a constant $C>0$ depending only on the finite angular momentum transfer~\eqref{eq:EDE} and the second moment of $\mu_0$ such that the following holds:
	\begin{enumerate}
		\item In the case $\gamma\in [-2,0]$, we have
		\begin{align*}
			&\left|
			\int_{\R^3}\psi(v) d(\mu_0 - \mu_1)(v)
			\right|\le C d_B^\epsilon(\mu_0,\mu_1)
		\end{align*}
		for any function $\psi$ with Lipschitz semi-norm bounded by 1.
		\item In the case $\gamma \in [-4,-2)$, we have
		\begin{align*}
			&\left|
			\int_{\R^3}\psi(v) d(\mu_0 - \mu_1)(v)
			\right| \le C d_B^\epsilon(\mu_0,\mu_1)
		\end{align*}
		for any function $\psi$ with second derivative bounded by 1.
	\end{enumerate}
\end{lemma}
\begin{proof}
	We will only show the proof of the first estimate in the case $\gamma \in [-2,0]$. The proof of the second estimate differs only by estimating $\bn \psi$ using the second estimate of~\Cref{lem:estbnv2} instead of the first. Without loss of generality we can assume that $d_B^\epsilon(\mu_0,\mu_1)<\infty$. We take $M$ the optimal collision rate in the sense of~\Cref{lem:Erbarmetvel}. Fix Lipschitz $\psi: \Rthree \to \R$ with $\text{Lip}(\psi)\le 1$ and we recall the first estimate of~\Cref{lem:estbnv2}
	\begin{align*}
		|\bn \psi| &\le \left|\psi\left(v+ \frac{1}{2}|v-v_*|(\sigma-k)\right) - \psi(v)\right| + \left|\psi\left(v_* - \frac{1}{2}|v-v_*|(\sigma - k)\right) - \psi(v_*)\right| \\
		&\le |v-v_*||\sigma - k|.
	\end{align*}
	Using $\psi$ as a test function in the collision rate equation (which we justify at the end) connecting $\mu_0$ to $\mu_1$ we have
	\begin{align*}
		\left|
		\iR \psi d\mu_0 \right.&\left.- \iR \psi d\mu_1
		\right| = \frac{1}{4}
		\left|
		\int_0^1 \iiRs \iS \bn \psi M d\sigma dv_* dv dt
		\right| \le \frac{1}{4}\int_0^1 \iiRs \iS |v-v_*||\sigma - k| |M| d\sigma dv_* dv dt \\
		&\le \frac{1}{4}\left(
		\int_0^1\iiRs \iS \frac{|M|^2}{\Lambda(f)B^\epsilon}d\sigma dv_* dv dt
		\right)^\frac{1}{2} \left(
		\int_0^1 \iiRs \iS |v-v_*|^2|\sigma - k|^2 \Lambda(f) B^\epsilon d\sigma dv_* dv dt
		\right)^\frac{1}{2}.
	\end{align*}
	At this point, we recognize the first term as the time integrated $\epsilon$-Boltzmann action. In the second term, we can apply~\Cref{cor:lambdaf}. Since $M$ is optimal we can estimate the previous expression by
	\begin{align*}
		\frac{1}{4}d_B^\epsilon(\mu_0,\mu_1)\left(
		\int_0^1 \iiRs \iS |v-v_*|^2|\sigma - k|^2 \left(
		\frac{f^{'}f_*^{'} + f f_*}{2}
		\right)B^\epsilon d\sigma dv_* dv dt
		\right)^\frac{1}{2}.
	\end{align*}
	By~\Cref{fig:pic1} or directly from the definitions, we have
	$
	|\sigma' - k'| = |k - \sigma|.
	$
	Hence, the arithmetic mean is just $ff_*$ upon symmetrisation ($B^\epsilon$ is invariant when $(v,v_*)\leftrightarrow (v',v_*')$). This leads to
	\begin{align*}
		&\quad \left|\iR \psi d\mu_0 - \iR \psi d\mu_1
		\right|\le \frac{1}{4}d_B^\epsilon(\mu_0,\mu_1) \left(
		\int_0^1 \iiRs \iS |v-v_*|^2|\sigma - k|^2 B^\epsilon f f_* d\sigma dv_* dv dt
		\right)^\frac{1}{2}.
	\end{align*}
	Now, we change representation from $\sigma$-representation to $(\theta,\phi)$-representation, see \Cref{sec:sigmarep}. We recall from~\eqref{eq:equivang} that
	$
	|\sigma - k|^2 = 2(1-\cos\theta).
	$
	Substituting this leads to the further estimate
	\begin{align*}
		\left|\iR \psi d\mu_0 \right.&\left. - \iR \psi d\mu_1
		\right|\le \frac{\sqrt{2}}{4}d_B^\epsilon(\mu_0,\mu_1) \left(
		\int_0^1 \!\!\iiRs \!\int_{\theta=0}^{\epsilon/2}\int_{\phi = 0}^{2\pi}\!\!(1-\cos\theta)|v-v_*|^{2+\gamma}\beta^\epsilon(\theta) f f_* d\theta d\phi dv_* dv dt 
		\right)^\frac{1}{2} \\
		&= \frac{\sqrt{2}}{4}d_B^\epsilon(\mu_0,\mu_1)\left(
		\int_0^1 \!\!\iiRs \!\int_{\theta=0}^{\epsilon/2}\int_{\phi = 0}^{2\pi}\!\!(1-\cos\theta)|v-v_*|^{2+\gamma} \frac{\pi^3}{\epsilon^3}\beta\left(\frac{\pi\theta}{\epsilon}\right) f f_* d\theta d\phi dv_* dv dt 
		\right)^\frac{1}{2}. 
	\end{align*}
	We perform the change of variables $\chi = \pi\theta/\epsilon$ and directly compute the $\phi$ integral to get
	\begin{align*}
		\left|\iR \psi d\mu_0 - \iR \psi d\mu_1
		\right| &\le \frac{\sqrt{\pi}}{2}d_B^\epsilon(\mu_0,\mu_1) \left(
		\int_0^1 \iiRs \int_{\chi=0}^{\pi/2}(1 - \cos \frac{\epsilon\chi}{\pi})|v-v_*|^{2+\gamma} \frac{\pi^2}{\epsilon^2}\beta(\chi) f f_* d\chi dv_* dv dt 
		\right)^\frac{1}{2}.
	\end{align*}
	We eliminate the factor of $1/\epsilon^2$ by the inequality $1 - \cos x \le \frac{1}{2}x^2$ in~\eqref{eq:boundcos} when $x\in [0,\pi]$ to give
	\begin{align*}
		\left|\iR \psi d\mu_0 - \iR \psi d\mu_1
		\right|\le \frac{\sqrt{2\pi}}{4}d_B^\epsilon(\mu_0,\mu_1) \left(
		\int_0^1 \iiRs \int_{\chi = 0}^{\pi/2} \chi^2 |v-v_*|^{2+\gamma} \beta(\chi) f f_* d\chi dv_* dv dt
		\right)^\frac{1}{2}.
	\end{align*}
	The integral decouples and the proof is complete recalling the finite angular momentum transfer~\eqref{eq:betaint}.
	
	We now address the use of time-independent Lipschitz-bounded functions as test functions in the collision rate equation for $\gamma\in[-2,0]$. Fix $\phi \in C_c^\infty((0,1)\times \R^3)$ and repeat the previous estimates. In particular by the first estimate of~\Cref{lem:estbnv2}, notice that the drift term can be estimated as follows
		\begin{align*}
			&\quad \left|
			\int_0^1\iiRs \iS \bn \phi M
			\right| \le \int_0^1\text{Lip}_v(\phi)\iiRs \iS |v-v_*||\sigma - k| |M| \\
			&\le \frac{1}{2}\int_0^1 \text{Lip}_v(\phi) \left\{
			\iiRs \iS \frac{|M|^2}{\Lambda(f)B^\epsilon} + \iiRs \iS |v-v_*|^2 |\sigma - k|^2 \Lambda(f) B^\epsilon
			\right\}.
		\end{align*}
		The finite metric derivative and second moment bounds from~\eqref{eq:bounds} and the finite angular momentum~\eqref{eq:betaint} give estimates on both the integrals in the curly brackets. This leaves only the dependence on the Lipschitz semi-norm of the test function so by density, one can take test functions with bounded Lipschitz dependence in $v\in \R^3$. The argument is similar for $\gamma\in[-4,-2)$ by using instead the second estimate of~\Cref{lem:estbnv2} to enlarge the class of test functions to those with bounded second derivative. The time independence can be treated by considering test functions $\phi_k(t,v) = \eta_k(t)\zeta(v)$ for $\eta_k \in C_c^\infty((0,1))$ and $\zeta\in C_c^\infty(\R^3)$ where $\eta_k$ is a smooth approximation of the indicator function on (0,1) for $k\in \mathbb{N}$. A detailed argument for this point can be found, for instance in~\cite[Lemma 8.1.2]{AGS08}.
\end{proof}
\subsection{Grazing collision limit of the continuity equations}
\label{sec:lanlim}
Now that we understand compactness for $(f^\epsilon, M^\epsilon)$, we need to verify that the limit points actually satisfy the $GCE$. Given $\psi\in C_c^\infty((0,T)\times\R^3)$ a test function, we recall the generalized continuity equations for the $\epsilon$-Boltzmann ($(f^\epsilon, M^\epsilon) \in CRE$) and Landau ($(f,M) \in GCE$) equations
\begin{equation}
	\label{eq:boltzcont}
	\int_0^T \iR \partial_t \psi  f^\epsilon dv dt + \frac{1}{4}\int_0^T \iiRs \iS \bn \psi M^\epsilon d\sigma dv_* dv dt = 0,
\end{equation}
and
\begin{equation}
	\label{eq:lancont}
	\int_0^T \iR \partial_t\psi f dv dt + \frac{1}{2} \int_0^T \iiRs \tn \psi \cdot M dv_* dv dt = 0.
\end{equation}
We can directly compare the first terms from the weak convergence $f^\epsilon \to f$ from~\Cref{cor:limcurve}. It remains to understand precisely the convergence in the transport term. Recalling~\Cref{prop:limm}, we know that $\left(|v-v_*|^q \left(1 + \left[|v|^2 + |v_*|^2 \right]^\frac{\delta}{2} \right)\theta M^\epsilon\right)_{\epsilon>0}$ is compact for $q, \delta$ satisfying
\[
q = \left\{
\begin{array}{cl}
	1,     &\gamma\in[-2,0)  \\
	2,     &\gamma\in[-4,-2) 
\end{array}
\right., \quad 0 < \delta < \left\{
\begin{array}{ll}
	-\frac{\gamma}{2}     &\gamma\in[-2,0)  \\
	-\frac{\gamma}{2} - 1,     &\gamma\in[-4,-2) 
\end{array}
\right..
\]
We define a `lift' mapping whose use will soon be clear
\[
L_{q,\delta} : M \in \mathcal{M}_B \mapsto \frac{|v-v_*|^{-\frac{\gamma}{2}-q}}{4\left(1 + \left[|v|^2 + |v_*|^2 \right]^\frac{\delta}{2}\right)} \iS M \, p \, d\sigma \in \mathcal{M}_L.
\]
Recall $p \in \Sk$ as defined in~\Cref{sec:spherical} in the integral over $\Stwo$. The motivation for this comes from looking at the formal grazing collision limit of~\eqref{eq:boltzcont}. Along the subsequence of convergence in~\Cref{prop:limm}, let us write
\[
|v-v_*|^{q}\left(1 + \left[|v|^2 + |v_*|^2 \right]^\frac{\delta}{2} \right)\theta M^\epsilon \overset{\epsilon \downarrow 0}{\to} M_{q,\delta}^0.
\]
This suggests to multiply and divide by $\theta$ within the integral in~\eqref{eq:boltzcont}. By~\Cref{lem:estbnv2}, we obtain
\[
\frac{1}{\theta}\bn \psi \overset{\epsilon\downarrow 0}{\to} \frac{1}{2}|v-v_*| (\nabla \psi - \nabla_* \psi_*)\cdot p = \frac{1}{2}|v-v_*|^{-\frac{\gamma}{2}}p\cdot \tn \psi.
\]
Multiplying and dividing by $|v-v_*|^{q}\left(1 + \left[|v|^2 + |v_*|^2 \right]^\frac{\delta}{2} \right)\theta$ in the transport term of~\eqref{eq:boltzcont} and omitting the time integral, we have
\begin{align*}
	&\frac{1}{4}\iiRs \iS \bn \psi M^\epsilon = \frac{1}{4} \iiRs \iS \frac{|v-v_*|^{q}\left(1 + \left[|v|^2 + |v_*|^2 \right]^\frac{\delta}{2} \right)\theta M^\epsilon}{\theta|v-v_*|^{q}\left(1 + \left[|v|^2 + |v_*|^2 \right]^\frac{\delta}{2} \right)} \bn \psi  d\sigma dv_* dv.
\end{align*}
In order to apply weak-strong convergence, we need to ensure that $|v-v_*|^{-q}\left(1 + \left[|v|^2 + |v_*|^2 \right]^\frac{\delta}{2} \right)^{-1}\frac{\bn \psi}{\theta}$ decays when $v, \, v_* \to \infty$ uniformly in $\epsilon>0$. Recall the meaning of weak-strong convergence; if $\psi_n$ converges strongly to $\psi$ in $C_0(\R^6\times \Stwo)$ and the sequence of signed Radon measures $M_n$ converges weakly-* to $M$ in $\mathcal{M}(\R^6\times \Stwo)$, then $\iiRs \iS \psi_n dM_n \to \iiRs \iS \psi dM$. By our choice of $q$ depending on $\gamma\in[-4,0)$ and~\Cref{lem:estbnv2}, we can estimate
\[
|v-v_*|^{-q}\frac{\bn \psi}{\theta} \lesssim \left\{
\begin{array}{cl}
	\text{Lip}(\psi)     &q=1, \, \gamma\in[-2,0)  \\
	\|D^2\psi\|_{L^\infty}     &q = 2, \, \gamma\in [-4,-2) 
\end{array}
\right..
\]
By the convergence result in~\Cref{lem:epszerobnpsiv2}, we obtain
\[
|v-v_*|^{-q}\frac{\bn \psi}{\theta} \to \frac{1}{2}|v-v_*|^{1-q}p\cdot (\nabla - \nabla_*)\psi = \frac{1}{2}|v-v_*|^{-q-\frac{\gamma}{2}}p\cdot \tn \psi.
\]
Therefore, we can pass to the limit $\epsilon\downarrow 0$,
\begin{align*}
	\frac{1}{4}\iiRs \iS \bn \psi M^\epsilon &\overset{\epsilon\downarrow 0}{\to} \frac{1}{8}\iiRs \frac{|v-v_*|^{-q-\frac{\gamma}{2}}}{\left(1 + \left[|v|^2 + |v_*|^2 \right]^\frac{\delta}{2} \right)}\tn \psi \cdot \iS M_{q,\delta}^0 p d\sigma dv_* dv \\
	&= \frac{1}{2}\iiRs \tn \psi \cdot L_{q,\delta}(M_{q,\delta}^0)dv_* dv.
\end{align*}
This shows that $(f,L_{q,\delta}(M_{q,\delta}^0))$ is an admissible pair for the grazing continuity equation coming from the Boltzmann sequence $(f^\epsilon, M^\epsilon)$.

\begin{remark}
	\label{rem:totalcrosssection}
	We describe more precisely the various changes necessary to generalize our assumptions on the collision kernel from the discussion in~\Cref{rem:uppbdd}. Repeating the proof of~\Cref{prop:limm}, we can show that the family
	\[
	\left\{
	\frac{\theta M^\epsilon}{\sqrt{T^\epsilon(v,v_*)}}
	\right\}_{\epsilon>0}
	\]
	has uniformly bounded moments up to first order. To prove a similar result as~\Cref{lem:lowbdd}, one has to utilize both the Lipschitz (for large $|v-v_*| \gg 1$) and Hessian (for local $|v-v_*|\le 1$) estimates in~\Cref{lem:estbnv2} to obtain
	\[
	\left|
	\int_{\R^3}\psi(v) d(\mu_0 - \mu_1)(v)
	\right|\lesssim d_B^\epsilon(\mu_0,\mu_1),
	\]
	for test functions $\psi$ with Lipschitz semi-norm and second derivative bounded by 1.
	
	Finally, we discuss the grazing collision limit at the level of the generalized continuity equations. We need good estimates for the pairing of $\bn \psi$ against $M^\epsilon$ to show the grazing collision limit: 
	\[
	\int_0^1 \iiRs \iS \bn \psi M^\epsilon =\int_0^1 \iiRs \iS \frac{\bn \psi}{\theta} \sqrt{T^\epsilon} \left[\frac{\theta M^\epsilon}{\sqrt{T^\epsilon}} \right]\to \int_0^1\iiRs \tn \psi(v,v_*)\cdot \iS M^0p\;d\sigma dv_*dv,
	\]
	where
	\[
	\tn \psi(v,v_*) = |v-v_*|\sqrt{T(|v-v_*|)}\Pi[v-v_*](\nabla \psi - \nabla_* \psi_*).
	\]
	The convergence above uses again the Lipschitz and Hessian estimates in~\Cref{lem:estbnv2}, the uniform bound for $T^\epsilon$ \eqref{Tepsbound}, and the convergence of  $T^\epsilon$ to $T$ as in \eqref{Tepsconvergence}.
\end{remark}

\section{Lower semicontinuous convergence of the dissipations in the grazing collision limit}
\label{sec:gammadiss}
Throughout this section we consider $(f^\epsilon)_{\epsilon>0}$ a sequence of probability densities with uniformly bounded second moment and entropy
\begin{equation}\label{ass:gammadiss}
	\sup_{\epsilon>0}\int_{\R^3} |v|^2f^\epsilon(v)dv < + \infty, \quad \sup_{\epsilon>0}\int_{\R^3} f^\epsilon \log f^\epsilon dv < + \infty.    
\end{equation}
such that $f^\epsilon \overset{\sigma}{\rightharpoonup} f$ for some probability density $f$. We wish to show the lower semicontinuous convergence of the dissipation.
\begin{proposition}
	\label{prop:gammadiss}
	Assume $f^\epsilon \overset{\sigma}{\rightharpoonup} f$ with uniform second moment and entropy bounds \eqref{ass:gammadiss}. Then we have
	\[
	\liminf_{\epsilon\downarrow 0}D_B^\epsilon(f^\epsilon) \ge D_L(f).
	\]
\end{proposition}

\begin{proof} 
	We first reduce to the case
	\[
	\sup_{\epsilon>0} D_B^\epsilon(f^\epsilon) < + \infty.
	\]
	Indeed, without loss of generality, we may assume that $\liminf_{\epsilon\downarrow 0} D_B^\epsilon(f^\epsilon) < + \infty$ otherwise there is nothing to show. There is a subsequence $(\epsilon_n)_{n\in\mathbb{N}}$ such that $\epsilon_n \downarrow 0$ as $n\to \infty$ for which
		\[
		\sup_{n\in\mathbb{N}}D_B^{\epsilon_n}(f^{\epsilon_n})\le \liminf_{\epsilon\downarrow 0}D_B^\epsilon(f^\epsilon)+1 < + \infty.
		\]
	In this uniformly bounded dissipation setting, we can further say $\sqrt{f^\epsilon} \to \sqrt{f}$ in $L_{loc}^2$ by~\Cref{prop:sqrtcpct}.
	
	Collecting the results of~\Cref{sec:dissrep}, we obtain the following estimates for the dissipation
		\begin{align*}
			D_B^\epsilon(f^\epsilon) &\ge \sup_{\psi \in DS_c^\infty}\left\{
			-2 \iiRs \sqrt{f^\epsilon f_*^\epsilon} \left(\iS\bn \psi B^\epsilon\right) - \frac{1}{4}\iiRs \iS |\bn \psi|^2 B^\epsilon
			\right\},   \\
			D_L(f) &\le \sup_{\psi\in DS_c^\infty} \left\{
			- 4\iiRs \sqrt{ff_*} \tn \cdot \tn \psi - 2\iiRs |\tn \psi|^2
			\right\},
		\end{align*}
		where the test functions $\psi$ belong to the following class of smooth functions
		\[
		DS_c^\infty := \left\{
		\psi = \psi(v,v_*) \in C_c^\infty(\R^6;\R) \, : \, \left.
		\begin{array}{l}
			\psi(v,v_*) = \psi(v_*,v) \, \forall v,v_* \in \R^3, \\
			\exists \delta_\psi>0 \text{ s.t. } \psi(v,v_*) = 0 \,\, \forall |v-v_*| \le \delta_\psi
		\end{array}
		\right.
		\right\}.
		\]
	Up to some constants (consistent in both expressions below), we apply the results of~\Cref{sec:convergediss} and~\Cref{lem:epszerobnpsiv2,lem:epszerobnpsiavgv2} which state
	\begin{align*}
		\iiRs \sqrt{f^\epsilon f_*^\epsilon} \left(\iS\bn \psi B^\epsilon\right) \overset{\epsilon\downarrow 0}{\to} 2\iiRs \sqrt{ff_*} \tn \cdot \tn \psi \qquad \mbox{and} \qquad
		\iiRs \iS |\bn \psi|^2 B^\epsilon \overset{\epsilon\downarrow 0}{\to} 8\iiRs |\tn \psi|^2,
	\end{align*}
	for any fixed $\psi \in DS_c^\infty$. 
\end{proof}
\subsection{Affine representation of dissipations}
\label{sec:dissrep}
In this subsection, we seek to show
	\begin{align}
		\label{eq:dissbdd}
		\begin{split}
			D_B^\epsilon(f^\epsilon) &\ge \sup_{\psi \in DS_c^\infty}\left\{
			-2 \iiRs \sqrt{f^\epsilon f_*^\epsilon} \left(\iS\bn \psi B^\epsilon\right) - \frac{1}{4}\iiRs \iS |\bn \psi|^2 B^\epsilon
			\right\},   \\
			D_L(f) &\le \sup_{\psi\in DS_c^\infty} \left\{
			- 4\iiRs \iS \sqrt{ff_*} \tn \cdot \tn \psi - 2\iiRs |\tn \psi|^2
			\right\},
		\end{split}
	\end{align}
	where $DS_c^\infty$ was introduced in the proof of~\Cref{prop:gammadiss}. For ease of notation, we will drop the $\epsilon$ superscripts for the Boltzmann dissipation. Recall, the Boltzmann and Landau dissipations can be written as
\begin{align*}
	D_B(f) &= \frac{1}{4}\iiRs \iS [f'f_*' - ff_*]\log \frac{f'f_*'}{ff_*}B, \\
	D_L(f) &= 2\iiRs |v-v_*|^{2+\gamma}|\Pi[v-v_*](\nabla - \nabla_*)\sqrt{ff_*}|^2.
\end{align*}
The proof of the upper bound for the Landau dissipation in~\eqref{eq:dissbdd} is the content of the subsequent~\Cref{sec:landiss,sec:projtn} while the lower bound for the Boltzmann dissipation in~\eqref{eq:dissbdd} is the content of~\Cref{sec:boltzdiss}.
\subsubsection{Landau dissipation}
\label{sec:landiss}
We begin with a preliminary expression of the Landau dissipation.
\begin{proposition}
	\label{prop:lansqrt}
	We can express
	\begin{equation}
		\label{eq:sqrtlan}
		D_L(f) = \sup_{\xi \in C_c^\infty(\R^6;\R^3)}\left\{
		-4\iiRs \sqrt{ff_*} |v-v_*|^{1+\frac{\gamma}{2}}(\nabla - \nabla_*) \cdot (\Pi[v-v_*]\xi) - 2 \iiRs |\xi|^2
		\right\}.
	\end{equation}
\end{proposition}
\begin{proof}
	Let us denote the right-hand side of~\eqref{eq:sqrtlan} by $I_L(f)$; we want to show equality $D_L(f) = I_L(f)$.
	
	Showing $I_L(f) \le D_L(f)$ is straight-forward because if $D_L(f)< + \infty$, we can integrate by parts the differential operator $|v-v_*|^{1+\frac{\gamma}{2}}(\nabla-\nabla_*)\cdot \Pi[v-v_*]$ onto $\sqrt{ff_*}$ (since finite dissipation implies $\tn \sqrt{ff_*} \in L^2$) and then apply Cauchy-Schwarz and Young's inequality in the following way for fixed $\xi\in C_c^\infty(\R^6;\R^3)$
	\begin{align*}
		&-4\iiRs \sqrt{ff_*} |v-v_*|^{1+\frac{\gamma}{2}}(\nabla - \nabla_*) \cdot (\Pi[v-v_*]\xi) - 2 \iiRs |\xi|^2\\
		&\qquad\qquad= 4\iiRs |v-v_*|^{1+\frac{\gamma}{2}}\Pi[v-v_*]\left[(\nabla-\nabla_*)\sqrt{ff_*}\right]\cdot \xi - 2\iiRs |\xi|^2 \\
		&\qquad\qquad\le 4\left(
		\iiRs |v-v_*|^{2+\gamma}|\Pi[v-v_*](\nabla-\nabla_*)\sqrt{ff_*}|^2
		\right)^\frac{1}{2}\left(
		\iiRs |\xi|^2
		\right)^\frac{1}{2} -2 \iiRs |\xi|^2 \le D_L(f).
	\end{align*}
	
	Turning to the other direction, we wish to show that $D_L(f) \le I_L(f)$. We assume here that $I_L(f) < +\infty$ or else there is nothing to show. Define the linear operator acting on $\xi \in C_c^\infty(\R^6; \R^3)$ given by
	\[
	F : \xi \mapsto -4 \iiRs \sqrt{ff_*} |v-v_*|^{1+\frac{\gamma}{2}}(\nabla - \nabla_*)\cdot (\Pi[v-v_*]\xi).
	\]
	Since we assume $I_L(f) < +\infty$, we have
	\begin{align*}
		\sup_{\xi \in C_c^\infty, \|\xi\|_{L^2} = 1} \left\{\langle F, \xi\rangle - 2\right\} \le C < + \infty,
	\end{align*}
	So by density, $F$ extends uniquely to a bounded linear operator on $L^2(\R^6;\R^3)$. We consider now the continuous, coercive, symmetric, and bilinear form
	\[
	a(\xi,\eta) = 2 \iiRs \xi \cdot \eta, \quad \forall \xi,\eta \in L^2(\R^6;\R^3).
	\]
	By Lax-Milgram/Riesz representation, there is a unique Riesz representative $u\in L^2(\R^6;\R^3)$ such that
	\[
	2\iiRs u\cdot \xi = \langle F, \xi\rangle, \quad \forall \xi \in L^2(\R^6;\R^3).
	\]
	Moreover, $u = \tn \sqrt{ff_*}$ in $L^2$ and this characterizes the minimisation problem
	\[
	\frac{1}{2}a(u,u) - \langle F, u\rangle = \min_{\xi\in L^2}\left\{
	\frac{1}{2}a(\xi,\xi) - \langle F, \xi\rangle
	\right\}.
	\]
	Using the definitions of $a,\,\phi, $ and $u = \tn \sqrt{ff_*}$ as well as density, the above gives
	\[
	-2 \iiRs |\tn \sqrt{ff_*}|^2 = \inf_{\xi \in C_c^\infty}\left\{
	2 \iiRs |\xi|^2 + 8 \iiRs \sqrt{ff_*}|v-v_*|^{1+\frac{\gamma}{2}}(\nabla - \nabla_*)\cdot (\Pi[v-v_*]\xi)
	\right\}.
	\]
	Applying one sign change gives precisely
	$
	D_L(f) = I_L(f).
	$
\end{proof}
Intuitively, a near optimal $\xi \in C_c^\infty(\R^6;\R^3)$ in the right-hand side of~\eqref{eq:sqrtlan} needs to approximate $\tn \sqrt{ff_*}$. We have formulated~\eqref{eq:sqrtlan} in a big space of test functions without taking advantage of the anti-symmetry of variable swapping $v \leftrightarrow v_*$. To take advantage of symmetries, we define
\[
AS = \{V = V(v,v_*) \in L^2(\R^6;\R^3) \, | \, V(v,v_*) = -V(v_*,v) \text{ a.e. }v,v_*\in\R^3\},
\]
together with the smooth and compactly supported approximations
\[
AS_c^\infty = \left\{V = V(v,v_*) \in C_c^\infty(\R^6;\R^3) \, \left| \begin{array}{l}
	V(v,v_*) = - V(v_*,v) \, \forall v,\,v_* \in \R^3, \\
	\exists \delta >0 \text{ s.t. } V(v,v_*) = 0 \,\, \forall |v-v_*| \le \delta
\end{array}\right.\right\}.
\]
The anti-symmetry allows to write the following identity
\begin{equation}
	\label{eq:antisymm}
	V(v,v_*) = \frac{1}{2}(V(v,v_*) - V(v_*,v)), \quad \forall V \in AS.
\end{equation}
To shorten some notation, we will write
\[
\as{V}(v,v_*) := V(v_*,v), \quad \forall v,v_*\in\R^3.
\]
Using~\eqref{eq:antisymm}, it is easy to see that $AS$ is a closed subspace of $L^2(\R^6;\R^3)$ and hence is a Hilbert space with the $L^2$ inner product. Moreover, we have density of the smooth compactly supported approximations.
\begin{lemma}
	\label{lem:ASdense}
	$AS_c^\infty$ is dense in $AS$ with respect to the $L^2$ topology.
\end{lemma}
We skip the proof of \Cref{lem:ASdense}, as it follows from standard well-known arguments. Using density we can show the next characterization.
\begin{proposition}
	\label{prop:sqrtlanantisymm}
	The Landau dissipation can be written as
	\begin{equation}
		\label{eq:sqrtlanantisymm}
		D_L(f) = \sup_{V \in AS_c^\infty}\left\{
		-4 \iiRs \sqrt{ff_*} |v-v_*|^{1+\frac{\gamma}{2}}(\nabla - \nabla_*)\cdot (\Pi[v-v_*]V) -2 \iiRs |\Pi[v-v_*]V|^2
		\right\}.
	\end{equation}
\end{proposition}
\begin{proof}
	Replace $L^2(\R^6;\R^3)$ and $C_c^\infty(\R^6;\R^3)$ with $AS$ and $AS_c^\infty$, respectively and follow the same proof of~\Cref{prop:lansqrt} taking advantage of~\eqref{eq:antisymm}. This would lead to the majorant
	\[
	\sup_{V \in AS_c^\infty}\left\{
	-4 \iiRs \sqrt{ff_*} |v-v_*|^{1+\frac{\gamma}{2}}(\nabla - \nabla_*)\cdot (\Pi[v-v_*]V) -2 \iiRs |V|^2
	\right\}.
	\]
	Since $\Pi$ is a projection, we have $|\Pi[v-v_*]V| \le |V|$. Estimating  the second term of the affine representation in this way completes the proof.
\end{proof}

\subsubsection{Projecting a vector field onto the image of $\tn$}
\label{sec:projtn}
Starting from~\eqref{eq:sqrtlanantisymm}, our goal now is to replace the vector field $V \in AS_c^\infty$ by $\tn\psi$ for some $\psi\in DS_c^\infty$. In this section, the role of $DS_c^\infty$ which we introduced in the proof of~\Cref{prop:gammadiss} will be clarified. Given any $V \in AS_c^\infty$, we find $\psi$ such that
\begin{align*}
	&-4 \iiRs \sqrt{ff_*} |v-v_*|^{1+\frac{\gamma}{2}}(\nabla - \nabla_*)\cdot (\Pi[v-v_*] V)-2 \iiRs |\Pi[v-v_*]V|^2\\
	&\qquad\le -4 \iiRs \sqrt{ff_*} \tn\cdot\tn \psi-2 \iiRs |\tn \psi|^2.
\end{align*}
Said $\psi$ can be characterized by the projection of $V$ (or equivalently $\Pi[v-v_*]V)$ onto the image of $\tn$. More explicitly, $\psi$ will be obtained as the solution to the following minimization problem
\begin{equation}
	\label{eq:projtn}
	\min_{g\in DS_c^\infty}\iiRs |\tn g - \Pi[v-v_*]V|^2 = \iiRs |\tn \psi - \Pi[v-v_*]V|^2.
\end{equation}
We begin by investigating the solvability of the first order condition of this convex minimization problem which is given by the PDE in the following result.
	\begin{lemma}
		\label{lem:lemlaptn}
		For $V \in AS_c^\infty$, there exists a unique smooth solution $\psi\in DS_c^\infty$ to the following equation
		\begin{equation}
			\label{eq:tnPoisson}
			\tn \cdot \tn \psi = \tn \cdot (\Pi[v-v_*]V).
		\end{equation}
		Moreover, $\psi$ solves~\eqref{eq:projtn} and we have
		\begin{equation}
			\label{eq:L2decomp}
			\|\Pi[v-v_*]V\|^2 = \|\tn \psi\|^2 + \|\Pi[v-v_*]V - \tn \psi\|^2.
		\end{equation}
	\end{lemma}
	\begin{proof}
		After changing variables, we will construct $\psi$ as a superposition of solutions to the Laplace-Beltrami equation on spheres $|v-v_*|=r$.  Recalling some of the notations from~\Cref{sec:formalgc}, for fixed $v,v_*\in \R^3$ we consider the smooth and volume preserving coordinate transformation
		\[
		(v,v_*) \mapsto \left(
		\frac{v-v_*}{2}, \frac{v+v_*}{2}
		\right) =: (x,y).
		\]
		Given vector fields $V = V(v,v_*)$ and scalar functions $\psi = \psi(v,v_*)$, we will use the same symbols to denote their versions under this and future coordinate transformations $V = V(v,v_*) = V(x,y)$ and similarly for $\psi$. It is readily checked that $\nabla - \nabla_* = \nabla_x$ and similarly for the divergence. Notice that when $V \in AS_c^\infty$,~\eqref{eq:tnPoisson} reads
		\[
		|2x|^{2+\gamma}\nabla_x \cdot (\Pi[x]\nabla_x\psi(x,y)) = |2x|^{1+\frac{\gamma}{2}}\nabla_x\cdot (\Pi[x]V(x,y)),
		\]
		which is an equation in the $x = \frac{1}{2}(v-v_*)$ variable only. Henceforth, we consider fixed $y = \frac{1}{2}(v+v_*)$ as a parameter to the problem above. To further specify the problem, the compact support and anti-symmetry of $V \in AS_c^\infty$ translate into compactness in both the $x,y$ variables, and moreover $V$ vanishes in a neighbourhood of $\{x = 0\}$. So for some $0 < \delta \le R$ depending on the support of $V$, but uniform in the $y = \frac{1}{2}(v+v_*)$ variable, we consider the following elliptic PDE with homogeneous Dirichlet boundary conditions
		\begin{equation}
			\label{eq:xpde}
			\left\{
			\begin{array}{ll}
				|2x|^{1+\frac{\gamma}{2}}\nabla_x \cdot (\Pi[x]\nabla_x\psi(x,y)) = \nabla_x\cdot (\Pi[x]V(x,y))     &0 < \delta \le |2x| \le R  \\
				\psi(x,y) = 0     &|2x| \in \{\delta,R\}
			\end{array}
			\right..
		\end{equation}
		The weight $|2x|^{1+\frac{\gamma}{2}}$ on the left-hand side is well-behaved, since we avoid a neighbourhood of the singularity ${x=0}$. To reiterate, we will solve \eqref{eq:xpde} in $x$ for fixing the value of $y$ as a parameter. Since the dependence on $y$ of $V$ is smooth, it will follow any solution $\psi$ of~\eqref{eq:xpde} is also smooth in $y$. In terms of solvability of~\eqref{eq:xpde}, we make one further change of variables. Having fixed $y$ as a parameter, we consider the spherical decomposition of 
		\[
		x = rk, \quad r = |x|, \quad k = \frac{x}{|x|} \in \Stwo.
		\]
		Under these coordinates, we again identify $\psi = \psi(x,y) = \psi(k,r,y)$ and similarly for $V$. By the identities of~\Cref{lem:LB} and Corollary~\ref{cor:deromega}, we can consider $r\in [\delta,R]$ as another parameter so that~\eqref{eq:xpde} becomes an equation in the spherical variable $k$
		\begin{equation}
			\label{eq:kpde}
			\left\{
			\begin{array}{ll} 
				\Delta_{\Stwo}\psi(k,r,y) = 2^{-1-\frac{\gamma}{2}}r^{-\frac{\gamma}{2}}\nabla_k\cdot (\Pi[k]V(k,r,y))     &0 < \delta \le 2r \le R, k\in \Stwo  \\
				\psi(k,r,y) = 0     &2r \in \{\delta,R\}
			\end{array}
			\right..
		\end{equation}
		The interpretation of~\eqref{eq:kpde} is that, at every level set of $|x|$, \eqref{eq:xpde} is actually the Poisson problem for the Laplace-Beltrami operator in $\Stwo$ for the $k = \frac{x}{|x|}$ variable. Noticing that the right-hand side is a divergence of a smooth function, the integral over $\Stwo$ of the right-hand side vanishes, which is a necessary condition for solvability of the Poisson problem in a compact manifold. Using the usual method of the Lax-Milgram Theorem combined with the Poincar\'e inequality on the sphere (see~\cite{H99}), for each fixed $r\in(\delta,R)$ and $y\in\R^3$ \eqref{eq:kpde} admits a unique weak solution $\psi(\cdot,r,y) \in H^1(\Stwo)$ which is also smooth by standard elliptic regularity arguments. Finally, since $V$ is anti-symmetric when swapping $v \leftrightarrow v_*$ (which means reflecting $x \leftrightarrow - x$, or $k \leftrightarrow - k$), uniqueness of solutions gives that $\psi$ is symmetric 
		\[
		\psi(v,v_*) = \psi(v_*,v) \iff \psi(x,y) = \psi(-x,y)\iff \psi(k,r,y) = \psi(-k,r,y).
		\]
		Based on the regularity and symmetries of $V\in AS_c^\infty$, the previous discussion implies $\psi \in DS_c^\infty$. Returning to the minimization problem of~\eqref{eq:projtn}, we can deduce from convexity and our discussion on the solubility of the first order conditions \eqref{eq:tnPoisson} that there exists a unique $\psi\in DS_c^\infty$ such that
		\[
		\inf_{g\in DS_c^\infty}\|\tn g - \Pi[v-v_*]V\|_{L_{v,v_*}^2}^2 = \|\tn \psi - \Pi[v-v_*]V\|_{L_{v,v_*}^2}^2.
		\]
		We can interpret the solution operator for~\eqref{eq:kpde} as the orthogonal projection map of $V$ and $\Pi[v-v_*]V$ to the image of $\tn$. To see~\eqref{eq:L2decomp}, we add and subtract $\tn \psi$ in the $L^2$ norm of $\Pi[v-v_*]V$ to get
		\begin{align*}
			\|\Pi[v-v_*]V\|^2 =& \iiRs |\Pi[v-v_*]V - \tn \psi + \tn \psi|^2 \\
			= &\| \tn \psi\|^2 + \|\Pi[v-v_*]V - \tn \psi \|^2 + 2 \iiRs \tn \psi \cdot (\Pi[v-v_*]V - \tn \psi) \\
			= &\| \tn \psi\|^2 + \|\Pi[v-v_*]V - \tn \psi \|^2 - 2 \iiRs \psi \underbrace{(\tn\cdot (\Pi[v-v_*]V) - \tn \cdot \tn \psi)}_{=0}.
		\end{align*}
		The last line is obtained by integrating by parts the differential operator $\tn$ using the smoothness and compact support of $\psi$ and $V$. Of course, by our construction of $\psi$, the cross term contributes nothing owing to~\eqref{eq:tnPoisson}.
	\end{proof}
Using~\Cref{lem:lemlaptn}, we can further majorize the Landau dissipation from~\eqref{eq:sqrtlanantisymm}
\begin{align}
	\label{eq:sqrtlandiff}
	\begin{split}
		D_L(f) &\le \sup_{\psi \in DS_c^\infty} \left\{
		-4 \iiRs \sqrt{ff_*} \tn\cdot \tn \psi - 2\iiRs |\tn \psi|^2
		\right\}.
	\end{split}
\end{align}

\subsubsection{Boltzmann dissipation}
\label{sec:boltzdiss}
Before directly manipulating the Boltzmann dissipation, we insist on the appearance of a finite difference of $\sqrt{ff_*}$. Using~\Cref{cor:lambdaf}, we can lower bound the Boltzmann dissipation by
	\begin{align*}
		&\quad D_B(f) = \frac{1}{4}\iiRs \iS [f'f_*' - ff_*]\log \frac{f'f_*'}{ff_*} B \ge \iiRs \iS |\sqrt{f'f_*'}  - \sqrt{ff_*}|^2 B.
	\end{align*}
Let us refer to the lower bound as the reduced Boltzmann dissipation
\[
D_B^R(f) := \iiRs \iS |\sqrt{f'f_*'} - \sqrt{ff_*}|^2 B \,d\sigma dv_* dv.
\]
We want a similar representation for the reduced Boltzmann dissipation as we had for the Landau dissipation.

\begin{lemma} \label{lem:sqrtboltzdiss1}
	The reduced Boltzmann dissipation can be expressed as
	\begin{equation}\label{Rboltzmann}
		D_B^R(f) = \sup_{\tilde{\psi} \in L^2(\R^6 \times \Stwo; B(v-v_*,\theta)d\sigma dv_* dv)}\left\{
		2 \iiRs \iS (\sqrt{f'f_*'} - \sqrt{ff_*})\tilde{\psi} B d\sigma dv_* dv - \iiRs \iS |\tilde{\psi}|^2 B d\sigma dv_* dv\right\}.
	\end{equation}
\end{lemma}
\begin{proof}
	This proof is analogous to the proof of~\Cref{prop:lansqrt} in the Hilbert framework of $L^2$ with respect to the collision kernel $B$.
\end{proof}
As with the Landau dissipation, we seek to pass all the difference structure onto the test function.  Taking advantage of various pre-post collision velocity symmetries, we extend~\Cref{lem:sqrtboltzdiss1} to
\begin{lemma}
	\label{lem:sqrtboltzdiss2}
	The reduced Boltzmann dissipation can be minorized by 
	\[
	D_B^R(f) \ge \sup_{\psi \in DS_c^\infty}\left\{
	-2 \iiRs \sqrt{ff_*}\left(
	\iS \bn \psi B
	\right) - \frac{1}{4}\iiRs \iS |\bn \psi|^2 B
	\right\}.
	\]
\end{lemma}
\begin{proof}
	We are interested in showing the inequality, hence we only need to show that each $\psi \in DS_c^\infty$ induces an admissible $\tilde{\psi} \in L^2(B)$ and then de-symmetrize the first term in \eqref{Rboltzmann}.

	More specifically, for each $\psi \in DS_c^\infty$
	\[
	\frac{1}{2}\bn \psi= \psi(v',v_*') - \psi(v,v_*)=: \tilde{\psi}(v,v_*,\sigma) \in L^2(\R^6\times \Stwo;Bd\sigma dv_* dv)
	\]
	is an admissible test function for \eqref{Rboltzmann}. This is a consequence of~\Cref{lem:estbnv2} which provides the estimate
	\begin{equation}\label{est1}
		|\psi' - \psi|^2 \lesssim \indicator{0 < \delta \le |v-v_*| \le R} |v-v_*|^2 |\sigma - k|^2,
	\end{equation}
	where $\delta, R>0$ are the inner and outer radii of the support of $\psi\in DS_c^\infty$ with respect to $v-v_*$. 
	
	Next, we de-symmetrize the first term of \eqref{Rboltzmann}
	\begin{align*}
		2\iiRs \iS (\sqrt{f'f_*'} - \sqrt{ff_*})
		\tilde{\psi} B &= 2\iiRs \iS \sqrt{f'f_*'}\iS (\psi' - \psi)B  - 2\iiRs \sqrt{ff_*} \iS (\psi' - \psi) B \\
		&= - 4\iiRs \sqrt{ff_*}\iS (\psi'-\psi) B,
	\end{align*}
	which is justified as each of the integrals above are absolutely convergent. This follows from the estimate \eqref{est1}, $|\sigma - k|^2 \sim \theta^2$ from~\eqref{eq:equivmom}, and the finite angular momentum transfer assumption~\eqref{eq:betaint}. The desired inequality now follows.
\end{proof}

\subsection{Boltzmann gradient converges to Landau gradient}
\label{sec:convergediss}
The aim of this section is to show that for arbitrary $\psi \in DS_c^\infty$, we obtain
\begin{align}
	\iiRs \sqrt{f^\epsilon f_*^\epsilon} \left(\iS\bn \psi B^\epsilon\right) &\overset{\epsilon\downarrow 0}{\to} 2\iiRs \sqrt{ff_*} \tn \cdot \tn \psi \label{eq:int1diss} \\
	\iiRs \iS |\bn \psi|^2 B^\epsilon &\overset{\epsilon\downarrow 0}{\to} 8\iiRs |\tn \psi|^2. \label{eq:int2diss}
\end{align}
These limits are the final pieces needed to finish the proof of~\Cref{prop:gammadiss}. Recall there that we had reduced to the case with bounded dissipation for a fixed time
\[
\sup_{\epsilon>0}D_B^\epsilon(f^\epsilon) < + \infty,
\]
which implies the local $L_{loc}^2$ convergence $\sqrt{f^\epsilon} \to \sqrt{f}$ which we will use in the proof below.

The key to these limits is understanding the limiting behaviour of $\bn$ on $DS_c^\infty$ functions. Naturally, the crucial ingredients are~\Cref{lem:epszerobnpsiv2,lem:epszerobnpsiavgv2} which state
\begin{equation}
	\label{eq:bnpsiv2}
	\forall \psi \in DS_c^\infty, \quad \left\{
	\begin{array}{lcl}
		\frac{1}{\epsilon}\bn \psi     &\overset{\epsilon\downarrow 0}{\to} &\frac{\chi}{\pi} |v-v_*| p\cdot (\nabla - \nabla_*) \psi  \\
		\frac{1}{\epsilon^2}\iSk \bn \psi dp     &\overset{\epsilon\downarrow 0}{\to}    &\frac{\chi^2}{4\pi} (\nabla -\nabla_*) \cdot (|v-v_*|^2 \Pi[v-v_*](\nabla - \nabla_*)\psi)
	\end{array}
	\right..
\end{equation}
Recall the notation that $\chi = \pi\theta/\epsilon$ is the rescaled angle of collision and $p \in \Stwo$ is the orthogonal vector to $k = \frac{v-v_*}{|v-v_*|}$ shown in~\Cref{fig:pic1,fig:spherical,fig:phi}. These limits hold in the pointwise sense.
\begin{proof}[Proof of~\eqref{eq:int1diss} and~\eqref{eq:int2diss}]
	We begin with showing~\eqref{eq:int1diss}. Since we know $\sqrt{f^\epsilon} \to \sqrt{f}$, it remains to show that
	\[
	\iS \bn \psi B^\epsilon = |v-v_*|^\gamma\int_{\theta=0}^{\epsilon/2} \beta^\epsilon(\theta) \iSk \bn \psi \, dp d\theta \overset{\epsilon\downarrow 0}{\to} 2\tn\cdot \tn \psi.
	\]
	Since $\psi\in DS_c^\infty$, we are in the nice situation of avoiding all problems in the $v,\,v_*$ variables because $\psi$ is supported in 
	\[
	|v|+|v_*| \le R, \quad 0 < \delta \le |v-v_*| \le R,
	\]
	for some $\delta,R>0$. Combining this localisation with the third estimate of~\Cref{lem:estbnv2},
	\[
	\left|
	\iSk \bn \psi dp
	\right|\lesssim_\psi |v-v_*|^2 \theta^2,
	\]
	we get domination in $L^2$ recalling finite angular momentum transfer~\eqref{eq:betaint}, so we only have to show pointwise a.e. convergence of the previous limit. By rescaling $\theta = \epsilon\chi/\pi$ and applying~\eqref{eq:bnpsiv2}, we have
	\begin{align*}
		\iS \bn \psi B^\epsilon &= |v-v_*|^\gamma \int_0^{\pi/2} \beta(\chi)\frac{\pi^2}{\epsilon^2}\iSk \bn \psi dp d\chi \\
		&\overset{\epsilon\downarrow 0}{\to} |v-v_*|^{\gamma}\frac{\pi}{4}\int_0^{\pi/2}\chi^2\beta(\chi) (\nabla - \nabla_*)\cdot (|v-v_*|^2\Pi[v-v_*](\nabla - \nabla_*)\psi) = 2\tn \cdot \tn \psi,
	\end{align*}
	again recalling the normalization $\int_0^{\pi/2}\chi^2\beta(\chi)d\chi = 8/\pi$.
	
	We turn to showing~\eqref{eq:int2diss}. Arguing as in the proof of~\eqref{eq:int1diss}, since $\psi$ is compactly supported in $v,v_*$, and in $\{|v-v_*| \ge \delta > 0\}$, we only need to show
	\[
	\iS |\bn \psi|^2 B^\epsilon = |v-v_*|^\gamma\int_{\theta=0}^{\epsilon/2}\beta^\epsilon(\theta)\iSk |\bn \psi|^2 dp d\theta \overset{\epsilon \downarrow 0}{\to} 8|\tn \psi|^2, \quad \text{a.e. }v,v_*\in\R^3.
	\]
	This is because of the first estimate of~\Cref{lem:estbnv2} which again gives the right majorant against $B^\epsilon$
	\[
	|\bn \psi|^2 \lesssim_\psi |v-v_*|^2 \theta^2.
	\]
	By the same rescaling $\theta = \epsilon\chi/\pi$, we have
	\begin{align*}
		\iS |\bn \psi|^2 B^\epsilon &= |v-v_*|^\gamma \int_0^{\pi/2}\!\!\!\!\!\beta(\chi) \pi^2\iSk \left|
		\frac{1}{\epsilon}\bn \psi
		\right|^2 dp d\chi \overset{\epsilon\downarrow 0}{\to} |v-v_*|^{2+\gamma} \int_0^{\pi/2} \!\!\!\!\!\beta(\chi) \chi^2 \iSk |p \cdot (\nabla - \nabla_*)\psi|^2 dp d\chi \\
		&= \frac{8}{\pi} |v-v_*|^{2+\gamma}\iSk p^ip^j (\nabla- \nabla_*)^i\psi (\nabla - \nabla_*)^j \psi dp = 8 |v-v_*|^{2+\gamma} |\Pi[v-v_*](\nabla-\nabla_*)\psi|^2.
	\end{align*}
	In the last line, we have used~\Cref{lem:pilem} for the computation of
	\[
	\iSk p\otimes p dp = \pi \Pi[k],
	\]
	while remembering that since $\Pi[k]$ is a projection matrix, the quadratic form satisfies
	$
	z^T \Pi[k]z = |\Pi[k]z|^2$, for all $z \in \R^3.
	$
\end{proof}

\section{Lower semicontinuous convergence of metric derivatives in the grazing collision limit}
\label{sec:gammamd}
We consider a sequence of curves $(f^\epsilon)_{\epsilon>0}$ satisfying the uniform moment and metric derivative bounds~\eqref{eq:bounds}. In particular, we know that a convergent subsequence exists to $f$ by~\Cref{cor:limcurve}. Along this subsequence and in parallel to the general affine representation strategy in~\Cref{sec:gammadiss}, we seek to prove
\begin{proposition}
	\label{prop:gammamd}
	Consider a sequence of curves $(f^\epsilon)_{\epsilon>0}$ satisfying the uniform bounds in second moment, entropy, dissipation, and metric derivative~\eqref{eq:bounds}. Along possibly a further subsequence for which $f^\epsilon\overset{\sigma}{\rightharpoonup} f$ from~\Cref{cor:limcurve}, we have
	\[
	\liminf_{\epsilon \downarrow 0} |\dot{f}^\epsilon|_\epsilon^2(t) \ge |\dot{f}|_L^2(t), \quad \text{a.e. } t\in [0,T].
	\]
\end{proposition}
As in~\Cref{sec:dissrep}, we will prove~\Cref{prop:gammamd} by looking at the affine representations of the metric derivatives. Without loss of generality, we assume $\sup_{\epsilon>0}|\dot{f}^\epsilon|_\epsilon^2(t)<+\infty$ by taking a subsequence, if necessary. 
	\begin{lemma}
		\label{lem:mdrep}
		We consider $(f^\epsilon,M_B^\epsilon)$ and $(f,M)$ curves for $CRE$ (in the sense of Erbar~\cite{E19}) and $GCE$ (in the sense of~\cite{CDDW20}), respectively. We assume that $M_B^\epsilon$ and $M$ are optimal collision rates in the sense that their associated metric derivatives can be written as the respective action of these curves
		\begin{align*}
			|\dot{f}^\epsilon|_\epsilon^2(t) = \frac{1}{4}\iiRs \iS \frac{|M_B^\epsilon|^2}{\Lambda(f^\epsilon) B^\epsilon}d\sigma dv_* dv \qquad \mbox{and} \qquad
			|\dot{f}|_L^2(t) = \frac{1}{2}\iiRs \iS \frac{|M|^2}{ff_*} dv_* dv.
		\end{align*}
		Then, we have the following affine representations.
		\begin{align}
			\label{eq:mdrep}
			\begin{split}
				|\dot{f}^\epsilon|_\epsilon^2(t)    &= \sup_{\psi \in C_c^\infty(\R^3)}\left\{
				\frac{1}{2}\iiRs \iS M_B^\epsilon \bn \psi - \frac{1}{4}\iiRs \iS |\bn \psi|^2 \Lambda(f^\epsilon) B^\epsilon
				\right\}, \\
				|\dot{f}|_L^2(t)    &= \sup_{\psi \in C_c^\infty(\R^3)}\left\{
				\iiRs M\cdot \tn \psi - \frac{1}{2}\iiRs |\tn \psi|^2 ff_* 
				\right\}.
			\end{split}
		\end{align}
	\end{lemma}
	\begin{proof}
		We only show the proof for the first of these claimed equations since the argument for the Landau metric derivative is analogous. We shall also drop the superscript $\epsilon$ and abbreviate $\Lambda = \Lambda(f^\epsilon)$ for ease of notation. Arguing as in~\Cref{sec:landiss}, we have
		\begin{equation}
			\label{eq:mdrep1}
			\frac{1}{4}\iiRs \iS \frac{|M_B|^2}{\Lambda B} = \sup_{\xi \in L^2(\Lambda B)} \left\{
			\frac{1}{2}\iiRs \iS M_B\xi - \frac{1}{4}\iiRs \iS |\xi|^2 \Lambda B
			\right\}
		\end{equation}
		which follows from similar lines of reasoning as~\Cref{prop:lansqrt} and~\Cref{lem:sqrtboltzdiss1}.
		
		We now want to replace the test function $\xi \in L^2(\Lambda B)$ by $\bn \psi$ for $\psi \in C_c^\infty(\R^3)$. By the optimality of $M_B$ and the finite action assumption, \Cref{lem:Erbarmetvel}, we can write $M_B = U \Lambda B$ where the density $U$ can be approximated by
		\[
		U \in \overline{\{
			\bn \psi \, | \, \psi \in C_c^\infty(\R^3)
			\}}^{L^2(\Lambda B)}.
		\]
		Take $(\psi_n)_{n\in\mathbb{N}}\subset C_c^\infty(\R^3)$ an approximating sequence so that $\bn \psi_n \to U$ in $L^2(\Lambda B)$ as $n\to \infty$. The expression inside the supremum with $\xi  = \bn \psi_n$ of~\eqref{eq:mdrep1} has the following limit
		\begin{align*}
			& \iiRs \iS \frac{M_B}{\Lambda B} \bn \psi_n \Lambda B - \frac{1}{2}\iiRs \iS |\bn \psi_n|^2\Lambda B\overset{n\to \infty}{\to} \iiRs \iS \frac{M_B}{\Lambda B}U - \frac{1}{2}\iiRs \iS |U|^2 \Lambda B
			= \iiRs \iS \frac{|M_B|^2}{2\Lambda B}.
		\end{align*}
		This establishes the first equation of~\eqref{eq:mdrep}. As mentioned before, the second equation of~\eqref{eq:mdrep} follows analogously.
	\end{proof}
	With~\Cref{lem:mdrep}, we are in a position to prove~\Cref{prop:gammamd}.
	\begin{proof}[Proof of~\Cref{prop:gammamd}]
		
		Without loss of generality, take $M_B^\epsilon$ optimal collision rates in the sense of~\Cref{lem:Erbarmetvel} so  that the Boltzmann metric derivative is the action of $(f^\epsilon,M_B^\epsilon)$
		\[
		|\dot{f}^\epsilon|_\epsilon^2(t) = \frac{1}{4}\iiRs \iS \frac{|M_B^\epsilon|^2}{\Lambda(f^\epsilon) B^\epsilon}d\sigma dv_* dv.
		\]
		By the computations of~\Cref{sec:lanlim}, 
		$$
		\mbox{for } q  = \left\{
		\begin{array}{cl}
			1,     &\gamma \in [-2,0)  \\
			2,     &\gamma \in [-4,-2) 
		\end{array}
		\right., \mbox{ sufficiently small } 0 < \delta < \left\{
		\begin{array}{ll}
			-\frac{\gamma}{2},     &\gamma\in[-2,0)  \\
			-\frac{\gamma}{2}-1,     &\gamma\in[-4,-2) 
		\end{array}
		\right.,
		$$
		and along some subsequence $\epsilon \downarrow 0$, we have that $(f, L_{q,\delta}(M_{q,\delta}))$ is an admissible pair in the grazing continuity equation where
		\[
		|v-v_*|^q \left(1 + \left[|v|^2 + |v_*|^2 \right]^\frac{\delta}{2} \right) \theta M^\epsilon \overset{\epsilon\downarrow 0}{\rightharpoonup} M_{q,\delta}, \quad L_{q,\delta}(M_{q,\delta}) = \frac{|v-v_*|^{-\gamma/2-q}}{4\left(1 + \left[|v|^2 + |v_*|^2 \right]^\frac{\delta}{2} \right)}\iS M_{q,\delta} p d\sigma.
		\]
		As well, since $L_{q,\delta}(M_{q,\delta})$ is admissible, there exists (by \Cref{lem:Erbarmetvel}) a unique grazing rate $M\in \mathcal{M}_L$ so that we have the following inequality for the Landau metric derivative
		\[
		|\dot{f}|_L^2(t) = \frac{1}{2}\iiRs \frac{|M|^2}{ff_*}dv_* dv \le \frac{1}{2}\iiRs \frac{|L_{q,\delta}(M_{q,\delta})|^2}{ff_*}dv_* dv.
		\]
		Recall the affine representations of the metric derivatives from~\Cref{lem:mdrep}
		\[
		\frac{1}{4}\iiRs \iS \frac{|M_B^\epsilon|^2}{\Lambda(f^\epsilon)B^\epsilon} = \sup_{\psi\in C_c^\infty(\R^3)}\left\{
		\frac{1}{2}\iiRs \iS M_B^\epsilon \bn \psi - \frac{1}{4}\iiRs \iS |\bn \psi|^2 \Lambda(f^\epsilon) B^\epsilon
		\right\}
		\]
		and
		\[
		\frac{1}{2}\iiRs \frac{|M|^2}{ff_*} = \sup_{\psi\in C_c^\infty(\R^3)}\left\{
		\iiRs M\cdot \tn \psi - \frac{1}{2}\iiRs |\tn \psi|^2 ff_*
		\right\}.
		\]
		It remains to establish that, for the $L_{q,\delta}(M_{q,\delta})$ rate (coming from the sequence $\{M_B^\epsilon\}_{\epsilon}$ in $CRE$) and the corresponding unique optimal rate $M$ (coming directly $GCE$), we must have
		\begin{align}
			\label{eq:liftdual}
			\begin{split}
				\sup_{\psi\in C_c^\infty(\R^3)}\left\{
				\iiRs M\cdot \tn \psi - \frac{1}{2}\iiRs |\tn \psi|^2 ff_*
				\right\} &= \sup_{\psi\in C_c^\infty(\R^3)}\left\{
				\iiRs L_{q,\delta}(M_{q,\delta})\cdot \tn \psi - \frac{1}{2}\iiRs |\tn \psi|^2 ff_*
				\right\}.
			\end{split}
		\end{align}
		This follows because
		\[
		M = L_{q,\delta}(M_{q,\delta}) + \underbrace{(M - L_{q,\delta}(M_{q,\delta}))}_{\tn\cdot\text{-free}}.
		\]
		Here, we say `$\tn\cdot$-free' to mean that for any $\psi \in C_c^\infty(\R^3)$, we have
		\[
		\iiRs \iS \tn \psi \cdot d(M - L_{q,\delta}(M_{q,\delta})) = 0,
		\]
		which follows since both $M, \, L_{q,\delta}(M_{q,\delta})$ solve the grazing continuity equation with the same $f$
		\begin{align*}
			\frac{d}{dt}\iR \psi f_t(v) dv + \frac{1}{2}\iiRs \tn \psi \cdot dM_t(v,v_*)= \frac{d}{dt}\iR \psi f_t(v) dv + \frac{1}{2}\iiRs \tn \psi \cdot d(L_{q,\delta}(M_{q,\delta}))(v,v_*), 
		\end{align*}
		for all $\psi \in C_c^\infty(\R^3)$. Looking at the level of the dual formulations, for fixed $\psi \in C_c^\infty(\R^3)$ we claim
		\begin{align}
			\frac{1}{2}\iiRs \iS M_B^\epsilon \bn \psi &\overset{\epsilon\downarrow 0}{\to }\iiRs L_{q,\delta}(M_{q,\delta}) \cdot \tn \psi, \label{eq:liftconverge}\\
			\limsup_{\epsilon\downarrow 0}\frac{1}{4}\iiRs \iS |\bn \psi|^2 \Lambda(f^\epsilon)B^\epsilon &\le\frac{1}{2}\iiRs |\tn \psi|^2 ff_*. \label{eq:limsupnabla}
		\end{align}
		Once~\eqref{eq:liftconverge} and~\eqref{eq:limsupnabla} are proven for fixed $\psi \in C_c^\infty$, this establishes
		\[
		\liminf_{\epsilon\downarrow 0}|\dot{f}^\epsilon|_B^2(t) \ge \sup_{\psi\in C_c^\infty(\R^3)}\left\{
		\iiRs L_{q,\delta}(M_{q,\delta})\cdot \tn \psi - \frac{1}{2}\iiRs |\tn \psi|^2 ff_*
		\right\}.
		\]
		Recalling~\eqref{eq:liftdual}, the right-hand side is the affine representation of $|\dot{f}|_L^2(t)$ we can conclude the proof.
		
		To prove~\eqref{eq:liftconverge}, we use the estimates for $\bn \psi$ from~\Cref{lem:estbnv2} combined with the appropriate choice of $q\in\left[-\frac{\gamma}{2},1-\frac{\gamma}{2}\right)$ and small $\delta>0$ depending on $\gamma$. We recall the computations from~\Cref{sec:lanlim} - the first step is to reveal the correct sequence of measures from the scaled compactness~\Cref{prop:limm};
		\begin{align*}
			\iiRs \iS M_B^\epsilon \bn \psi &= \iiRs \iS \left(|v-v_*|^q\left(1 + \left[|v|^2 + |v_*|^2 \right]^\frac{\delta}{2} \right) \theta M_B^\epsilon \right)|v-v_*|^{-q}\left(1 + \left[|v|^2 + |v_*|^2 \right]^\frac{\delta}{2} \right)^{-1}\frac{\bn \psi}{\theta} \\
			&\overset{\epsilon \downarrow 0}{\to} \frac{1}{2}\iiRs |v-v_*|^{-q-\frac{\gamma}{2}}\left(1 + \left[|v|^2 + |v_*|^2 \right]^\frac{\delta}{2} \right)^{-1}\tn \psi \cdot \iS M_{q,\delta} p d\sigma dv_* dv \\
			&= 2\iiRs \tn \psi \cdot L_{q,\delta}(M_{q,\delta})dv_* dv.
		\end{align*}
		The justification for the weak-strong convergence is also found in~\Cref{sec:lanlim}.
		
		Turning to the proof of~\eqref{eq:limsupnabla}, we use~\Cref{cor:lambdaf} and symmetry, to write
		\begin{align*}
			\iiRs \iS |\bn \psi|^2 B^\epsilon\Lambda(f^\epsilon) \le \iiRs \iS |\bn \psi|^2 B^\epsilon \left[ 
			\frac{{f^\epsilon}'{f^\epsilon}_*' + f^\epsilon f_*^\epsilon}{2}
			\right] = \iiRs f^\epsilon f_*^\epsilon \left(\iS|\bn \psi|^2 B^\epsilon\right). 
		\end{align*}
		We first provide an integrable majorant (uniformly against the measure $f^\epsilon f_*^\epsilon$) for the term integrated over $\Stwo$ which is again provided by~\Cref{lem:estbnv2}. We have
		\begin{align*}
			\iS |\bn \psi|^2 B^\epsilon &= \pi^2|v-v_*|^\gamma \int_0^{\pi/2}\beta(\chi)\iSk \left|\frac{1}{\epsilon}\bn \psi\right|^2 dp d\chi \\
			&\lesssim \int_0^{\pi/2}\chi^2\beta(\chi)d\chi \left\{
			\begin{array}{cl}
				\text{Lip}(\psi)|v-v_*|
				^{2+\gamma},     &\gamma\in [-2,0)  \\
				\|D^2\psi\|_{L^\infty}|v-v_*|^{4+\gamma},     &\gamma\in [-4,-2)
			\end{array}
			\right. \lesssim_{\psi,\beta} 1+|v-v_*|^l.
		\end{align*}
		Here $0<l<2$ is some power which gives strictly subquadratic growth. By the uniformly bounded moments assumption~\eqref{eq:bounds}, this is uniformly integrable against $f^\epsilon f_*^\epsilon$ so we can pass to the weak-strong limit. According to~\Cref{lem:epszerobnpsiv2}, we have the pointwise limit
		\begin{align*}
			\iS |\bn& \psi|^2 B^\epsilon = \pi^2 |v-v_*|^\gamma \int_0^{\pi/2}\beta(\chi) \iSk \left|
			\frac{1}{\epsilon}\bn \psi
			\right|^2 dp d\chi \\
			&\overset{\epsilon \downarrow 0}{\to} \frac{1}{4}|v-v_*|^{2+\gamma}\int_0^{\pi/2}\!\!\!\!\!\chi^2 \beta(\chi) \iSk \!\!|p\cdot (\nabla - \nabla_*)\psi|^2 dp d\chi = 2|v-v_*|^{2+\gamma}|\Pi[v-v_*](\nabla - \nabla_*)\psi|^2 = 2|\tn \psi|^2.
		\end{align*}
		Here, we have used~\Cref{lem:pilem} and the usual finite angular momentum transfer~\eqref{eq:betaint}.
	\end{proof}
\begin{appendices}
	\section{Some inequalities and the spherical Laplacian}
	\label{sec:lm}
	We recall a useful extension of the AM-GM inequality which is of independent interest and provides some useful bounds concerning $\Lambda(f)$. We also write down a few useful computations for the projection matrix $\Pi$. Lastly, we give an expression for the spherical Laplacian in terms of $\Pi$.
	
	\begin{lemma}
		[ALG inequality]
		\label{lem:ALG}
		The logarithmic mean separates the arithmetic and geometric means;
		\[
		\sqrt{ab} \le \frac{b-a}{\log b - \log a} \le \frac{a+b}{2}, \quad \forall a,b>0, \, \, a\ne b. 
		\]
		Equality is achieved in any of the inequalities if and only if $a=b$ and equality is achieved in all of them, where the logarithmic mean between $a=b>0$ is defined as $a$.
	\end{lemma}
	\begin{proof}
		We follow the elegant proof by S\'andor~\cite{S16}. First we claim the following
		\begin{equation}
			\label{eq:algclaim}
			\frac{4}{(t+1)^2} < \frac{1}{t} < \frac{1}{2\sqrt{t}} + \frac{1}{2t\sqrt{t}}, \quad \forall t > 1.
		\end{equation}
		The left-hand inequality of~\eqref{eq:algclaim} follows from
		$
		(t+1)^2 - 4t = (t-1)^2 > 0.
		$
		For the right-hand inequality of~\eqref{eq:algclaim}, we use Young's (strict) inequality since $t>1$
		\[
		\frac{1}{t} = \frac{1}{t^\frac{1}{4}}\cdot \frac{1}{t^\frac{3}{4}} < \frac{1}{2t^\frac{1}{2}} + \frac{1}{2t^\frac{3}{2}}.
		\]
		Assume without loss of generality that $0 < a < b$. We integrate~\eqref{eq:algclaim} from $1$ to $b/a > 1$. This leads to
		\[
		2\frac{b-a}{b+a} < \log b - \log a < \sqrt{\frac{b}{a}} - \sqrt{\frac{a}{b}} = \frac{b-a}{\sqrt{ab}}.
		\]
		Dividing by $b-a$ and inverting yields the result.
	\end{proof}
		\begin{corollary}
			For all $a, b >0$ with $a\neq b$ we have
			\[
			(a - b)\log \frac{a}{b} \ge 4|\sqrt{a} - \sqrt{b}|^2.
			\]
		\end{corollary}
		\begin{proof}
			We write $a-b$ as a difference of squares to deduce
			\begin{align*}
				(a-b)\log \frac{a}{b} &= 2(\sqrt{a} - \sqrt{b})(\sqrt{a} + \sqrt{b})(\log \sqrt{a} - \log \sqrt{b}) \\
				&=4|\sqrt{a} - \sqrt{b}|^2 \left(\frac{\sqrt{a}+\sqrt{b}}{2}\right)\left(\frac{\log \sqrt{a} - \log \sqrt{b}}{\sqrt{a} - \sqrt{b}}\right) \\
				&\ge 4|\sqrt{a} - \sqrt{b}|^2.
			\end{align*}
			In the last line, we used~\Cref{lem:ALG} after recognising the arithmetic and logarithmic means between $\sqrt{a}$ and $\sqrt{b}$.
		\end{proof}
	\begin{corollary}
		[$\Lambda(f)$ bounds]
		\label{cor:lambdaf}
		$\Lambda(f) = \frac{f'f_*' - ff_*}{\log f'f_*' - \log ff_*}$ grows `quadratically' in $f$ (in the sense of tensor products);
		\[
		\sqrt{ff_*f'f_*'} < \Lambda(f) < \frac{f'f_*' + ff_*}{2}.
		\]
		Moreover, there holds
			\[
			(f'f_*' - ff_*)\log \frac{f'f_*'}{ff_*} \ge 4|\sqrt{f'f_*'} - \sqrt{ff_*}|^2.
			\]
	\end{corollary}
	\begin{lemma}
		\label{lem:pilem}
		Suppose $(\hat{k},\hat{h},\hat{i})$ is an orthonormal basis of $\Rthree$. There holds
		\[
		\int_0^{2\pi} (\cos\phi \hat{h} + \sin \phi \hat{i})\otimes (\cos \phi \hat{h} + \sin \phi \hat{i}) d\phi = \pi\Pi[\hat{k}] = \pi(I - \hat{k}\otimes \hat{k}).
		\]
		This is equivalent to
		\[
		\int_{\Sk} p\otimes p dp = \int_0^{2\pi} p \otimes p d\phi = \pi \Pi[k],
		\]
		where $p$ is orthonormal to $k$ with azimuthal angle $\phi$ i.e.
		$
		p = \cos \phi h + \sin \phi i.
		$
	\end{lemma}
	\begin{proof}
		In the basis of $(\hat{k},\hat{h},\hat{i})$, we can represent the matrix in the integral as
		\begin{align*}
			\int_0^{2\pi}(\cos\phi \hat{h} + \sin \phi \hat{i})\otimes (\cos \phi \hat{h} + \sin \phi \hat{i}) d\phi &= \int_0^{2\pi} \left(
			\begin{array}{ccc}
				0     &0    &0  \\
				0     &\cos^2\phi    &\cos\phi \sin\phi \\
				0   &\cos \phi \sin\phi     &\sin^2 \phi
			\end{array}
			\right)d\phi    \\
			&= \left(
			\begin{array}{ccc}
				0     &0    &0  \\
				0     &\pi  &0  \\
				0     &0    &\pi
			\end{array}
			\right) = \pi \left(
			I - \hat{k} \otimes \hat{k}
			\right) = \pi \Pi[\hat{k}].
		\end{align*}
	\end{proof}
	
	Next we turn to explicit expressions for the spherical Laplacian/Laplace-Beltrami operator, see \cite{RimGeo} for details. Consider a smooth function $f : \R^d \to \R$. Let us write $x\in \R^d$ as $x = r\omega$ for $r = |x|$ and $\omega = \frac{x}{|x|} \in \mathbb{S}^{d-1}$. The spherical Laplacian of $f$ denoted $\Delta_{\mathbb{S}^{d-1}}f$ is obtained by 
	\[
	\Delta_{\mathbb{S}^{d-1}}f(x) = \Delta f \left(\frac{x}{|x|}\right)
	\]
	and it satisfies
	\begin{equation}
		\label{eq:sphereLaplacian}
		\Delta f = \partial_r^2 f + \frac{d-1}{r}\partial_r f + \frac{1}{r^2}\Delta_{\mathbb{S}^{d-1}}f.
	\end{equation}
	We write the spherical Laplacian in terms of $\nabla_\omega$ and $\Pi[\omega]$. 
	\begin{lemma}
		\label{lem:LB}
		Under the spherical coordinates $x = r\omega$, and for smooth functions $\phi = \phi(x) = \phi(r,\omega)$, the spherical Laplacian of $\phi$ reads
		\[
		\Delta_{\mathbb{S}^{d-1}}\phi = \nabla_\omega \cdot \left(
		\Pi[\omega]\nabla_\omega \phi
		\right),
		\]
		where $\nabla_\omega$ is the differential operator applied to the zero-homogeneous extension of functions on the sphere;
		\[
		\nabla_\omega \phi(\omega) = \nabla_x \phi \left(
		\frac{x}{|x|}
		\right).
		\]
		$\nabla_\omega\cdot$ is the adjoint of $\nabla_\omega$.
	\end{lemma}
	\begin{proof}
		We recompute $\Delta \phi$ using spherical coordinates. Then we identify the corresponding $\Delta_{\mathbb{S}^{d-1}}\phi$ term in~\eqref{eq:sphereLaplacian}. For every index $i$, the chain rule gives
		\begin{align*}
			[\nabla \phi]^i = \partial^i \phi = \frac{\partial r}{\partial x^i}\partial_r \phi + \frac{\partial \omega^j}{\partial x^i}\partial_{\omega^j}\phi = \omega^i \partial_r \phi + \frac{1}{r}\Pi^{ij}[\omega]\partial_{\omega^j}\phi.
		\end{align*}
		Here, we have recalled the simple computations
		$
		\frac{\partial r}{\partial x^i} = \omega^i$ and $ \frac{\partial \omega^j}{\partial x^i} = \frac{1}{r}\Pi^{ij}[\omega].
		$
		Writing the Laplacian with repeated indices, we further compute
		\begin{align*}
			\Delta \phi &= \partial^i \partial^i \phi = \frac{\partial r}{\partial x^i} \partial_r \left(
			\omega^i \partial_r \phi + \frac{1}{r}\Pi^{ij}[\omega]\partial_{\omega^j}\phi
			\right) + \frac{\partial \omega^k}{\partial x^i} \partial_{\omega^k}\left(
			\omega^i \partial_r \phi + \frac{1}{r}\Pi^{ij}[\omega]\partial_{\omega^j}\phi
			\right)     \\
			&= \omega^i \left(\omega^i \partial_r^2 \phi + \underbrace{\Pi^{ij}[\omega]\partial_r \left(
				\frac{1}{r}\partial_{\omega^j}\phi
				\right)}_{\perp \omega^i}\right) + \frac{1}{r}\Pi^{ik}[\omega]\left(
			\delta^{ik}\partial_r\phi + \underbrace{\omega^i \partial_{\omega^j}\partial_r \phi}_{\perp \Pi^{ik}[\omega]} + \frac{1}{r}\partial_{\omega^k}\left(
			\Pi^{ij}[\omega]\partial_{\omega^j}\phi
			\right)
			\right).
		\end{align*}
		Here, we have expanded the derivatives using the previous computations. In particular, the underbraced terms contribute nothing (as expected since these are the mixed derivatives in the radial and spherical directions). Recall now that
		\[
		\omega^i\omega^i = 1, \quad \Pi^{ik}[\omega]\delta^{ik} = \text{trace} \left(\Pi[\omega]\right) = d-1, \quad \partial_{\omega^k} \Pi^{ij}[\omega] = - (\delta^{ik}\omega^j + \delta^{jk}\omega^i).
		\]
		Using these identities, we further simplify
		\begin{align*}
			&\quad \Delta \phi = \partial_r^2 \phi + \frac{(d-1)}{r}\partial_r \phi + \frac{1}{r^2}\Pi^{ik}[\omega] \left(
			-(\delta^{ik}\omega^j + \underbrace{\delta^{jk}\omega^i}_{\perp \Pi^{ik}[\omega]}) \partial_{\omega^j}\phi + \Pi^{ij}[\omega]\partial_{\omega^k}\partial_{\omega^j}\phi
			\right) \\
			&= \partial_r^2 \phi + \frac{(d-1)}{r}\partial_r \phi + \frac{1}{r^2}\left(
			-(d-1)\omega^j \partial_{\omega^j}\phi + \Pi^{jk}[\omega]\partial_{\omega^k}\partial_{\omega^j}\phi
			\right).
		\end{align*}
		We can repackage the spherical Laplacian term in another neat way by noticing that
		\[
		\nabla_\omega \cdot (\Pi[\omega] \nabla_\omega \phi) = \partial_{\omega^k} \left(
		\Pi^{jk}[\omega]\partial_{\omega^j}\phi
		\right) = -(d-1)\omega^j \partial_{\omega^j}\phi + \Pi^{jk}[\omega]\partial_{\omega^k}\partial_{\omega^j}\phi.
		\]
		Putting this back gives
		\[
		\Delta \phi = \partial_r^2 \phi + \frac{(d-1)}{r}\partial_r \phi + \frac{1}{r^2} \nabla_\omega\cdot (\Pi[\omega]\nabla_\omega \phi).
		\]
	\end{proof}
	\begin{corollary}
		\label{cor:deromega}
		Under the same notations as~\Cref{lem:LB}, in particular $x = r\omega$, for any smooth vector field $V$, we have
		\[
		\nabla_x\cdot (\Pi[x]V) = \frac{1}{r}\nabla_\omega\cdot (\Pi[\omega] V).
		\]
		Moreover, for smooth $\phi$, we have
		\[
		\nabla_x\cdot (\Pi[x]\nabla_x \phi) = \frac{1}{r^2}\nabla_\omega\cdot (\Pi[\omega] \nabla_\omega \phi).
		\]
	\end{corollary}
	\begin{proof}
		The first identity is a direct computation. For the second identity, repeat the calculations in ~\Cref{lem:LB} noticing that $\Pi[x]$ applied to $\nabla_x \phi$ removes the radial derivative contribution.
	\end{proof}
	\section{Strong compactness from bounded Boltzmann dissipation}
	\label{sec:strcpct}
	The purpose of this section is to derive an estimate guaranteeing strong compactness in the grazing collision limit $\sqrt{f^\epsilon} \to \sqrt{f}$ in $L_{loc}^2$. We repeat here the main results of~\cite{ADVW00, AV02} which we emphasize are independent of the grazing collision parameter $\epsilon \downarrow 0$ provided the finite angular momentum transfer~\eqref{eq:betaint} and uniform moments and entropy bounds hold~\eqref{eq:bounds}. Let $f_R$ denote $f\chi_R$ where $\chi_R$ is a smooth cut-off function equal to 1 in $B_R$ and vanishing outside of $B_{R+1}$ (we will make this precise later). The estimate we wish to show is
	\begin{equation}
		\label{eq:strcompact}
		\int_{\R^3} \left|\mathcal{F}\left[\sqrt{f_R^\epsilon}\right](\xi)\right|^2 \min(|\xi|^2, \,|\xi|^\nu) d\xi \le C_R(D_B^\epsilon(f^\epsilon) + 1), \quad \forall R>1,
	\end{equation}
	where $\mathcal{F}$ stands for the Fourier transform and the constant $C_R>0$ depending on $R>1$ is \textit{independent} of $\epsilon>0$ besides uniform bounds on the moments and entropy (pointwise in time versions of~\eqref{eq:bounds}). We recall from~\Cref{sec:sigmarep} that $\nu>0$ is the quantity which controls the angular singularity of the collision kernel. As in~\ref{ass:beta}, we insist on decoupling $\nu \in (0,2), \gamma \in [-4,0]$.
	
	As in~\cite{ADVW00}, we first outline the main steps and postpone the details. We will show~\eqref{eq:strcompact} for the particular Boltzmann collision kernel
	\[
	B^\epsilon(z,\sigma) = |z|_{kin}^\gamma b^\epsilon\left(
	\frac{z}{|z|}\cdot \sigma
	\right), \quad |z|_{kin}^\gamma = \left\{
	\begin{array}{cc}
		1     &|z| \le 1  \\
		|z|^\gamma     &|z| \ge 1 
	\end{array}
	\right. \le 1.
	\]
	The dissipation associated to this kernel is certainly less than the dissipation for those kernels without cutting off in the kinetic singularity near $v = v_*$ (such as those we consider from~\ref{ass:beta}).
	\begin{proof}[Proof of~\eqref{eq:strcompact}]
		For ease of notation, we identify $f \equiv f^\epsilon$. Proving~\eqref{eq:strcompact} in this setting of cut-off kinetic singularities clearly implies the full generality of the result since $|z|_{kin}^\gamma \le |z|^\gamma$. Cutting off the angular singularity part of $B^\epsilon$ if necessary, and then passing to the limit, we can rewrite the Boltzmann dissipation using the pre-post-collisional change of velocities
		\begin{align*}
			D_B^\epsilon(f) &= - \iiRs \iS B^\epsilon(f'f_*' - ff_*) \log f d\sigma dv_* dv = \iiRs\iS B^\epsilon ff_* \log \frac{f}{f'}d\sigma dv_* dv \\
			&= \iiRs \iS B^\epsilon f_* \left(
			f\log \frac{f}{f'} - f + f'
			\right)d\sigma dv_* dv  + \iiRs \iS B^\epsilon f_* (f - f') d\sigma dv_* dv.
		\end{align*}
		According to the cancellation lemma (\Cref{lem:cancel}), we can estimate the second integral with
		\[
		\left|\iiRs \iS B^\epsilon f_* (f'-f) d\sigma dv_* dv\right| \le C_1,
		\]
		with $C_1$ being a constant depending only on the moments and entropy. For the first integral, we make the square root appear with the classical inequality
		\[
		x\log \frac{x}{y} - x + y \ge (\sqrt{x} - \sqrt{y})^2, \quad \forall x,y >0,
		\]
		which can be proven by reducing to the case $y=1$ and applying the ALG inequality (\Cref{lem:ALG}). Continuing, we have
		\begin{align*}
			D_B^\epsilon(f) + C_1 &\ge \iiRs \iS B^\epsilon f_* (\sqrt{f'} - \sqrt{f})^2 d\sigma dv_* dv = \iiRs \iS |v-v_*|_{kin}^\gamma b^\epsilon(k\cdot \sigma) f_* (\sqrt{f'}-\sqrt{f})^2 d\sigma dv_* dv,
		\end{align*}
		and we also recall $k = (v-v_*)/|v-v_*|$. We set now $F(v) = \sqrt{f(v)}$ and use $F_*, \,F', \,F_*'$ as usual. Having revealed $\sqrt{f}$, we apply a smooth cut-off and pass to Fourier space. For $R>1$ we take $\chi_R \in C_c^\infty(\R^3)$ a smooth indicator function on $B_R$ such that
		$
		0 \le \chi_R \le 1$, $\chi_R|_{B_R} = 1$, and $\text{supp}\chi_R \subset B_{R+1}.$
		According to the truncation lemma (\Cref{lem:truncation}), there are constants $C_2, \, C_3>0$ such that $C_2$ depends only on~\ref{ass:bddmom} while $C_3$ depends only on $R$ and $\gamma$ such that 
		\begin{align*}
			\iiRs \iS B^\epsilon f_* (\sqrt{f'} - \sqrt{f})^2 d\sigma dv_* dv + C_2\ge C_3\iiRs \iS b^\epsilon (k\cdot \sigma) f_* \chi_{R*} (F'\chi_R' - F \chi_R)^2d\sigma dv_* dv.
		\end{align*}
		Using~\Cref{lem:fourier}, we are able to pass to Fourier variables so that the last integral can be minorized by
		\begin{align*}
			\iiRs \iS b^\epsilon (k\cdot \sigma) f_* \chi_{R*}& (F'\chi_R' - F \chi_R)^2d\sigma dv_* dv \\
			&\ge \frac{1}{2(2\pi)^3}\int_{\R^3} |\mathcal{F} [F\chi_R](\xi)|^2 \left\{
			\iS b^\epsilon\left(
			\frac{\xi}{|\xi|}\cdot \sigma
			\right) (\mathcal{F}[f\chi_R](0) - |\mathcal{F}[f\chi_R](\xi^-)|) d\sigma
			\right\}d\xi,
		\end{align*}
		where $\xi^- = ( \xi - |\xi|\sigma)/2$. Finally, the integral in curly brackets can be estimated using~\Cref{lem:avgfourier} so that there is a constant $C_4>0$ depending on the uniform bounds of moments and entropy and finite angular momentum transfer~\eqref{eq:betaint} giving
		\[
		\iS b^\epsilon\left(
		\frac{\xi}{|\xi|}\cdot \sigma
		\right) (\mathcal{F}[f\chi_R](0) - |\mathcal{F}[f\chi_R](\xi^-)|) d\sigma \ge C_4 \min(|\xi|^2,|\xi|^\nu).
		\]
		Putting these considerations together, we have
		\begin{align*}
			\frac{C_3 C_4}{2(2\pi)^3} \int_{\R^3}|\mathcal{F} [F\chi_R](\xi)|^2 \min(|\xi|^2,|\xi|^\nu) d\xi \le D_B^\epsilon(f) + C_1 + C_2.
		\end{align*}
	\end{proof}
	The rest of this section is devoted to (re)proving the lemmas that were invoked in the previous proof. In particular, we repeat the proofs involving estimates pertaining to the collision kernel $B^\epsilon$ since we want to make certain that our constants are independent of $\epsilon>0$ besides uniformly bounded moments and entropy.
	\begin{lemma}[Cancellation lemma]
		\label{lem:cancel}
		For almost every $v_* \in \R^3$ and $\epsilon>0$ sufficiently small, we have
		\[
		\int_{\R^3}\iS B^\epsilon(v-v_*,\sigma) (f'-f) d\sigma dv = [f*S^\epsilon](v_*),
		\]
		where $S^\epsilon$ is given by
		\[
		S^\epsilon(z) = 2\pi |z|_{kin}^\gamma \int_0^{\epsilon/2} \left[\cos^{-3}\left(\frac{\theta}{2}\right) - 1\right]\beta^\epsilon(\theta)d\theta.
		\]
		Moreover, we have the trivial estimate
		$
		|S^\epsilon(z)| \le 12.
		$
		Finally, the previous estimates lead to
		\[
		\left|
		\iiRs \iS B^\epsilon f_* (f'-f)
		\right| = \left|
		\iiRs f(v) f(v_*) S^\epsilon(v-v_*)
		\right| \le 12.
		\]
	\end{lemma}
	\begin{proof}
		As in the beginning of the proof of~\eqref{eq:strcompact}, we split the `gain' and `loss' part of the integral by an approximation argument, cutting off the angular singularity as necessary. Focusing on the gain term, for fixed $\sigma \in \Stwo$ and $v_* \in \R^3$, we consider the change of coordinates $v\mapsto v'$.
		\begin{figure}[H]
			\centering
			\begin{tikzpicture}
				\draw (0,0) circle (4cm);
				\draw[thick,purple,->] (180:4.06cm) -- (180:2.06cm) node[pos=0.5, below] {$k$};
				\draw[red] (180:3.06cm) arc (0:15:1cm) node[pos=0.59,right] {$\theta/2$};
				\draw[Circle-Circle] (180:4.06cm) -- (0:4.06cm) node[pos=0,left] {$v_*$} coordinate (v) node[pos=1,right] {$v = \psi_\sigma(v')$} node[pos=0.5,above left] {$\frac{v+v_*}{2}$};
				\draw[black, ->] (180:4.06cm) -- (30:4.06cm);
				\draw[thick, magenta, ->, shift = {(-4.06cm,0)}, rotate = -15] (0,0) -- (30:2cm) node[pos=0.5, above] {$k'$};
				\draw[thick,purple,->] (180:4.06cm) -- (180:2.06cm) node[pos=0.5, below] {$k$};
				\draw[thick,purple,->] (0,0) -- (0:2cm) node[pos=0.5, below] {$k$};
				\draw[Circle-Circle] (210:4.06cm) -- (30:4.06cm) node[pos=0,left] {$v_*'$} coordinate (vp) node[pos=1,right] {$v'$};
				\draw[thick,blue,->] (0,0) -- (30:2cm) node[pos=0.5, above] {$\sigma$};
				\filldraw[black] (0,0) circle (1.5pt); 
				\draw[dashed, thin] (0,0) -- (15:4cm);
				\draw[->,Periwinkle] (4,0) -- (30:4) node[pos=0.6, left] {$\omega$};
				\draw[red] (0:1cm) arc (0:30:1cm) node[pos=0.4,left] {$\theta$};
				\begin{scope}[rotate around = {15:(0,0)}]
					\draw[thin,Dandelion] (3.5,0) -- (3.5,-0.36) -- (3.86,-0.36);
				\end{scope}
				\draw[thin,Dandelion, shift = {(-0.25cm,1.05cm)}, rotate = 15] (3.5,0) -- (3.5,-0.36) -- (3.86,-0.36);
			\end{tikzpicture}
			\caption{Geometry of elastic binary collisions with additional angles.}
			\label{fig:pic1cancel}
		\end{figure}
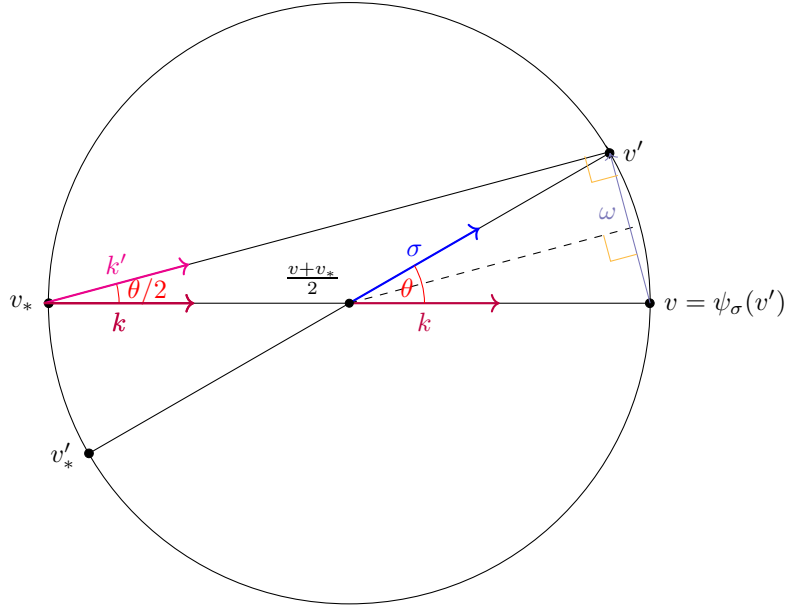
		Recalling
		\[
		v' = \frac{v+v_*}{2}+ \frac{|v-v_*|}{2}\sigma = v - \frac{|v-v_*|}{2}(\sigma - k) =  v_* +\frac{|v-v_*|}{2}(k+\sigma),
		\]
		the first of these identities implies the following equality for the Jacobian
		\[
		\left|
		\frac{\partial v'}{\partial v}
		\right| = \left|
		\frac{1}{2}I + \frac{1}{2}k \otimes \sigma
		\right| = \frac{1}{8}(1 + k\cdot \sigma).
		\]
		Graphically, see~\Cref{fig:pic1cancel}, we can switch from $k = \frac{v-v_*}{|v-v_*|}$ to $k' = \frac{v'-v_*}{|v'-v_*|}$ using the standard half-angle trigonometric identity
		$
		1+ k\cdot \sigma = 1 + \cos\theta = 2\cos^2\frac{\theta}{2} = 2 (k'\cdot \sigma)^2, 
		$
		where the last equality can be seen pictorially or by employing the same trigonometric identity from the definition of $k'$ using
		\[
		k'\cdot \sigma = \frac{1+k\cdot \sigma}{|k+\sigma|}, \quad |k + \sigma| = 2 \cos \frac{\theta}{2}.
		\]
		This leads to another form of the Jacobian determinant
		\[
		\left|
		\frac{\partial v'}{\partial v}
		\right| =\frac{1}{8}(1 + k\cdot \sigma) = \frac{(k'\cdot \sigma)^2}{4}.
		\]
		Now, since $\theta \in [0,\pi/2]$ (see~\Cref{sec:spherical}), we therefore have
		$
		k'\cdot \sigma = \cos \frac{\theta}{2} \ge \frac{1}{\sqrt{2}}.
		$
		This shows that the transformation is invertible and we define the inverse transfromation $v' \mapsto \psi_\sigma(v') = v$. Employing similar trigonometric identities as before, some computations lead to
		\[
		|v-v_*| = \frac{|v'-v_*|}{k'\cdot \sigma} \iff |\psi_\sigma(v) - v_*| = \frac{|v-v_*|}{k\cdot \sigma},
		\]
		since the collision map is involutive. Returning to the change of variable, we have
		\begin{align*}
			\int_{\R^3}\iS B^\epsilon(|v-v_*|,k\cdot\sigma) f(v') dv d\sigma &= \int_{\R^3}\iS B^\epsilon(|v-v_*|,k\cdot \sigma) f(v') \left|
			\frac{\partial v}{\partial v'}
			\right|dv' d\sigma \\
			&= \int_{\R^3}\iS B^\epsilon(|v - v_*|, 2(k'\cdot \sigma)^2 - 1) f(v') \frac{4}{(k'\cdot \sigma)^2}dv' d\sigma \\
			&= \int_{k\cdot \sigma \ge 1/\sqrt{2}} B^\epsilon(|\psi_\sigma(v) - v_*|, 2(k\cdot \sigma)^2 - 1) f(v) \frac{4}{(k\cdot \sigma)^2}dv d\sigma,
		\end{align*}
		where we just relabel $v\leftrightarrow v'$ in the last line. Inserting this back into the difference, we obtain
		\begin{align*}
			\int_{\R^3}\iS B^\epsilon(|v-v_*|,k\cdot \sigma) (f' - f)d\sigma dv 
			&= \int_{\R^3}f(v) \left[
			\int_{k\cdot \sigma \ge 1/\sqrt{2}} B^\epsilon\left(\frac{|v-v_*|}{k\cdot \sigma}, 2(k\cdot \sigma)^2 - 1\right) \frac{4}{(k\cdot \sigma)^2} d\sigma
			\right. \\
			&\quad \left.- \int_{k\cdot \sigma \ge 0}B^\epsilon(v-v_*,k\cdot \sigma) d\sigma\right]dv.
		\end{align*}
		Thus, we identify the term in square brackets as $S^\epsilon(|v-v_*|)$. Focusing again on the gain part, we change to spherical coordinates (see~\Cref{fig:spherical}), remembering now that $\cos \theta = k\cdot \sigma \ge 1/\sqrt{2}$, so $\theta \in [0,\pi/4]$ and therefore, we have
		\begin{align*}
			\int_{k\cdot \sigma \ge 1/\sqrt{2}} B^\epsilon\left(\frac{|v-v_*|}{k\cdot \sigma}, 2(k\cdot \sigma)^2 - 1\right) \frac{4}{(k\cdot \sigma)^2} d\sigma &= \iSk \int_0^{\pi/4}\frac{4 \sin\theta}{\cos^2\theta} B^\epsilon\left(
			\frac{|v-v_*|}{\cos\theta}, \cos 2\theta
			\right)d\theta dp \\
			&= 2\pi \int_0^{\pi/4} \frac{2 \sin 2\theta}{\cos^3\theta} B^\epsilon\left(
			\frac{|v-v_*|}{\cos\theta}, \cos 2\theta
			\right)d\theta \\
			&= 2\pi |v-v_*|_{kin}^\gamma \int_0^{\epsilon/2}\cos^{-3}\left(\frac{\theta}{2}
			\right)\beta^\epsilon(\theta)d\theta.
		\end{align*}
		In the last line, we doubled the integration region while also decomposing $B^\epsilon$ with respect to $\beta^\epsilon$ with $\epsilon>0$ sufficiently small so that
		$\theta \in [0,\epsilon/2] \implies \cos_{kin}^\gamma(\theta) = 1$.
		This completes the identification of $S^\epsilon$. 
		
		Turning to the $L^\infty$ bound, we note that the fundamental theorem of calculus gives the estimate
		\begin{align*}
			\cos^{-3}\frac{\theta}{2} - 1 &= \int_0^1 \frac{d}{dt} \cos^{-3}\frac{t\theta}{2} dt  = \frac{3}{2}\theta\int_0^1  \cos^{-4}\left(
			\frac{t\theta}{2}
			\right)\sin \left(
			\frac{t\theta}{2}
			\right)dt \le \frac{3}{2}\theta \sin \left(
			\frac{\theta}{2}
			\right) \sim \frac{3}{4}\theta^2.
		\end{align*}
		Thus, we obtain
		\begin{align*}
			|S^\epsilon(z)| \le \frac{3\pi}{2}\int_0^{\epsilon/2}\theta^2 \beta^\epsilon(\theta) d\theta \le 12.
		\end{align*}
		The final estimate of the lemma is now easy because
		\begin{align*}
			\left|
			\iiRs \iS B^\epsilon f_* (f'-f)d\sigma dv_* dv
			\right| &= \left|
			\iiRs f_* f S^\epsilon(v-v_*)dv dv_*
			\right| \le 12.
		\end{align*}
	\end{proof}
	\begin{lemma}[Truncation lemma]
		\label{lem:truncation}
		One can take the constant
		\[
		C_2 = 150 \pi \left(\iiRs (|v|^2 + |v_*|^2) f_* fdv_* dv\right) \left(\int_{\theta=0}^{\epsilon/2} \theta^2 \beta^\epsilon(\theta)d\theta\right) <+\infty
		\]
		such that for all $R>1$, we have
		\begin{align*}
			\iiRs \iS B^\epsilon f_* (F' - F)^2 d\sigma dv_* dv + C_2\ge \frac{(2\sqrt{2}(R+1))^\gamma}{2} \iiRs \iS b^\epsilon(k\cdot \sigma) f_* \chi_{R*}(F' \chi_R' - F\chi_R)^2 d\sigma dv_* dv,
		\end{align*}
		where we recall the notation $F = \sqrt{f}$ and $\chi_R\in C_c^\infty(\R^3)$ is a smooth indicator function such that
		\[
		0 \le \chi_R \le 1, \quad \chi_R|_{B_R} = 1, \quad \text{supp}\chi_R \subset B_{R+1}.
		\]
	\end{lemma}
	\begin{proof}
		Firstly, it is clear that
		$
		f_* (F' - F)^2 \ge f_* \chi_{R*} (F' - F)^2 \chi_R^2.
		$
		We wish to pair the indicators with $F$ in the right velocity variables so we estimate
		\begin{align*}
			(F'\chi_R' - F\chi_R)^2 &= (F'(\chi_R' - \chi_R) + (F' - F)\chi_R)^2 \le 2F'^2 (\chi_R' - \chi_R)^2 + 2(F' - F)^2 \chi_R^2.
		\end{align*}
		Including $B^\epsilon$, we have
		\begin{align}
			\label{eq:truncmed}
			\begin{split}
				B^\epsilon(|v-v_*|,k \cdot \sigma)  (F' - F)^2
				&\ge |v-v_*|_{kin}^\gamma b^\epsilon(k\cdot \sigma)  \chi_{R*} (F' - F)^2 \chi_R^2 \\
				&\ge |v-v_*|_{kin}^\gamma b^\epsilon(k\cdot \sigma) \left[ 
				\frac{1}{2} \chi_{R*}(F'\chi_R' - F\chi_R)^2
				- 
				\chi_{R*} F'^2(\chi_R' - \chi_R)^2
				\right].
			\end{split}
		\end{align}
		Now, for the second term with the minus sign, we use similar Mean Value estimates as in~\Cref{lem:estbnv2} to deduce
		$
		(\chi_R' - \chi_R)^2 \le \text{Lip}(\chi_R)^2 |v-v_*|^2|\sigma - k|^2.
		$
		We choose $\chi_R$ in such a way that $\text{Lip}(\chi_R) \le 2$ (i.e. its height changes by 1 over a horizontal distance of 0.95). Before proceeding with the estimate of the second term, we write down the following relation which can be obtained as in the proof of the cancellation lemma (\Cref{lem:cancel})
		\begin{equation}
			\label{eq:kpk}
			|v' - v_*| = \frac{1}{2}|\sigma + k| |v-v_*|.
		\end{equation}
		Recalling $dv = \frac{4}{(k'\cdot \sigma)^2}dv'$ and similar pre-post-collision velocity relations from~\Cref{lem:cancel}, the integral of the second term of~\eqref{eq:truncmed} can be estimated by
		\begin{align*}
			&\quad \iiRs \iS |v-v_*|_{kin}^\gamma b^\epsilon(k\cdot \sigma) f_* \chi_{R*} f'(\chi_R' - \chi_R)^2 d\sigma dv_* dv \\
			&\le 16\iiRs \iS |v-v_*|_{kin}^\gamma|v-v_*|^2 b^\epsilon (2(k'\cdot \sigma)^2 - 1) f_* \chi_{R*}f' \frac{|\sigma - k|^2}{(k'\cdot \sigma)^2}dv'dv_* d\sigma.
		\end{align*}
		By expanding the square, one obtains
		\begin{equation}
			\label{eq:sigpmk}
			|\sigma + k|^2 =4(k'\cdot \sigma)^2, \quad |\sigma - k|^2 = 4(1 - (k'\cdot\sigma)^2),
		\end{equation}
		which allows the estimate to continue as
		\begin{align*}
			&\quad \iiRs \iS |v-v_*|_{kin}^\gamma b^\epsilon(k\cdot \sigma) f_* \chi_{R*} f'(\chi_R' - \chi_R)^2 d\sigma dv_* dv \\
			&\le 64 \iiRs \iS \left(
			\frac{|v'-v_*|}{(k'\cdot \sigma)}
			\right)_{kin}^\gamma |v'-v_*|^2 b^\epsilon(2(k'\cdot \sigma)^2 - 1) f_* \chi_{R*} f' \frac{1 - (k'\cdot \sigma)^2}{(k'\cdot \sigma)^4}dv' dv_* d\sigma.
		\end{align*}
		Relabelling $v'$ as $v$ and moving to polar coordinates, we finally have
		\begin{align*}
			&\quad \iiRs \iS |v-v_*|_{kin}^\gamma b^\epsilon(k\cdot \sigma) f_* \chi_{R*} f'(\chi_R' - \chi_R)^2 d\sigma dv_* dv \\
			&\le 64 \iiRs \iSk \int_{\theta = 0}^{\epsilon/2}\left(
			\frac{|v-v_*|}{\cos\theta}
			\right)_{kin}^\gamma |v-v_*|^2 \beta^\epsilon(2\theta) f_* \chi_{R*}f \cos^{-4}\theta (1 - \cos^2 \theta) d\theta dp dv_* dv \\
			&\le 150 \pi \left(\iiRs (|v|^2 + |v_*|^2) f_* fdv_* dv\right) \left(\int_{\theta=0}^{\epsilon/2} \theta^2 \beta^\epsilon(\theta)d\theta\right) =: C_2 < +\infty.
		\end{align*}
		In the last inequality, we bluntly estimated the negative powers of $\cos \theta \sim 1$ since $\theta \le \epsilon/2$.
		
		Turning to the first term of~\eqref{eq:truncmed}, we combine~\eqref{eq:kpk} and~\eqref{eq:sigpmk} together with the identification $k'\cdot \sigma = \cos \theta/2$ (see the proof of~\Cref{lem:cancel}) to deduce for $\theta \in [0,\pi/2]$
		\[
		|v'-v_*| = \cos\frac{\theta}{2}|v-v_*| \implies |v'-v_*| \le |v-v_*| \le \sqrt{2}|v'-v_*|.
		\]
		This implies that whenever $|v_*| \le R+1$ and at least one of $|v|\le R+1$ or $|v'| \le R+1$ hold, we immediately obtain $|v-v_*|^2 \le 8(R+1)^2$. In this case, we can estimate the kinetic contribution
		\[
		|v-v_*|_{kin}^\gamma \ge (2\sqrt{2}(R+1))^\gamma.
		\]
		Adding $C_2$ to both sides of~\eqref{eq:truncmed} and integrating, we obtain
		\begin{align*}
			2\iiRs \iS B^\epsilon f_* (F' - F)^2 + 2 C_2 \ge (2\sqrt{2}(R+1))^\gamma \iiRs \iS b^\epsilon(k\cdot \sigma) f_* \chi_{R*}(F'\chi_R' - F \chi_R)^2.
		\end{align*}
	\end{proof}
	\begin{lemma}[Fourier representation]
		\label{lem:fourier}
		For $f \in L^1(\R^3)$ and $f\ge 0$, we have
		\begin{align*}
			&\quad \iiRs \iS b^\epsilon(k\cdot \sigma) f_*(F' - F)^2 d\sigma dv_* dv \ge  \frac{1}{2(2\pi)^3}\int_{\R^3}|\mathcal{F}[F](\xi)|^2 \left\{
			\iS b^\epsilon\left(
			\frac{\xi}{|\xi|}\cdot \sigma
			\right)(\mathcal{F}[f](0) - |\mathcal{F}[f](\xi^-)|^2
			\right\}d\xi,
		\end{align*}
		with the notation $\xi^- = \frac{\xi - |\xi|\sigma}{2}$ recalling $F = \sqrt{f}$.
		
		Furthermore, there is a constant $C_f$ depending only on bounds for the entropy, mass, and energy of $f$ such that for every $\xi \in \R^3$ we have
		\[
		\mathcal{F}[f](0) - |\mathcal{F}[f](\xi)| \ge C_f \min (|\xi|^2,1).
		\]
	\end{lemma}
	For the first part of the lemma, we direct the reader to~\cite[Section 5, Corollary 3]{ADVW00}. We only show the second estimate of the result to verify that the constant $C_f$ can be taken independently of $\epsilon>0$. 
	\begin{proof}
		Recall that for real numbers $a,b$ one can take $\theta = \tan^{-1}(a/b) \in \R$ such that
		$
		\sqrt{a^2 + b^2} = a\cos \theta - b\sin \theta.
		$
		Applying this to the real and imaginary parts of the Fourier transform of $f$, there is some $\theta \in \R$ such that
		\begin{align*}
			\mathcal{F}[f](0) - |\mathcal{F}[f](\xi)| = \int_{\R^3} f(v) (1 - \cos (v\cdot \xi + \theta)) dv    
			= 2\int_{\R^3} f(v) \sin^2\left(
			\frac{v\cdot \xi + \theta}{2}
			\right)dv   \ge 2\sin^2 \delta \int_{B_r\cap A_\delta}\!\!\!\!\!f(v)dv.
		\end{align*}
		Here, $r>0$ is some (large) radius to be specified later. For $\delta>0$, we consider the set $
		A_\delta := \{
		v\in \R^3 \, : \, \forall p\in \mathbb{Z}, \, |v\cdot \xi + \theta - 2\pi p | \ge 2\delta
		\}$.
		The partition $\R^d = (\R^d \setminus B_r) \cup (B_r \cap A_\delta) \cup (B_r \cap (\R^d \setminus A_\delta))$ leads to the estimate
		\begin{align*}
			\sin^2 \delta \int_{B_r\cap A_\delta} f(v) dv &= \sin^2\delta \left(
			\int_{\R^d} - \int_{\R^d \cap B_r} - \int_{B_r \cap (\R^d \setminus A_\delta)}
			\right)f(v) dv \\
			&\ge \sin^2\delta \left(
			\|f\|_{L^1} - \frac{\|f\|_{L_2^1}}{r^2} - \int_{B_r \cap (\R^d \setminus A_\delta)} f(v) dv
			\right).
		\end{align*}
		We now further investigate the set
		$
		B_r \cap (\R^d \setminus A_\delta) = \{
		v\in\R^3 \, : \, |v| \le r, \, \exists p \in \mathbb{Z} \text{ s.t. } |v\cdot \xi + \theta - 2\pi p| \le 2\delta
		\}$.
		By considering (rotate and translate $v$ as appropriate) $\xi = k e_1, \,  \theta = 0$, with $k>0$ and $e_1 = (1,0,0)$, one can show
		\[
		|B_r \cap (\R^d \setminus A_\delta)| \le \frac{4\delta}{|\xi|}(2r)^{d-1}\left(
		3 + \frac{r|\xi|}{\pi}
		\right).
		\]
		More precisely, one should think of $B_r \cap (\R^d \setminus A_{\delta = 0})$ as the set of integer lattice points in $B_r$ lying along the axial direction of $\xi/|\xi|$. So the inequality above estimates the measure of a $\delta$ neighbourhood version of this set. Continuing the estimate, we thus obtain
		\begin{equation}
			\label{eq:fouriermed}
			\mathcal{F}[f](0) - |\mathcal{F}[f](\xi)| \ge 2\sin^2 \delta \left(
			\|f\|_{L^1} - \frac{\|f\|_{L_2^1}}{r^2} - \sup_{|A| \le \frac{4\delta}{|\xi|}(2r)^{d-1}\left(
				3 + \frac{r|\xi|}{\pi}
				\right)}\int_A f(v) dv
			\right).
		\end{equation}
		In the case $|\xi|\ge 1$, notice that
		\[
		\frac{4\delta}{|\xi|}(2r)^{d-1}\left(
		3 + \frac{r|\xi|}{\pi}
		\right) \le 12\delta (2r)^{d-1} + \frac{6 \delta}{\pi}(2r)^d.
		\]
		Therefore, choose large $r>0$ and small $\delta>0$ such that the bracketed quantity is strictly positive (appealing to equi-integrability of $f$).
		In the case $|\xi| \le 1$, one again chooses large $r>0$ but small $\delta \sim |\xi|$ so that
		$
		\sin^2\delta \ge |\xi|^2.
		$
	\end{proof}
	\begin{lemma}[Fourier average estimate]
		\label{lem:avgfourier}
		For every $\xi \in \R^3$ and $\epsilon \le 1$ we have
		\[
		\iS  b^\epsilon\left(
		\frac{\xi}{|\xi|}\cdot \sigma
		\right)\min(|\xi^-|^2,1)d\sigma  \ge \frac{2c_1}{\pi}\left(\int_0^{\pi/2} \phi^{1-\nu}d\phi\right) \min(|\xi|^2,|\xi|^\nu),
		\]
		recalling the notations $c_1,\nu$ from~\ref{ass:beta} and $\xi^- = \frac{\xi - |\xi|\sigma}{2}$.
	\end{lemma}
	\begin{proof}
		From the definition of $\xi^-$, we have
		\[
		|\xi^-|^2 = \frac{|\xi|^2}{2}\left(
		1 - \frac{\xi}{|\xi|}\cdot \sigma
		\right).
		\]
		Using spherical coordinates with radial direction given by $\xi/|\xi|$ (see~\Cref{sec:spherical}), we use the lower bound of~\eqref{eq:boundcos} and directly integrate over $\mathbb{S}_{\xi^\perp}^1$ to obtain
		\begin{align*}
			\iS b^\epsilon\left(
			\frac{\xi}{|\xi|}\cdot \sigma
			\right)\min(|\xi^-|^2,1)d\sigma 
			&= \int_{\mathbb{S}_{\xi^\perp}^1}\int_{\theta=0}^{\epsilon/2}\beta^\epsilon(\theta) \min\left(
			\frac{|\xi|^2}{2}(1-\cos\theta), 1
			\right) d\theta d\xi^\perp \\
			&\ge \frac{4}{\pi} \int_{\theta=0}^{\epsilon/2}\beta^\epsilon(\theta) \min \left(
			\frac{|\xi|^2\theta^2}{2},1
			\right)d\theta.
		\end{align*}
		We introduce the change of variables $\theta = \epsilon\chi/\pi$ and the lower bound of~\ref{ass:beta} giving
		\begin{align*}
			\iS b^\epsilon\left(
			\frac{\xi}{|\xi|}\cdot \sigma
			\right)\min(|\xi^-|^2,1)d\sigma &\ge \frac{4}{\pi} \int_0^{\pi/2}\beta(\chi) \min \left(
			\frac{|\xi|^2\chi^2}{2},\frac{\pi^2}{\epsilon^2}
			\right)d\chi \ge \frac{4c_1}{\pi}\int_0^{\pi/2}\min \left(
			\frac{|\xi|^2\chi^2}{2},\frac{\pi^2}{\epsilon^2}
			\right) \frac{1}{\chi^{1+\nu}}d\chi.
		\end{align*}
		We use one more change of variable $\phi = |\xi|\chi$. In the case $|\xi| \ge 1$ we can further minorize by
		\[
		\frac{4c_1}{\pi}\left(\int_0^{\pi/2} \min \left(
		\frac{\phi^2}{2}, \frac{\pi^2}{\epsilon^2}
		\right)\frac{1}{\phi^{1+\nu}}d\phi\right) |\xi|^\nu .
		\]
		In the case $|\xi| \le 1$, we explicitly obtain
		\begin{align*}
			\frac{2c_1}{\pi} \left(
			\int_0^{|\xi|\pi/2}\phi^{1-\nu}d\phi
			\right) |\xi|^\nu = C |\xi|^2.
		\end{align*}
	\end{proof}
	
\end{appendices}
\section*{Acknowledgements}
JAC was supported by the Advanced Grant Nonlocal-CPD (Nonlocal PDEs for Complex Particle Dynamics: Phase Transitions, Patterns and Synchronization) of the European Research Council Executive Agency (ERC) under the European Union's Horizon 2020 research and innovation programme (grant agreement No. 883363). MGD was partially supported by CNPq-Brazil(\#308800/2019-2) and Instituto Serrapilheira. JW was funded by the University of Oxford Mathematical Institute Award Scholarship.

\bibliographystyle{abbrv}
\bibliography{refs}

\end{document}